\documentclass[11pt]{amsart}
\usepackage{amssymb}
\usepackage{amsmath}
\usepackage{amsthm}
\usepackage{braket}
\usepackage[dvipsnames]{xcolor}
\usepackage{hyperref}
\usepackage{pdfpages}
\usepackage{graphicx}
\usepackage{caption}
\usepackage{subcaption}
\usepackage{mathrsfs} 
\usepackage{tikz-cd} 
\usepackage{bm}
\usepackage{enumitem}
\usepackage{mathtools}
\usepackage{pgfplots}
\usepackage{import}
\usepackage{xifthen}
\usepackage{transparent}
\usepackage{adjustbox}
\usepackage{booktabs}
\usepackage{soul}
\usepackage{cancel}
\usepackage{scrextend}
\usepackage{array}
\usepackage{hhline}
\usepackage{footnote}
\usepackage{tablefootnote}

\newcolumntype{L}[1]{>{\raggedright\arraybackslash}p{#1}}
\newcolumntype{C}[1]{>{\centering\arraybackslash}p{#1}}
\newcolumntype{R}[1]{>{\raggedleft\arraybackslash}p{#1}}

\numberwithin{equation}{section}
\theoremstyle{plain}
\newtheorem{theorem}{Theorem}[section]

\theoremstyle{theorem}
\newtheorem{prop}[theorem]{Proposition}
\newtheorem{lemma}[theorem]{Lemma}
\newtheorem{cor}[theorem]{Corollary}

\newtheorem*{question*}{Question}

\theoremstyle{definition}
\newtheorem{definition}[theorem]{Definition}

\newtheorem*{remark}{Remark}

\newcommand{\R}{\mathbb{R}}
\newcommand{\C}{\mathbb{C}}

\newcommand{\Z}{\mathbb{Z}}
\newcommand{\Hyp}{\mathbb{H}}

\newcommand{\D}{\mathbb{D}}

\newcommand*\circled[1]{\tikz[baseline=(char.base)]{
            \node[shape=circle,draw,inner sep=0.2pt] (char) {#1};}}

\DeclareMathOperator{\Int}{Int}

\newcommand{\mate}{\bot \!\! \! {\bot}} 

\numberwithin{figure}{section}

\begin{document}

\title[Circle Packings, Reflection Groups and Anti-rational Maps]{Circle packings, kissing reflection groups and critically fixed anti-rational maps}

\begin{author}[R.~Lodge]{Russell Lodge}
\address{Department of Mathematics and Computer Science, Indiana State University, Terre Haute, IN 47809, USA}
\email{russell.lodge@indstate.edu}
\end{author}

\begin{author}[Y.~Luo]{Yusheng Luo}
\address{Institute for Mathematical Sciences, Stony Brook University, 100 Nicolls Rd, Stony Brook, NY 11794-3660, USA}
\email{yusheng.s.luo@gmail.com}
\end{author}

\begin{author}[S.~Mukherjee]{Sabyasachi Mukherjee}
\address{School of Mathematics, Tata Institute of Fundamental Research, 1 Homi Bhabha Road, Mumbai 400005, India}
\email{sabya@math.tifr.res.in}
\end{author}

\begin{abstract}
In this paper, we establish an explicit correspondence between kissing reflection groups and critically fixed anti-rational maps.
The correspondence, which is expressed using simple planar graphs, has several dynamical consequences.
As an application of this correspondence, we give complete answers to geometric mating problems for critically fixed anti-rational maps.
\end{abstract}

\maketitle

\setcounter{tocdepth}{1}
\tableofcontents

\section{Introduction}
Ever since Sullivan's translation of Ahlfors' finiteness theorem into a solution of a long standing open problem on wandering domains in the 1980s \cite{Sullivan85}, many more connections between the theory of Kleinian groups and the study of dynamics of rational functions on $\widehat\C$ have been discovered.
These analogies between the two branches of conformal dynamics, which are commonly known as the {\em Sullivan's dictionary}, not only provide a conceptual framework for understanding the connections, but motivates research in each field as well.

In this paper, we extend this dictionary by establishing a strikingly explicit correspondence between \begin{itemize}
\item {\em Kissing reflection groups}: groups generated by reflections along the circles of finite circle packings $\mathcal{P}$ (see \S \ref{sec:cc}), and
\item {\em Critically fixed anti-rational maps}: proper anti-holomorphic self-maps of $\widehat\C$ with all critical points fixed (see \S \ref{sec:cf}).
\end{itemize}
This correspondence can be expressed by a combinatorial model: a simple \emph{plane graph} $\Gamma$. Throughout this paper, a plane graph will mean a planar graph together with an embedding in $\widehat\C$.
We say that two plane graphs are \emph{isomorphic} if the underlying graph isomorphism is induced by an orientation preserving homeomorphism of $\widehat\C$.
A graph $\Gamma$ is said to be {\em $k$-connected} if $\Gamma$ contains more than $k$ vertices and remains connected if any $k-1$ vertices and their corresponding incident edges are removed.
Our first result shows that
\begin{theorem}\label{thm:c}
The following three sets are in natural bijective correspondence:
\begin{itemize}
\item $\{2$-connected, simple, plane graphs $\Gamma$ with $d+1$ vertices up to isomorphism of plane graphs$\}$,\\

\item $\{$Kissing reflection groups $G$ of rank $d+1$ with connected limit set up to QC conjugacy$\}$,\\

\item $\{$Critically fixed anti-rational maps $\mathcal{R}$ of degree $d$ up to M{\"o}bius conjugacy$\}$.
\end{itemize}
Moreover, if $G_\Gamma$ and $\mathcal{R}_\Gamma$ correspond to the same plane graph $\Gamma$, then the limit set $\Lambda(G_\Gamma)$ is homeomorphic to the Julia set $\mathcal{J}(\mathcal{R}_\Gamma)$ via a dynamically natural map.
\end{theorem}

The correspondence between graphs and kissing reflection groups comes from the well-known circle packing theorem (see Theorem \ref{thm:cpt}):
given a kissing reflection group $G$ with corresponding circle packing $\mathcal{P}$, the plane graph $\Gamma$ associated to $G$ in Theorem \ref{thm:c} is the {\em contact graph} of the circle packing $\mathcal{P}$.
The $2$-connectedness condition for $\Gamma$ is equivalent to the connectedness condition for the limit set of $G$ (see Figure \ref{fig:dtwe}).

\begin{figure}[h!]
\begin{subfigure}{.40\textwidth}
  \centering
  \includegraphics[width=.95\linewidth]{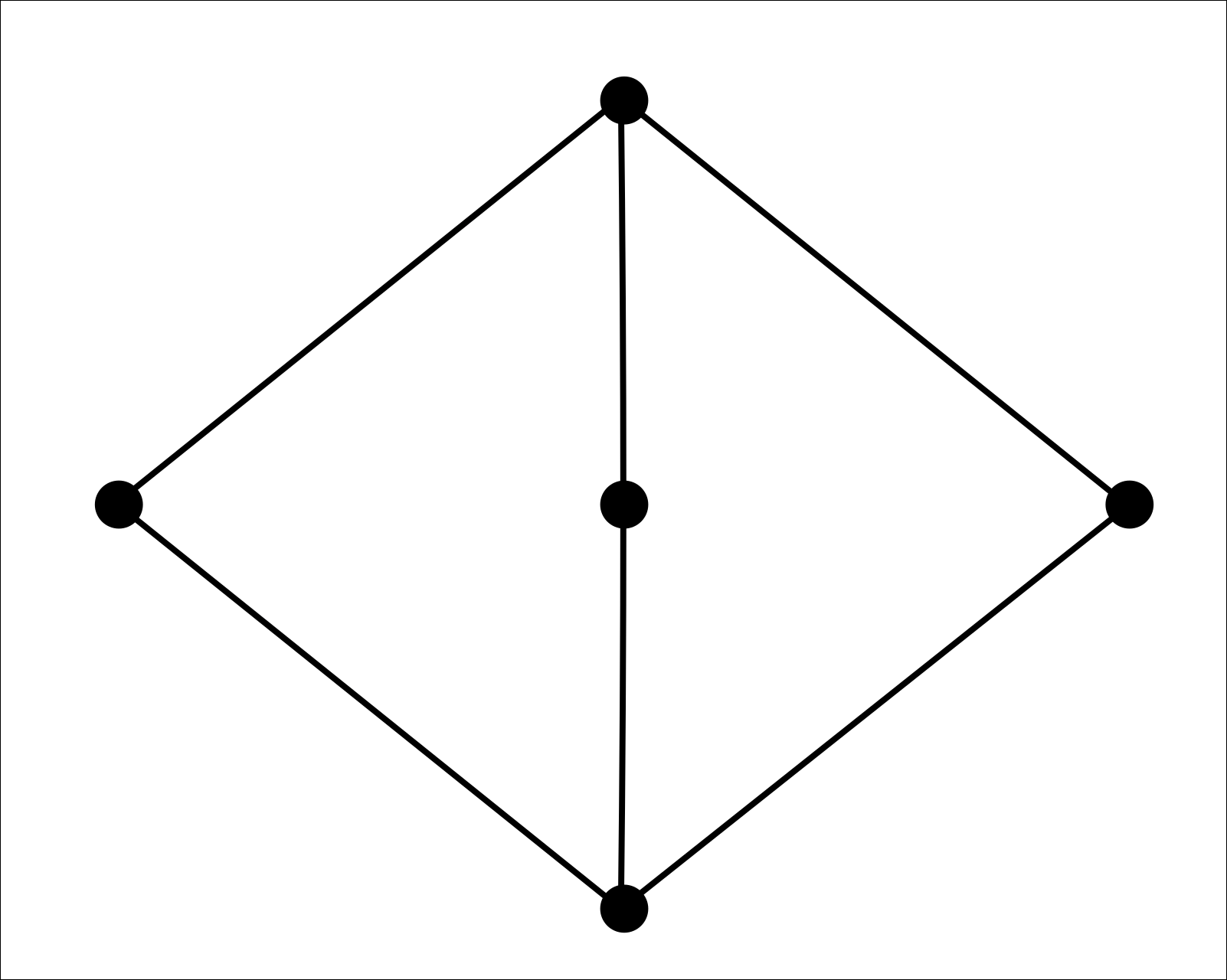}  
  \caption{A $2$-connected simple plane graph $\Gamma$ with five vertices.}
\end{subfigure}
\begin{subfigure}{.55\textwidth}
  \centering
   \includegraphics[width=.95\linewidth]{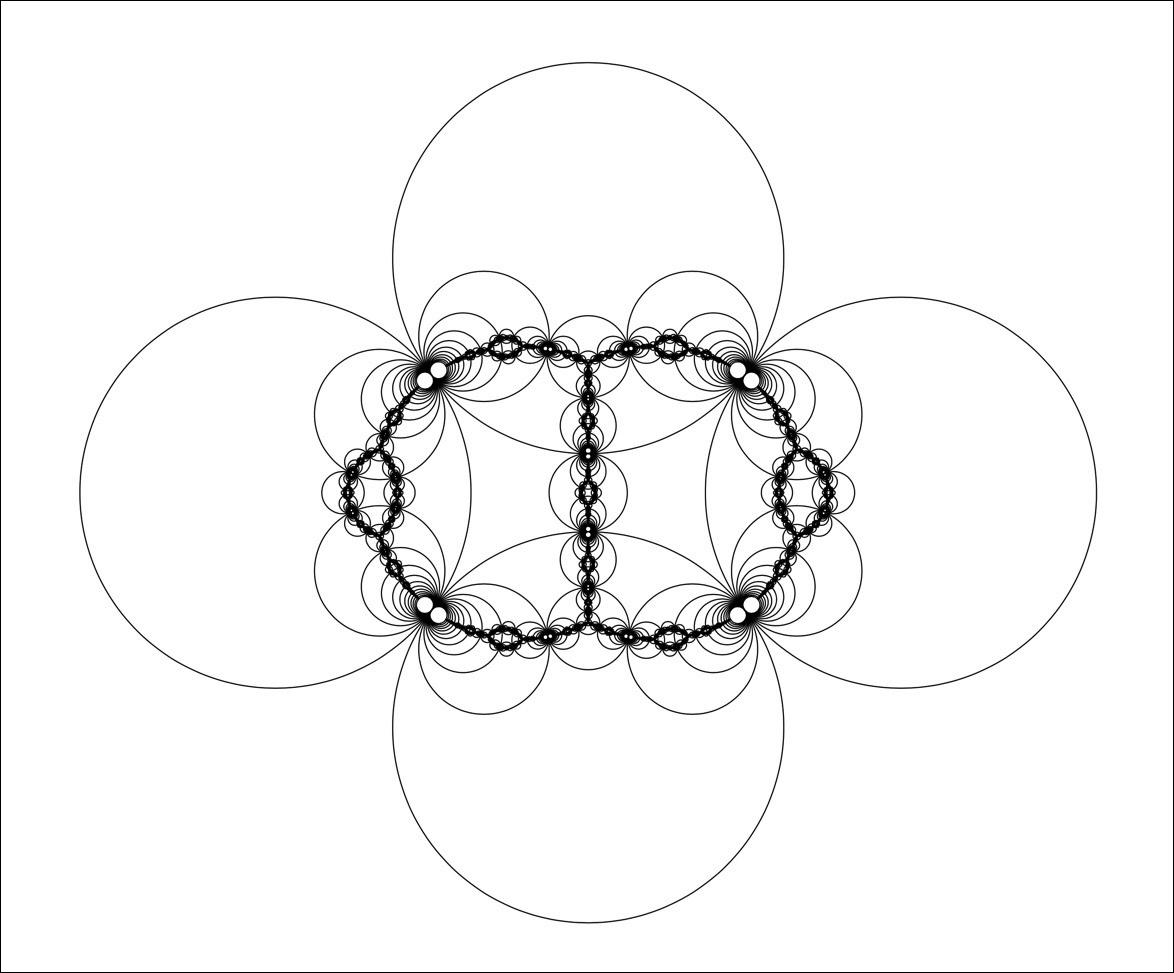}  
  \caption{The limit set of a kissing reflection group $G_\Gamma$. The four outermost circles together with the central circle form a finite circle packing $\mathcal{P}$ whose contact graph is $\Gamma$. The kissing reflection group $G_\Gamma$ is generated by reflections along the five circles in $\mathcal{P}$.}
\end{subfigure}
\begin{subfigure}{.95\textwidth}
  \centering
    \includegraphics[width=.95\linewidth]{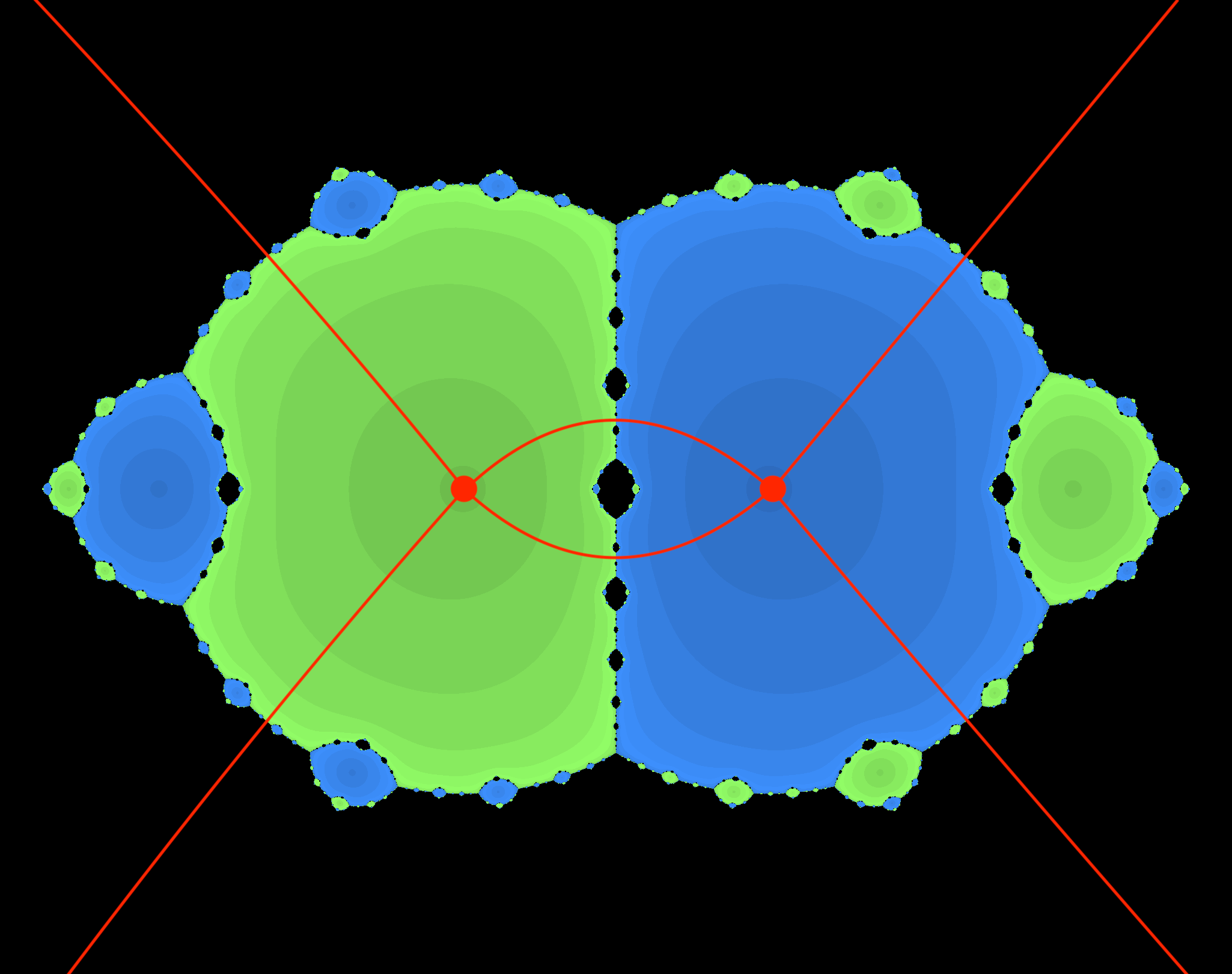}  
  \caption{The Julia set of the degree $4$ critically fixed anti-rational map $\mathcal{R}_\Gamma$. The Tischler graph $\mathscr{T}$, which is drawn schematically in red with two finite vertices and one vertex at infinity, is the planar dual of $\Gamma$.}
  \label{fig:mr}
\end{subfigure}
\caption{An example of the correspondence.}
\label{fig:dtwe}
\end{figure}

On the other hand, given a critically fixed anti-rational map $\mathcal{R}$, we consider the union $\mathscr{T}$ of all {\em fixed internal rays} in the invariant Fatou components (each of which necessarily contains a fixed critical point of $\mathcal{R}$), known as the {\em Tischler graph} (cf. \cite{Tis89}). Roughly speaking, Tischler graphs are to critically fixed (anti-)rational maps what Hubbard trees are to postcritically finite (anti-)polynomials: both are forward invariant graphs containing the postcritical points.
We show that the planar dual of $\mathscr{T}$ is a $2$-connected, simple, plane graph.
The plane graph $\Gamma$ we associate to $\mathcal{R}$ in Theorem \ref{thm:c} is the planar dual $\mathscr{T}^{\vee}$ of $\mathscr{T}$ (see Figure \ref{fig:dtwe}).

In the special case when $\Gamma$ is the planar dual of a triangulation of $\widehat{\C}$, the existence of the critically fixed anti-rational map $\mathcal{R}_\Gamma$ was proved in \cite{LLMM19}.

We remark that the correspondence between $G_\Gamma$ and $\mathcal{R}_\Gamma$ through the plane graph $\Gamma$ is dynamically natural. Indeed, we associate, following \cite[\S 4]{LLMM19}, a map $\mathcal{N}_\Gamma$ to the group $G_\Gamma$ with the properties that $\mathcal{N}_\Gamma$ and $G_\Gamma$ have the same grand orbits (cf. \cite{Nielsen,BS}), and the homeomorphism between $\Lambda(G_\Gamma)$ and $\mathcal{J}(\mathcal{R}_\Gamma)$ conjugates $\mathcal{N}_\Gamma$ to $\mathcal{R}_\Gamma$. See \S \ref{dyn_corr_subsec} for more details.

The asymmetry between the QC conjugacies on the group side and M\"obius conjugacies on the anti-rational map side is artificial.
Since the dynamical correspondence is between the limit set and the Julia set,
Theorem~\ref{thm:c} yields a bijection between quasiconformal conjugacy classes of kissing reflection groups (of rank $d+1$) with connected limit set and hyperbolic components having critically fixed anti-rational maps (of degree $d$) as centers.

\subsection*{The geometric mating problems.}
In complex dynamics, {\em polynomial mating} is an operation first introduced by Douady in \cite{Douady83} that takes two suitable polynomials $P_1$ and $P_2$, and constructs a richer dynamical system by carefully pasting together the boundaries of their filled Julia sets so as to obtain a copy of the Riemann sphere, together with a rational map $R$ from this sphere to itself (see Definition \ref{defn:gm} for the precise formulation).
It is natural and important to understand which pairs of polynomials can be mated, and which rational maps are matings of two polynomials.
The analogous question in the Kleinian group setting can be formulated in terms of Cannon-Thurston maps for degenerations in the Quasi-Fuchsian space (see \cite{Hubbard12} and \S \ref{sec:krg} for related discussions).
The correspondence in Theorem~\ref{thm:c} allows us to explicitly characterize kissing reflection groups and critically fixed anti-rational maps that arise as matings.

We say that a simple plane graph $\Gamma$ with $n$ vertices is {\em outerplanar} if it has a face with all $n$ vertices on its boundary.
It is said to be {\em Hamiltonian} if there exists a Hamiltonian cycle, i.e., a closed path visiting every vertex exactly once. 

\begin{theorem}\label{thm:gm}
Let $\Gamma$ be a $2$-connected, simple, plane graph. Let $G_\Gamma$ and $\mathcal{R}_\Gamma$ be a kissing reflection group and a critically fixed anti-rational map associated with $\Gamma$.
Then the following hold true.
\begin{itemize}
\item $\Gamma$ is outerplanar $\Leftrightarrow$ $\mathcal{R}_\Gamma$ is a critically fixed anti-polynomial $\Leftrightarrow$ $G_\Gamma$ is a function group.
\item $\Gamma$ is Hamiltonian $\Leftrightarrow$ $\mathcal{R}_\Gamma$ is a mating of two polynomials $\Leftrightarrow$ $G_\Gamma$ is a mating of two function groups $\Leftrightarrow$ $G_\Gamma$ is in the closure of the quasiconformal deformation space of the regular ideal polygon reflection group.
\end{itemize}
\end{theorem}

It is known that a rational map may arise as the geometric mating of more than one pair of polynomials (in other words, the \emph{decomposition/unmating} of a rational map into a pair of polynomials is not necessarily unique). This phenomenon was first observed in \cite{Wittner88}, and is referred to as {\em shared matings} (see \cite{Rees10}).
In our setting, we actually prove that each Hamiltonian cycle of $\Gamma$ gives an unmating of $\mathcal{R}_\Gamma$ into two anti-polynomials.
Thus, we get many examples of shared matings coming from different Hamiltonian cycles in the associated graphs.

We now address the converse question of mateability in terms of laminations. 
Let us first note that the question of mateability for kissing reflection groups can be answered using Thurston's double limit theorem and the hyperbolization theorem. 
In the reflection group setting, possible degenerations in the Quasi-Fuchsian space (of the regular ideal polygon reflection group) are listed by a pair of geodesic laminations on the two conformal boundaries which are invariant under some orientation reversing involution $\sigma$ (see \S \ref{sec:ads}). 
All these $\sigma$-invariant laminations turn out to be multicurves on the associated conformal boundaries.
A pair of simple closed curves is said to be {\em parallel} if they are isotopic under the natural orientation reversing identification of the two conformal boundary components.
A pair of laminations is said to be {\em non-parallel} if no two components are parallel.
If we lift a multicurve to the universal cover, we get two invariant laminations on the circle.
Then they are are non-parallel if and only if the two laminations share no common leaf under the natural identification of the two copies of the circle.
Thurston's hyperbolization theorem asserts that in our setting, the degeneration along a pair of laminations exists if and only if this pair is non-parallel.

For a marked anti-polynomial, we can also associate a lamination via the B\"otthcher coordinate at infinity.
Similar as before, we say a pair of (anti-polynomial) laminations are {\em non-parallel} if they share no common leaf under the natural identification of the two copies of the circle.
When we glue two filled Julia sets using the corresponding laminations, the resulting topological space may not be a $2$-sphere.
We call this a {\em Moore obstruction}.
We prove the following more general mateability theorem for post-critically finite anti-polynomials, which in particular answers the question of mateability of two critically fixed anti-polynomials.
\begin{theorem}\label{thm:mig}
Let $P_1$ and $P_2$ be two marked anti-polynomials of equal degree $d\geq 2$, where $P_1$ is critically fixed and $P_2$ is postcritically finite, hyperbolic.
Then there is an anti-rational map $R$ that is the geometric mating of $P_1$ and $P_2$ if and only if there is no Moore obstruction.

Consequently, if both $P_1$ and $P_2$ are critically fixed, then they are geometrically mateable if and only if the corresponding laminations are non-parallel.
\end{theorem}

\subsection*{Acylindrical manifolds and gasket sets.}
A circle packing is a connected collection of (oriented) circles in $\widehat{\C}$ with disjoint interiors.
We say that a closed set $\Lambda$ is a {\em round gasket} if
\begin{itemize}
\item $\Lambda$ is the closure of some infinite circle packing; and
\item the complement of $\Lambda$ is a union of round disks which is dense in $\widehat{\C}$.
\end{itemize}
We will call a homeomorphic copy of a round gasket a {\em gasket}.

Many examples of kissing reflection groups and critically fixed anti-rational maps have gasket limit sets and Julia sets (see Figure \ref{fig:ps}).
The correspondence allows us to classify all these examples (cf. \cite[Theorem~28]{KN19} and \cite{KMS93}).

\begin{theorem}\label{thm:glj}
Let $\Gamma$ be a $2$-connected, simple, plane graph. Then,

\noindent $G_\Gamma$ has gasket limit set $\iff$ $\mathcal{R}_\Gamma$ has gasket Julia set $\iff$ $\Gamma$ is $3$-connected.
\end{theorem}

The $3$-connectedness for the graph $\Gamma$ also has a characterization purely in terms of the inherent structure of the hyperbolic $3$-manifold associated with $G_\Gamma$.
Given a kissing reflection group $G_\Gamma$, the index $2$ subgroup $\widetilde{G}_\Gamma$ consisting of orientation preserving elements is a Kleinian group.
We say that $G_\Gamma$ is {\em acylindrical} if the hyperbolic $3$-manifold for $\widetilde{G}_\Gamma$ is acylindrical (see \S \ref{subsec:akrg} for the precise definitions).
We show that
\begin{theorem}\label{thm:a}
Let $\Gamma$ be a $2$-connected, simple, plane graph. Then,

\noindent $G_\Gamma$ is acylindrical $\iff$ $\Gamma$ is $3$-connected.
\end{theorem}

\begin{figure}[h!]
\begin{subfigure}{.3\textwidth}
  \centering
  \includegraphics[width=.95\linewidth]{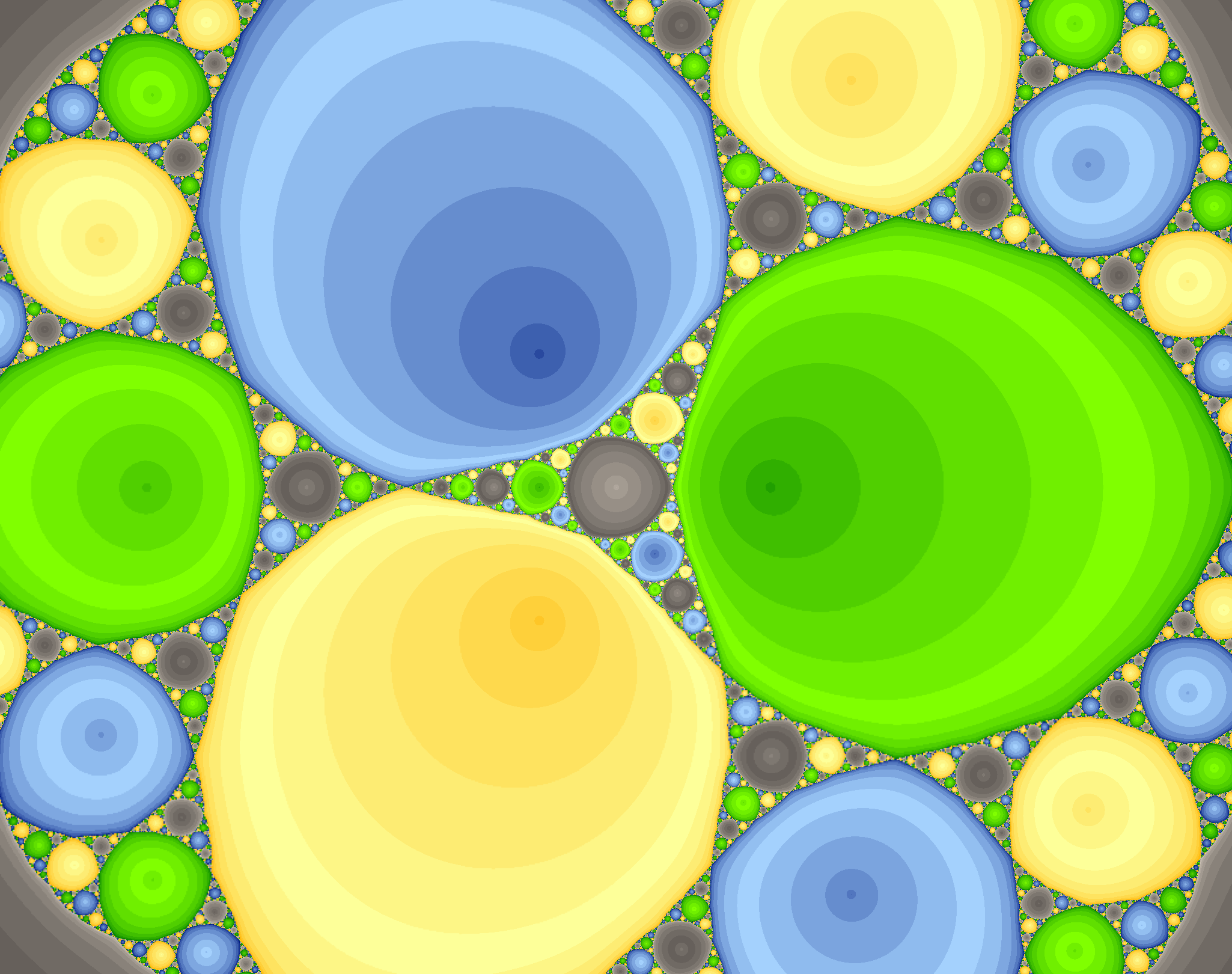}  
  \caption{Tetrahedron: $\frac{3\overline{z}^2}{2\overline{z}^3+1}$}
\end{subfigure}
\begin{subfigure}{.3\textwidth}
  \centering
  \includegraphics[width=.95\linewidth]{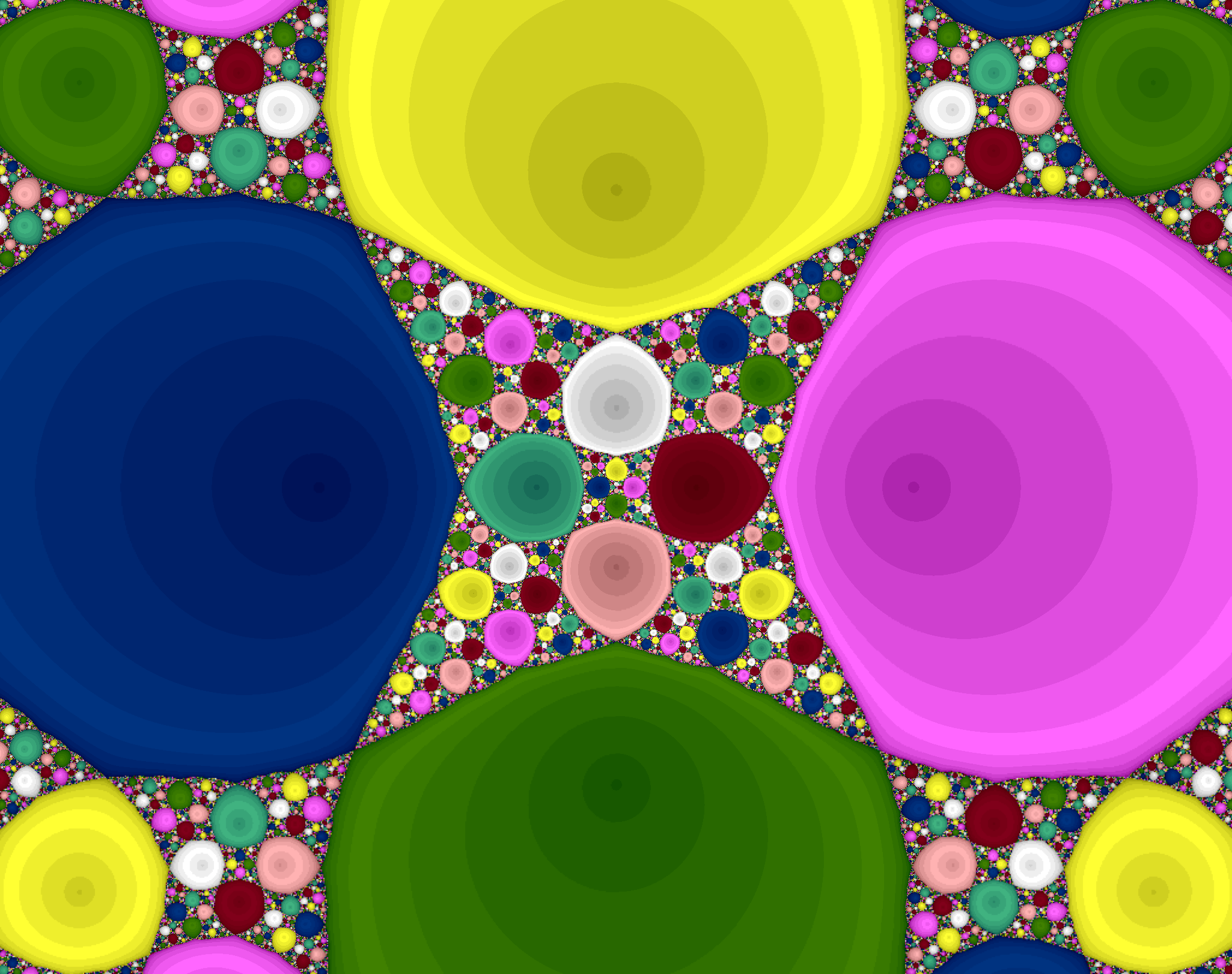}  
  \caption{Octahedron: $\frac{5\overline{z}^4+1}{\overline{z}^5+5\overline{z}}$}
\end{subfigure}
\begin{subfigure}{.3\textwidth}
  \centering
  \includegraphics[width=.95\linewidth]{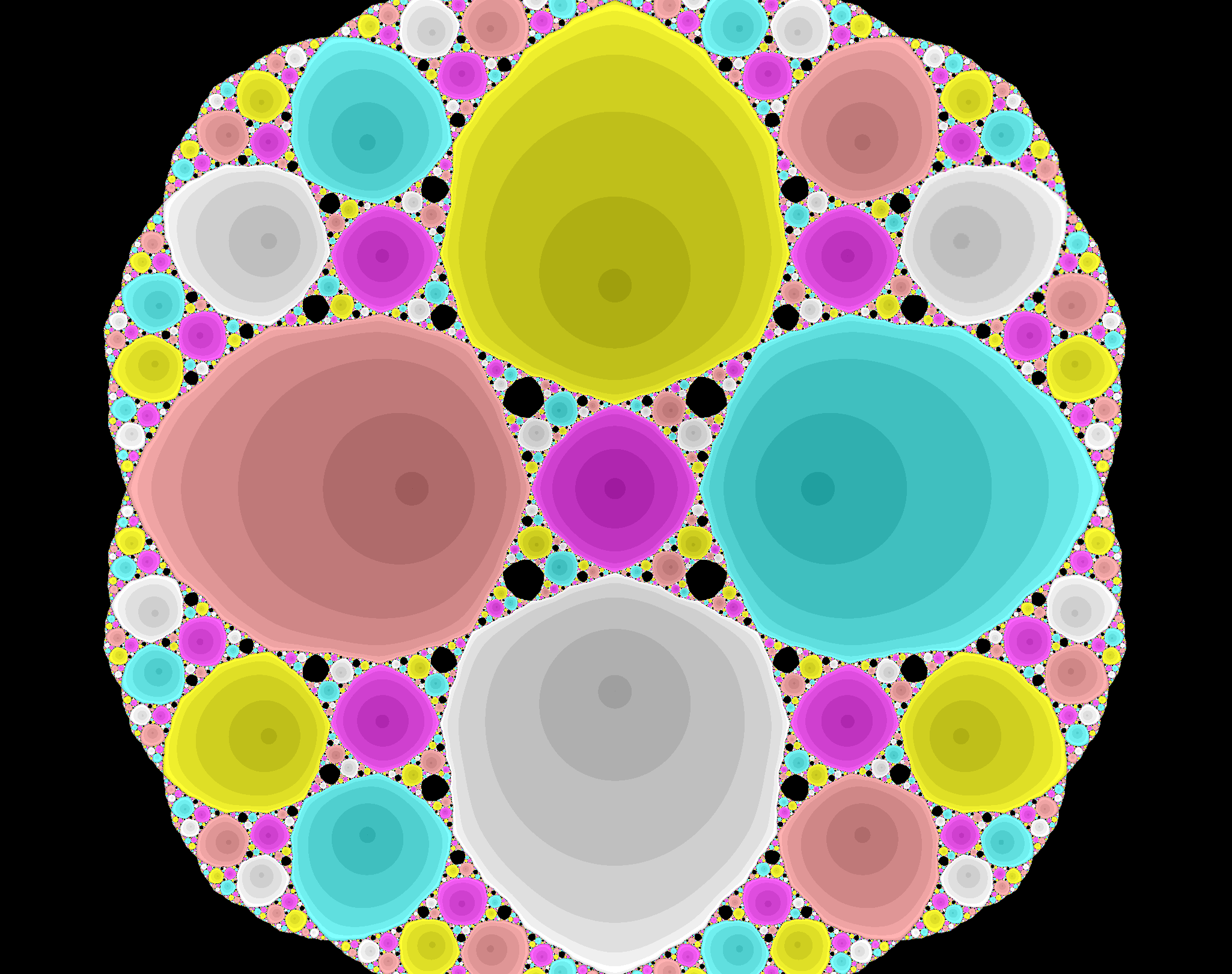}  
  \caption{Cube: $\frac{\overline{z}^7+7\overline{z}^3}{7\overline{z}^4+1}$}
\end{subfigure}

\begin{subfigure}{.46\textwidth}
 \centering
  \includegraphics[width=.95\linewidth]{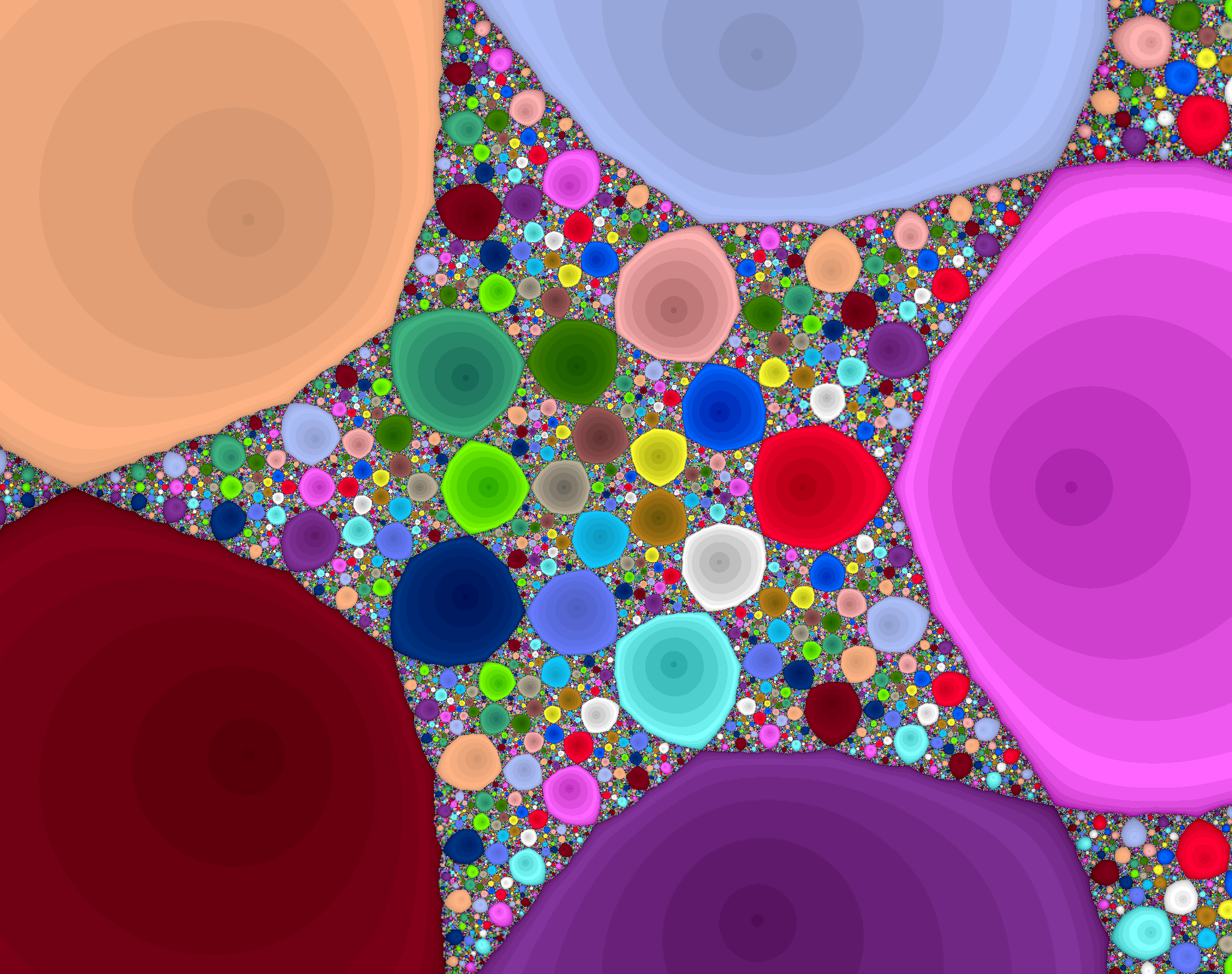}  
  \caption{\begin{tiny}Icosahedron:\end{tiny} \\$\frac{11\overline{z}^{10}+66\overline{z}^5-1}{\overline{z}^{11}+66\overline{z}^6-11\overline{z}}$}
\end{subfigure}
\begin{subfigure}{.46\textwidth}
  \centering
  \includegraphics[width=.95\linewidth]{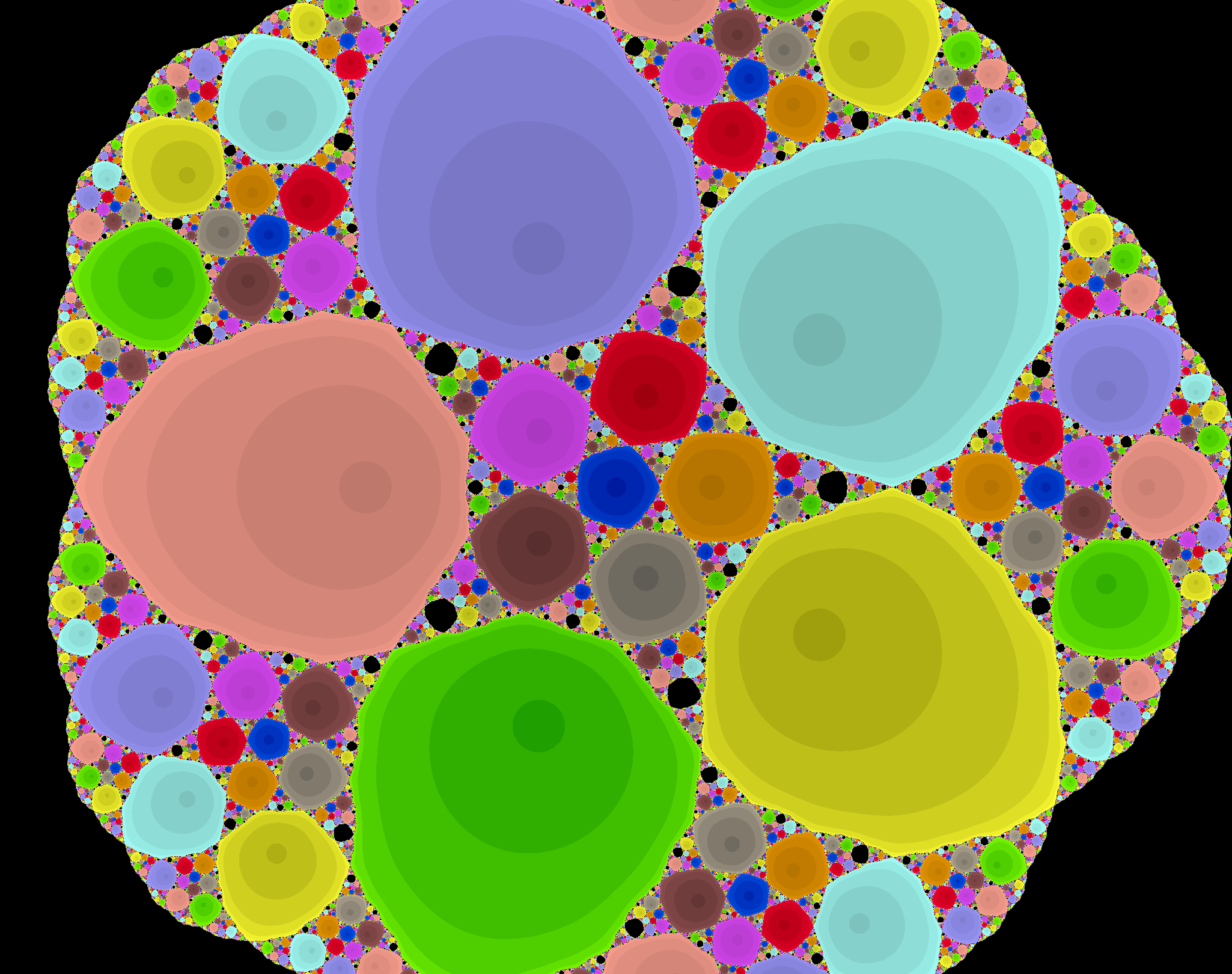} 
 \caption{ \begin{tiny}Dodecahedron:\end{tiny} \\$\frac{\overline{z}^{19}-171\overline{z}^{14}+247\overline{z}^9+57\overline{z}^4}{-57\overline{z}^{15}+247\overline{z}^{10}+171\overline{z}^5+1}$}
\end{subfigure}
\caption{Critically fixed anti-rational maps associated to Platonic solids. The Fatou components are colored according to their grand orbit. The Tischler graph $\mathscr{T}$, which is the planar dual of $\Gamma$, is visible in the Figures by connecting the centers of critical fixed Fatou components.}
\label{fig:ps}
\end{figure}

\subsection*{Parameter space implications of the dictionary}

We now briefly mention some natural questions regarding parameter spaces of anti-rational maps raised by the aforementioned dictionary between kissing reflection groups and critically fixed anti-rational maps.

Acylindrical manifolds play an important role in three dimension geometry and topology.
Relevant to our discussion, Thurston proved that the deformation space of an acylindrical $3$-manifold is bounded \cite{Thurston86} (this is famously known as Thurston's compactness theorem). The analogue of deformation spaces in holomorphic dynamics is {\em hyperbolic components}.
The analogy established in Theorems~\ref{thm:glj} and~\ref{thm:a} leads one to ask whether there is a counterpart of Thurston's compactness theorem for hyperbolic components of critically fixed anti-rational maps with gasket Julia sets (see \cite[Question~5.3]{McM95} for the corresponding question in the convex cocompact setting). 
In a forthcoming paper, a suitably interpreted boundedness result will be proved for the hyperbolic components in consideration \cite{LLM22}.

In \cite{HT}, Hatcher and Thurston studied the topology of the moduli space of marked circle packings with $n$ circles in $\C$. In particular, they showed that the map that sends a marked circle packing to the centers of the circles yields a homotopy equivalence between the moduli space of marked circle packings and the configuration space of $n$ marked points in $\C$. Motivated by this statement, one may ask whether the union of the closures of the hyperbolic components of degree $d$ critically fixed anti-rational maps have a non-trivial topology. 
This question will be answered affirmatively in a future work, drawing exact parallels between contact structures of quasiconformal deformation spaces of kissing reflection groups and the corresponding hyperbolic component closures \cite{LLM22}.

We now summarize these correspondences in the following table.\footnote{Here we have also included the parameter space implications, which will be studied in a subsequent paper.}

\begin{center}
\begin{tabular}{|C{0.24\textwidth} | C{0.33\textwidth}| C{0.33\textwidth}|}
     \hline & & \\
\textbf{Simple planar graph} & \textbf{Kleinian groups} & \textbf{Complex dynamics}  \\ & & \\ \hhline{|===|}
 Connected            &  Kissing reflection group             &   -    \\ \hline 
\vspace{0.05mm} $2$-connected                  & \vspace{0.05mm} Connected limit set                                                      & Critically fixed anti-rational map          \\ \hline
\vspace{0.1mm} $3$-connected/ Polyhedral            & Gasket limit set/ Acylindrical/  Bounded QC deformation space & Gasket Julia set/ Bounded pared deformation space                            \\ \hline
\vspace{0.1mm} Outerplanar                  & Function kissing reflection group                                                      & Critically fixed anti-polynomial          \\ \hline
\vspace{0.1mm} Hamiltonian                  & Closure of Quasi-Fuchsian space                                          & Mating of two anti-polynomials             
\\ \hline
\vspace{0.1mm} A marked Hamiltonian cycle                  & \vspace{0.1mm} A pair of non-parallel multicurves                                                      & A pair of non-parallel anti-polynomial laminations
\\ \hline
\vspace{0.1mm} A graph $\Gamma_1 < \Gamma_2$                 & The QC deformation space $\mathcal{QC}(\Gamma_2) \subseteq \partial \mathcal{QC}(\Gamma_1)$                                          & The hyperbolic component $\mathcal{H}_{\Gamma_1}$ bifurcates to $\mathcal{H}_{\Gamma_2}$        
\\ \toprule
\end{tabular} 
\end{center}

\subsection*{Notes and references.}
Aspects of the Sullivan's dictionary were already anticipated by Fatou \cite[p. 22]{Fat29}.
Part of the correspondence in Theorem \ref{thm:c} was recently established in \cite{LLMM19} where the Tischler graphs were assumed to be triangulations. A classification of critically fixed anti-rational maps has also been obtained in \cite{Gey20}.
There is also a connection between Kleinian reflection groups, anti-rational maps, and Schwarz reflections of quadrature domains explained in \cite{LLMM18,LLMM19,LMM19,LMM20}. 
In the holomorphic setting, critically fixed rational maps have been studied in \cite{CGNPP15, H19}.

Many critically fixed anti-rational maps have large symmetry groups.
The examples corresponding to the five platonic solids are listed in Figure \ref{fig:ps}.
The counterparts in the holomorphic setting were constructed in \cite{DMcM89} (see also \cite{Buff, HJM12}).

The connections between number theoretic problems for circle packings and equidistribution results for Kleinian groups can be found in \cite{KO11, OS16, KN19}.

\subsection*{Some questions}
The dictionary between kissing reflection groups and critically fixed anti-rational maps raises some natural questions. Let us fix a $2$-connected, simple, plane graph $\Gamma$.
\begin{itemize}
\item Although the limit and Julia sets $\Lambda(G_\Gamma)$ and $\mathcal{J}(\mathcal{R}_\Gamma)$ are homeomorphic, the dynamically natural homeomorphism is not a quasisymmetry because it sends parabolic fixed points to repelling ones. In fact, we believe that there is no quasiconformal homeomorphism of the sphere carrying one to the other. In the case when $\Gamma$ is $3$-connected, this follows from the observation that the boundaries of two components of the domain of discontinuity of $G_\Gamma$ touch at a cusp (with zero angle), while the boundaries of two Fatou components of $\mathcal{R}_\Gamma$ touch at a repelling (pre-) fixed point (with a positive angle). However, in general, we do not know how to rule out the existence of `exotic' quasiconformal homeomorphisms between $\Lambda(G_\Gamma)$ and $\mathcal{J}(\mathcal{R}_\Gamma)$.

Moreover, since the quasisymmetry group of a fractal is a quasiconformal invariant, it will be interesting to know the quasisymmetry groups of $\Lambda(G_\Gamma)$ and $\mathcal{J}(\mathcal{R}_\Gamma)$. Note that according to \cite{LLMM19}, the quasisymmetry groups of these two fractals are isomorphic when the dual of $\Gamma$ is an unreduced triangulation.

\item Our proof of the existence of a homeomorphism between $\Lambda(G_\Gamma)$ and $\mathcal{J}(\mathcal{R}_\Gamma)$ only makes use of the conformal dynamics of the group and the anti-rational map on $\widehat{\C}$. One would like to know if there is a direct $3$-dimensional interpretation of this result.
\end{itemize}

\subsection*{Structure of the paper.}
We collect various known circle packing theorems in \S \ref{sec:cp}.
Based on this, we prove the connection between kissing reflection groups and simple plane graphs in \S \ref{sec:cc}.
In particular, the group part of Theorem \ref{thm:c}, Theorem \ref{thm:gm}, Theorem \ref{thm:glj}, and Theorem \ref{thm:a} are proved in Proposition \ref{prop:2c}, Propositions \ref{prop:qb}, \ref{prop:fkrg}, \ref{prop:mg}, Proposition \ref{prop:3g}, and Proposition \ref{prop:3connected} respectively.

Critically fixed anti-rational maps are studied in \S \ref{sec:cf}.
The anti-rational map part of Theorem \ref{thm:c} is proved in Proposition \ref{prop:tc}.
Once this is established, the anti-rational map part of Theorem \ref{thm:gm} and Theorem \ref{thm:glj} follow from their group counterparts as explained in Corollary \ref{cor:h3}.
Theorem \ref{thm:mig} is proved in Proposition \ref{prop:mig} and Corollary \ref{cor:mig}.

\vspace{2mm}

\noindent\textbf{Acknowledgements.} The authors would like to thank Curt McMullen for useful suggestions. They also thank Mikhail Hlushchanka for stimulating discussions. The authors are grateful to the anonymous referee for their valuable input.

The third author thanks the Institute for Mathematical Sciences, and Simons Center for Geometry and Physics at Stony Brook University for their hospitality during part of the work on this project. 

\section{Circle Packings}\label{sec:cp}
In this paper, a {\em circle packing} $\mathcal{P}$ is a connected finite collection of (oriented) circles in $\widehat{\C}$ with disjoint interiors.
Unless stated otherwise, the circle packings in this paper are assumed to contain at least three circles.
The combinatorics of configuration of a circle packing can be described by its {\em contact graph} $\Gamma$: we associate a vertex to each circle, and two vertices are connected by an edge if and only if the two associated circles intersect.
The embedding of the circles in $\widehat\C$ determines the isomorphism class of its contact graph as a plane graph.
It can also be checked easily that the contact graph of a circle packing is simple.
This turns out to be the only constraint for the graph (See \cite[Chapter 13]{Thurston78}).

\begin{theorem}[Circle Packing Theorem]\label{thm:cpt}
Every connected, simple, plane graph  is isomorphic to the contact graph of some circle packing.
\end{theorem}

The definition of contact graphs can be easily generalized to an infinite collection of circles with disjoint interiors.
In this paper, an {\em infinite circle packing} $\mathcal{P}$ is an infinite collection of (oriented) circles in $\hat\C$ with disjoint interiors, whose contact graph is connected.

\subsection*{$k$-connected graphs.}
A graph $\Gamma$ is said to be {\em $k$-connected} if $\Gamma$ contains more than $k$ vertices and remains connected if any $k-1$ vertices and their corresponding incident edges are removed.
A graph is $\Gamma$ is said to be {\em polyhedral} if $\Gamma$ is the $1$-skeleton of a convex polyhedron.
According to Steinitz's theorem, a graph is polyhedral if and only if it is $3$-connected and planar.

Given a polyhedral graph, we have a stronger version of the circle packing theorem \cite{Sch92} (cf. midsphere for canonical polyhedron).

\begin{theorem}[Circle Packing Theorem for polyhedral graphs]\label{thm:gcpt}
Suppose $\Gamma$ is a polyhedral graph, then there is a pair of circle packings whose contact graphs are isomorphic to $\Gamma$ and its planar dual.
Moreover, the two circle packings intersect orthogonally at their points of tangency.

This pair of circle packings is unique up to M\"obius transformations.
\end{theorem}

\subsection*{Marked contact graphs.}
In many situations, it is better to work with a marking on the graph as well as the circle packing.
A {\em marking} of a graph $\Gamma$ is a choice of the graph isomorphism
$$
 \phi: \mathscr{G}\longrightarrow \Gamma,
$$
where $\mathscr{G}$ is the underlying abstract graph of $\Gamma$.
We will refer to the pair $(\Gamma, \phi)$ as a {\em marked graph}.

Given two marked plane graphs $(\Gamma_1, \phi_1)$ and $(\Gamma_2, \phi_2)$ with the same underlying abstract graph $\mathscr{G}$, we say that they are equivalent if $\phi_2\circ \phi_1^{-1}:\Gamma_1 \longrightarrow \Gamma_2$ is an isomorphism of plane graphs.

Similarly, a circle packing $\mathcal{P}$ is said to be marked if the associated contact graph is marked.

\section{Kissing Reflection Groups}\label{sec:cc}
Let $\Gamma$ be a marked connected simple plane graph.
By the circle packing theorem, $\Gamma$ is (isomorphic to) the contact graph of some marked circle packing 
$$
\mathcal{P}=\{C_1,..., C_n\}.
$$
We define the {\em kissing reflection group} associated to this circle packing $\mathcal{P}$ as
$$
G_\mathcal{P} := \langle g_1,..., g_n\rangle,
$$
where $g_i$ is the reflection along the circle $C_i$.

We denote the group of all M{\"o}bius automorphisms of $\widehat{\C}$ by $\textrm{Aut}^+(\widehat{\C})$, and the group of all M{\"o}bius and anti-M{\"o}bius automorphisms of $\widehat{\C}$ by $\textrm{Aut}^\pm(\widehat{\C})$.
Note that since a kissing reflection group is a discrete subgroup of $\textrm{Aut}^\pm(\widehat{\C})$, definitions of limit set and domain of discontinuity can be easily extended to kissing reflection groups.
We shall use $\widetilde{G}_\mathcal{P}$ to denote the index two subgroup of $G_\mathcal{P}$ consisting of orientation preserving elements.
Note that $\widetilde{G}_\mathcal{P}$ lies on the boundary of Schottky groups.

We remark that if $\mathcal{P}'$ is another circle packing realizing $\Gamma$,
since the graph $\Gamma$ is marked, there is a canonical identification of the circle packing $\mathcal{P}'$ with $\mathcal{P}$.
This gives a canonical isomorphism between the kissing reflection groups $G_{\mathcal{P}'}$ and $G_\mathcal{P}$, which is induced by a quasiconformal map.
We refer to $\Gamma$ as the contact graph associated to $G_\mathcal{P}$.

\subsection{Limit set and domain of discontinuity of kissing reflection groups.}\label{limit_connected_subsec}
Let $\mathcal{P} = \{C_1,..., C_n\}$ be a marked circle packing, and $D_i$ be the associated open disks for $C_i$.
Let $P$ be the set consisting of points of tangency for the circle packing $\mathcal{P}$.
Let 
$$
\Pi = \widehat{\C}\setminus\left(\bigcup_{i=1}^n D_i \cup P\right).
$$
Then $\Pi$ is a fundamental domain of the action of $G_\mathcal{P}$ on the domain of discontinuity $\Omega(G_\mathcal{P})$.
Denote $\Pi = \bigcup_{i=1}^k \Pi_i$ where $\Pi_i$ is a component of $\Pi$.
Note that when the limit set of $G_\mathcal{P}$ is connected, each component of the domain of discontinuity is simply connected (i.e., a conformal disk), and each component $\Pi_i$ is a closed ideal polygon in the corresponding component of the domain of discontinuity bounded by arcs of finitely many circles in the circle packing.

Since $G_\mathcal{P}$ is a free product of finitely many copies of $\Z/2\Z$, each element $g \in G_\mathcal{P}$ admits a shortest expression $g = g_{i_1}...g_{i_l}$ ($i_r\neq i_{r+1}$, for $r\in\{1,\cdots, l-1\}$) with respect to the standard generating set $S = \{g_1,..., g_n\}$. The integer $l$ is called the \emph{length} of the group element $g$, thought of as a word in terms of the generators.

The following lemma follows directly by induction.
\begin{lemma}\label{lem:im}
Let $g=g_{i_1}...g_{i_l}$ be an element of length $l$, then $g\cdot \Pi \subseteq \overline{D}_{i_1}$.
\end{lemma}

We set $\Pi^1:= \bigcup_{i=1}^n g_i\cdot \Pi$ and $\Pi^{j+1} = \bigcup_{i=1}^n g_i\cdot (\Pi^j\setminus\overline{D}_i)$.
For consistency, we also set  $\Pi^0:= \Pi$.
We call $\Pi^l$ the \emph{tiling} of \emph{level} $l$. The following lemma justifies this terminology.

\begin{lemma}\label{lem:genl}
$$\Pi^{l} = \bigcup_{\vert g\vert = l} g\cdot \Pi.$$
\end{lemma}
\begin{proof}
We will prove this by induction. The base case is satisfied by the definition of $\Pi^1$.
Assume that $\Pi^{j} = \bigcup_{\vert g\vert = j} g\cdot \Pi$.
Let $g = g_{i_1}g_{i_2}...g_{i_{j+1}}$ be of length $j+1$.
Note that $i_1 \neq i_2$.
By Lemma \ref{lem:im}, $g_{i_2}g_{i_3}...g_{i_{j+1}}\cdot \Pi$ does not intersect $\overline{D}_{i_1}$, thus
$g\cdot \Pi \subseteq g_{i_1} \cdot (\Pi^j\setminus\overline{D}_{i_1})$.
So $\bigcup_{\vert g\vert = j+1} g\cdot \Pi \subseteq \Pi^{j+1}$.
The reverse inclusion can be proved similarly.
\end{proof}

Since the domain of discontinuity is the union of tiles of all levels, Lemma \ref{lem:genl} implies that
$$
\Omega(G_\mathcal{P}) = \bigcup_{g\in G_{\mathcal{P}}} g\cdot \Pi = \bigcup_{i=0}^\infty \Pi^i.
$$

Similarly, we set $\mathcal{D}^i = \widehat\C \setminus \bigcup_{j=0}^i\Pi^j$.
We remark that $\mathcal{D}^i$ is neither open nor closed: it is a finite union of open disks together with the orbit of $P$ under the group elements of length up to $i$ on the boundaries of these disks.

Let $\overline{\mathcal{D}^i}$ be its closure.
Note that each $\overline{\mathcal{D}^i}$ is a union of closed disks for some (possibly disconnected) finite circle packing.
Indeed, 
$$
\overline{\mathcal{D}^0} = \overline{\widehat\C\setminus\Pi^0} = \bigcup_{i=1}^n \overline{D}_i
$$ 
is the union of the closed disks corresponding to the original circle packing $\mathcal{P}$.
By induction, we have that at level $i+1$,
\begin{align}\label{eqn:1}
\overline{\mathcal{D}^{i+1}} = \bigcup_{j=1}^n g_j \cdot \overline{\mathcal{D}^i\setminus\overline{D}_j}
\end{align}
is the union of the images of the level $i$ disks outside of $\overline{D}_j$ under $g_j$.

We also note that the sequence $\overline{\mathcal{D}}^i$ is nested, and thus the limit set
$$
\Lambda(G_\mathcal{P}) = \bigcap_{i=0}^\infty \mathcal{D}^i = \bigcap_{i=0}^\infty \overline{\mathcal{D}^i}.
$$
Therefore, we have the following expansive property of the group action on $\Lambda(G_\mathcal{P})$.
\begin{lemma}\label{lem:maxr}
Let $r_n$ be the maximum spherical diameter of the disks in $\mathcal{D}^n$. Then $r_n \to 0$.
\end{lemma}
\begin{proof}
Otherwise, we can construct a sequence of nested disks of radius bounded from below implying that the limit set contains a disk, which is a contradiction.
\end{proof}

We now prove the group part of Theorem~\ref{thm:c}.
\begin{prop}\label{prop:2c}
The kissing reflection group $G_\mathcal{P}$ has connected limit set if and only if the contact graph $\Gamma$ of $\mathcal{P}$ is $2$-connected.
\end{prop}
\begin{proof}
If $\Gamma$ is not $2$-connected, then there exists a circle (say $C_1$) such that the circle packing becomes disconnected once we remove it.
Then we see $\mathcal{D}^1 \subseteq \widehat\C$ is disconnected by Equation \ref{eqn:1}. This forces the limit set to be disconnected as well (see Figure \ref{fig:dls}).

On the other hand, if $\Gamma$ is $2$-connected, then $\mathcal{D}^1$ is connected by Equation \ref{eqn:1}. Now by induction and Equation \ref{eqn:1} again, $\mathcal{D}^i$ is connected for all $i$.
Thus, $\Lambda(G_\mathcal{P}) = \bigcap_{i=0}^\infty \mathcal{D}^i$ is also connected.
\end{proof}

Proposition \ref{prop:2c} and the definition of kissing reflection groups show that the association of a $2$-connected simple plane graph with a kissing reflection group with connected limit set is well defined and surjective.
To verify that this is indeed injective, we remark that if $\mathcal{P}$ and $\mathcal{P}'$ are two circle packings associated to two contact graphs that are non-isomorphic as plane graphs, then the closures of the fundamental domains $\overline{\Pi}$ and $\overline{\Pi'}$ are not homeomorphic.
Note that the touching patterns of different components of $\Pi$ or $\Pi'$ completely determine the structures of the pairing cylinders of the associated $3$-manifolds at the cusps (See \cite[\S 2.6]{Marden74}).
This means that the conformal boundaries of $\Hyp^3/\widetilde{G}_{\mathcal{P}}$ and $\Hyp^3/\widetilde{G}_{\mathcal{P}'}$ with the pairing cylinder structures are not the same.
Thus, the two kissing reflection groups $G_\mathcal{P}$ and $G_{\mathcal{P}'}$ are not quasiconformally isomorphic.

\begin{figure}[h!]
  \centering
  \includegraphics[width=.8\linewidth]{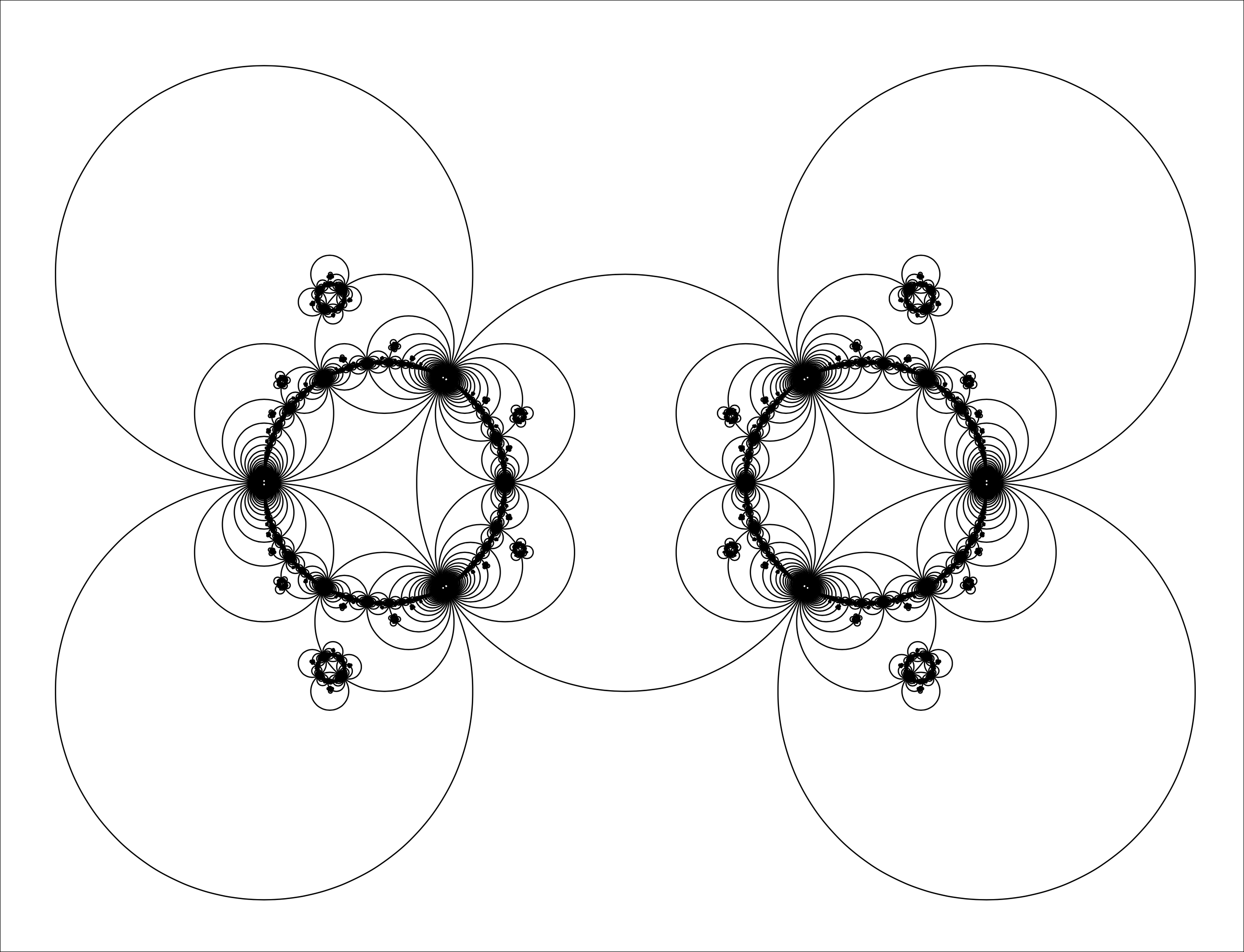}  
  \caption{A disconnected limit set for a kissing reflection group $G$ with non $2$-connected contact graph. $G$ is generated by reflections along the $5$ visible large circles in the figure.}
\label{fig:dls}
\end{figure}

\subsection{Acylindrical kissing reflection groups}\label{subsec:akrg}
Recall that $\widetilde{G}_\mathcal{P}$ is the index $2$ subgroup of $G_\mathcal{P}$ consisting of orientation preserving elements.
We set 
$$
\mathcal{M}(G_\mathcal{P}) := \Hyp^3 \cup \Omega(G_\mathcal{P})/\widetilde{G}_\mathcal{P}
$$
to be the associated $3$-manifold with boundary.
Note that the boundary 
$$
\partial \mathcal{M}(G_\mathcal{P}) = \Omega(G_\mathcal{P})/\widetilde{G}_\mathcal{P}
$$
is a finite union of punctured spheres.
Each punctured sphere corresponds to the double of a component of $\Pi$.

Let $F$ be a face of $\Gamma$. Then it corresponds to a component $\Pi_F$ of $\Pi$, which also corresponds to a component $R_F$ of $\partial\mathcal{M}(G_\mathcal{P})$.
More precisely, there is a unique component $\Omega_F$ of $\Omega(G_\mathcal{P})$ containing $\Pi_F$, and 
$$
R_F\cong \Omega_F / \text{stab}(\Omega_F),
$$
where $\text{stab}(\Omega_F)$ is the stabilizer of $\Omega_F$ in $\widetilde{G}_\mathcal{P}$.

A compact $3$-manifold $M^3$ with boundary is called {\em acylindrical} if $M^3$ contains no essential cylinders and is boundary incompressible.
Here an essential cylinder $C$ in $M^3$ is a closed cylinder $C$ such that $C\cap \partial M^3 = \partial C$, the boundary components of $C$ are not homotopic to points in $\partial M^3$ and $C$ is not homotopic into $\partial M^3$.
$M^3$ is said to be boundary incompressible if the inclusion $\pi_1(R) \xhookrightarrow{} \pi_1(M^3)$ is injective for every component $R$ of $\partial M^3$. (We refer the readers to \cite[\S 3.7, \S 4.7]{Marden16} for detailed discussions.)

Our manifold $\mathcal{M}(G_\mathcal{P})$ is not a compact manifold as there are parabolic elements (cusps) in $\widetilde{G}_\mathcal{P}$.
Thurston \cite{Thurston86} introduced the notion of {\em pared manifolds} to work with Kleinian groups with parabolic elements.
In our setting, we can also use an equivalent definition without introducing pared manifolds.
To start the definition, we note that for a geometrically finite group, associated with the conjugacy class of a rank one cusp, there is a pair of punctures $p_1, p_2$ on $\partial M^3$. 
If $c_1, c_2$ are small circles in $\partial M$ retractable to $p_1, p_2$, then there is a {\em pairing cylinder} $C$ in $M^3$, which is a cylinder bounded by $c_1$ and $c_2$ (see \cite[\S 2.6]{Marden74}, \cite[p. 125]{Marden16}).
\begin{definition}
A kissing reflection group $G_\mathcal{P}$ is said to be {\em acylindrical} if $\mathcal{M}(G_\mathcal{P})$ is boundary incompressible and every essential cylinder is homotopic to a pairing cylinder.
\end{definition}

Note that $\mathcal{M}(G_\mathcal{P})$ is boundary incompressible if and only if each component of $\Omega(G_\mathcal{P})$ is simply connected if and only if the limit set $\Lambda(G_\mathcal{P})$ is connected.
We also note that the acylindrical condition is a quasiconformal invariant, and hence does not depend on the choice of the circle packing $\mathcal{P}$ realizing a simple, connected, plane graph $\Gamma$.
In the remainder of this section, we shall prove the following characterization of acylindrical kissing Kleinian reflection groups.
\begin{prop}\label{prop:3connected}
The kissing reflection group $G_\mathcal{P}$ is acylindrical if and only if the contact graph $\Gamma$ of $\mathcal{P}$ is $3$-connected.
\end{prop}

This proposition will be proved after the following lemmas.
Let $\Gamma$ be a $2$-connected simple plane graph, and $\mathcal{P} = \{C_1,..., C_n\}$ be a realization of $\Gamma$.
Let $G_\mathcal{P}$ be the kissing reflection group, with generators $g_1,..., g_n$ given by reflections along $C_1,..., C_n$.
Note that a face $F$ of $\Gamma$ corresponds to a component $R_F$ of $\partial\mathcal{M}(G_\mathcal{P})$.

Any two non-adjacent vertices $v, w$ of the face $F$ give rise to an essential simple closed curve $\widetilde{\gamma}^F_{vw}$ on $R_F$ (a simple closed curve on a surface is essential if it is not homotopic to a point or a puncture).
More precisely, let $g_v, g_w$ be the reflections associated to the two vertices, then $g_vg_w \in \text{stab}(\Omega_F)$ is a loxodromic element under the uniformization of $R_F$ which gives the simple closed curve $\widetilde{\gamma}^F_{vw}$ on $R_F$.
Note that $g_vg_w$ itself may not be loxodromic as the vertices $v,w$ may be adjacent in some other faces. 
If this is the case, then we have an accidental parabolic element (see \cite[p. 198, 3-17]{Marden16}).

We first prove the following graph theoretic lemma.
\begin{lemma}\label{lem:graph3c}
Let $\Gamma$ be a $2$-connected simple plane graph. If $\Gamma$ is not $3$-connected, then there exist two vertices $v, w$ so that $v, w$ lie on the intersection of the boundaries of two faces $F_1$ and $F_2$. Moreover, they are non-adjacent for at least one of the two faces.
\end{lemma}

\begin{figure}[ht!]
\begin{tikzpicture}[thick]
   \node at (7,4) [circle,fill=black,inner sep=3pt] {};
      \node at (7,6) [circle,fill=black,inner sep=3pt] {};
         \node at (5,6.4) [circle,fill=black,inner sep=3pt] {};
   \node at (9,6.4) [circle,fill=black,inner sep=3pt] {}; 
      \node at (4.5,4.5) [circle,fill=black,inner sep=3pt] {};
         \node at (9.5,4.5) [circle,fill=black,inner sep=3pt] {};
            \node at (5.4,2.5) [circle,fill=black,inner sep=3pt] {};
      \node at (8.6,2.5) [circle,fill=black,inner sep=3pt] {};
         \node at (7,1) [circle,fill=black,inner sep=3pt] {};
         \node at (0,0) {};
         
         \draw[-,line width=1pt] (7,4)->(7,6);
   \draw[-,line width=1pt] (7,4)->(5.4,2.5);
  \draw[-,line width=1pt] (7,4)->(8.6,2.5);
  \draw[-,line width=1pt] (7,6)->(5,6.4);
    \draw[-,line width=1pt] (7,6)->(9,6.4);
  \draw[-,line width=1pt] (5,6.4)->(4.5,4.5);
 \draw[-,line width=1pt] (4.5,4.5)->(5.4,2.5);
  \draw[-,line width=1pt] (5.4,2.5)->(7,1);
  \draw[-,line width=1pt] (7,1)->(8.6,2.5);
  \draw[-,line width=1pt] (8.6,2.5)->(9.5,4.5);
  \draw[-,line width=1pt] (9.5,4.5)->(9,6.4);
  
  \node at (7,3.48) {\begin{Large}$v$\end{Large}};
  \node at (7,6.6) {\begin{Large}$v_{1,4}=v_{2,1}$\end{Large}};
  \node at (4.4,6.6) {\begin{Large}$v_{2,2}$\end{Large}};
  \node at (9.6,6.6) {\begin{Large}$v_{1,3}$\end{Large}};
  \node at (3.8,4.5) {\begin{Large}$v_{2,3}$\end{Large}};
  \node at (10.2,4.5) {\begin{Large}$v_{1,2}$\end{Large}};
  \node at (4,2.2) {\begin{Large}$v_{2,4}=v_{3,1}$\end{Large}};
  \node at (7,0.5) {\begin{Large}$v_{3,2}$\end{Large}};
  \node at (10,2.2) {\begin{Large}$v_{1,1}=v_{3,3}$\end{Large}};
  
  \node at (5.6,5)  {\begin{Large}$F_2$\end{Large}}; 
   \node at (8.4,5)  {\begin{Large}$F_1$\end{Large}};  
  \node at (7,2.3)  {\begin{Large}$F_3$\end{Large}};
   \end{tikzpicture}
   \caption{A schematic picture of the potentially non-simple cycle around $v$.}
   \label{graph_lemma_fig}
   \end{figure}
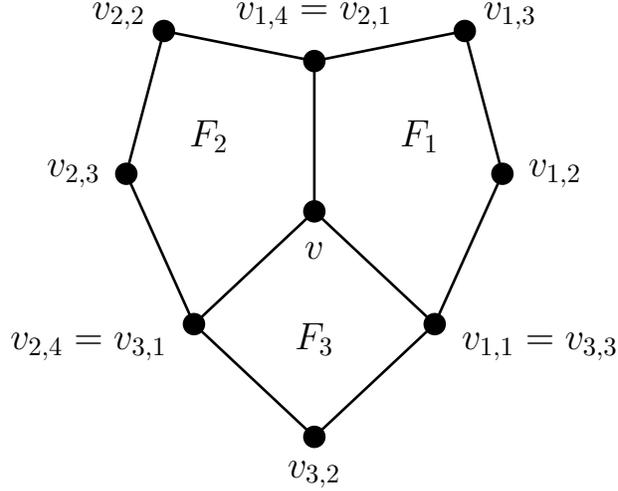

\begin{proof}
As $\Gamma$ is not $3$-connected, there exist two vertices $v, w$ so that $\Gamma\setminus\{v, w\}$ is disconnected.
Let $F_1,..., F_k$ be the list of faces that contain $v$ as a vertex.
Since $\Gamma$ is plane, we may assume that the faces $F_i$ are ordered around $v$ counter clockwise.
Since $\Gamma$ is plane and $2$-connected, each face $F_i$ is a Jordan domain.
Let $v_{i,0} = v, v_{i,1},..., v_{i, j_i}$ be the vertices of $F_i$ ordered counter clockwise.
Since the faces $F_i$ are ordered counterclockwise, we have that $v_{i,j_i} = v_{i+1, 1}$.
We remark that there might be additional identifications.
Then 
$$
v_{1,1}\to v_{1,2}\to...\to v_{1,j_1} (=v_{2,1}) \to v_{2,2}\to ...\to v_{k,j_k} = v_{1,1}
$$ 
form a (potentially non-simple) cycle $C$ (see Figure~\ref{graph_lemma_fig}). Since $\Gamma$ is $2$-connected, $\Gamma \setminus\{w\}$ is connected.
Thus, in particular, any vertex $p$ is connected to $C\setminus\{w\}$ in $\Gamma\setminus\{w\}$.
Thus, if $w \notin C$ or $w$ only appears once in $C$, then $C\setminus\{w\}$ is connected.
This would imply that $\Gamma\setminus\{v,w\}$ is connected, which is a contradiction.
Therefore, $w$ must appear at least twice in the cycle $C$.
Since each face is a Jordan domain, $w$ must appear on the boundaries of at least $2$ faces $F_{i_1}$ and $F_{i_2}$.

Since $\Gamma$ is simple, $w$ is adjacent to $v$ in at most $2$ faces, in which case $w = v_{i,j_i} = v_{i+1,1}$, i.e., it contributes to only one point in $C$.
Therefore, there exists a face on which $w$ is not adjacent to $v$.
This proves the lemma.
\end{proof}

We can now prove one direction of Proposition \ref{prop:3connected}.
\begin{lemma}\label{lem:fd3c}
If $G_\mathcal{P}$ is acylindrical, then $\Gamma$ is $3$-connected.
\end{lemma}
\begin{proof}
Note that $\Gamma$ must be $2$-connected as $\Lambda(G_\mathcal{P})$ is connected (by the boundary incompressibility condition). 
We will prove the contrapositive, and assume that $\Gamma$ is not $3$-connected.
Let $v,w$ be the two vertices given by Lemma \ref{lem:graph3c}.
There are two cases.

If $v, w$ are non-adjacent vertices in two faces $F_1, F_2$, then $g_vg_w$ gives a pair of essential simple closed curves on $R_{F_1}$ and $R_{F_2}$ in $\partial\mathcal{M}(G_\mathcal{P})$.
This pair bounds an essential cylinder (see the Cylinder Theorem in \cite[\S 3.7]{Marden16}), which is not homotopic to a pairing cylinder of two punctures (see $g_Cg_{C''}$ in Figure~\ref{fig:njd1}).

If $v, w$ are non-adjacent vertices in $F_1$ but adjacent vertices in $F_2$, then $g_vg_w$ corresponds to an essential simple closed curve in $R_{F_1}$, and a simple closed curve homotopic to a puncture in $R_{F_2}$.
Then $g_vg_w$ is an accidental parabolic, and the two curves bound an essential cylinder which is not homotopic to a pairing cylinder (see $g_Cg_{C'}$ in Figure~\ref{fig:njd1}).

Therefore, in either case, $G_\mathcal{P}$ is cylindrical.
\end{proof}

\subsection*{Gasket limit set.}
Recall that a closed set $\Lambda$ is a {\em round gasket} if
\begin{itemize}
\item $\Lambda$ is the closure of some infinite circle packing; and
\item the complement of $\Lambda$ is a union of round disks which is dense in $\widehat{\C}$.
\end{itemize}
We call a homeomorphic copy of a round gasket a {\em gasket}.
See \S~\ref{sec:cp} for our definition of infinite circle packings.

If $\Gamma$ is $3$-connected, then $\Gamma$ is a polyhedral graph.
Let $\Gamma^{\vee}$ be the planar dual of $\Gamma$. Then, Theorem \ref{thm:gcpt} gives a (unique) pair of circle packings $\mathcal{P}$ and $\mathcal{P}^{\vee}$ whose contact graphs are isomorphic to $\Gamma$ and $\Gamma^{\vee}$ (respectively) as plane graphs such that $\mathcal{P}$ and $\mathcal{P}^{\vee}$ intersect orthogonally at their points of tangency (see Figure~\ref{fig:cube}).
Let $G_\mathcal{P}$ be the kissing reflection group associated with $\mathcal{P}$.
Since the circle packing $\mathcal{P}^{\vee}$ is dual to $\mathcal{P}$, we have that
$$
\bigcup_{g\in G_\mathcal{P}} \bigcup_{C\in \mathcal{P}^{\vee}} g\cdot C
$$
is an infinite circle packing, and the limit is the closure
$$
\Lambda(G_\mathcal{P}) = \overline{\bigcup_{g\in G_\mathcal{P}} \bigcup_{C\in \mathcal{P}^{\vee}} g\cdot C}.
$$
Since $\Lambda(G_\mathcal{P})$ is nowhere dense, and the complement is a union of round disks, we conclude that $\Lambda(G_\mathcal{P})$ is a round gasket.

\begin{figure}[h!]
  \centering
 {\includegraphics[width=.64\linewidth]{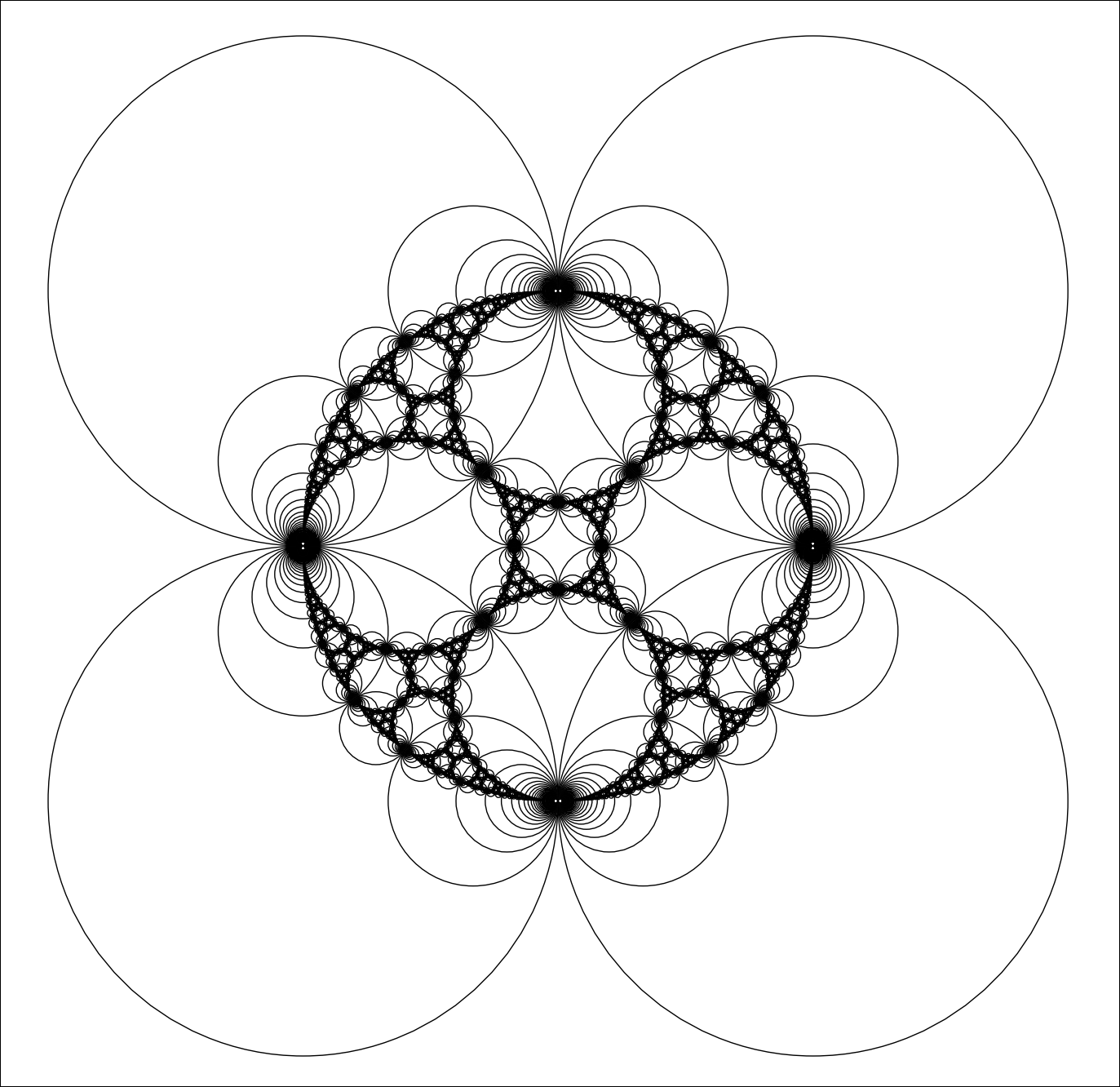}};
   \caption{The limit set of a kissing reflection group $G$ with a $3$-connected contact graph.}
   \label{fig:cube}
  \end{figure}

Note that each component of $\Omega(G_\mathcal{P})$ is of the form $g\cdot D$ where $g\in G_\mathcal{P}$ and $D$ is a disk in the dual circle packing $\mathcal{P}^{\vee}$.
By induction, we have the following.
\begin{lemma}
If $\Gamma$ is $3$-connected, then the closure of any two different components of $\Omega(G_\mathcal{P})$ only intersect at cusps.
\end{lemma}

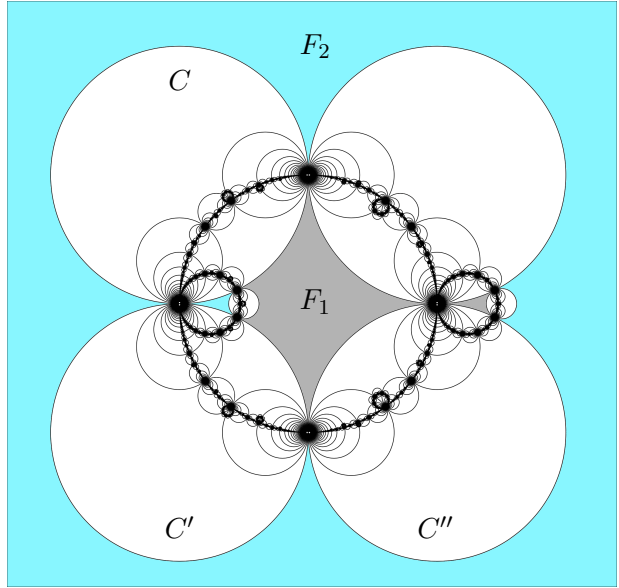
\begin{figure}[ht!]
  \centering
  \begin{tikzpicture}
   \node[anchor=south west,inner sep=0] at (2.5,0) {\includegraphics[width=.64\linewidth]{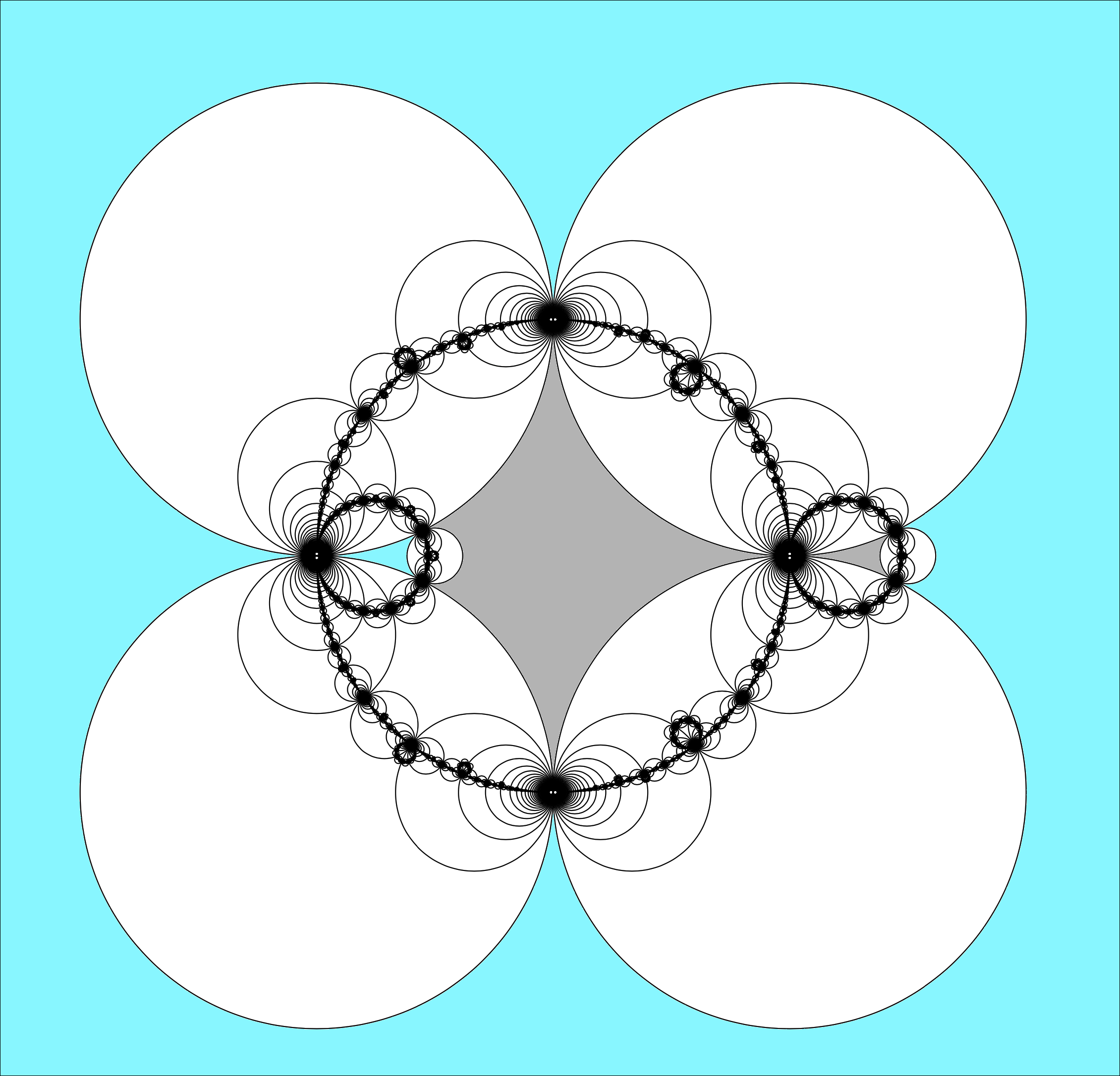}};
   \node at (4.8,6.75) {$C$};
   \node at (4.8,0.8) {$C'$};
   \node at (8.2,0.8) {$C''$};
   \node at (6.6,3.8) {$F_1$};
   \node at (6.6,7.2) {$F_2$};
  \end{tikzpicture}
   \caption{The limit set of a kissing reflection group $G$ with Hamiltonian but non $3$-connected contact graph. 
The unique Hamiltonian cycle of the associated contact graph $\Gamma$ divides the fundamental domain $\Pi(G)$ into two parts $\Pi^\pm$, which are shaded in grey and blue. With appropriate markings, $G$ is the mating of two copies of the group $H$ shown in Figure~\ref{fig:njd}.}

\label{fig:njd1}
\end{figure}

We have the following characterization of gasket limit set for kissing reflection groups.
\begin{prop}\label{prop:3g}
Let $\Gamma$ be a simple plane graph, then $\Lambda(G_\mathcal{P})$ is a gasket if and only if $\Gamma$ is $3$-connected.
\end{prop}
\begin{proof}
Indeed, from the above discussion, if $\Gamma$ is $3$ connected, then $\Lambda(G_\mathcal{P})$ is a gasket.

Conversely, if $\Gamma$ is not $2$-connected, then $\Lambda(G_\mathcal{P})$ is disconnected by Proposition \ref{prop:2c}, so it is not a gasket.
On the other hand, if $\Gamma$ is $2$-connected but not $3$-connected, by Lemma \ref{lem:graph3c}, we have two vertices $v, w$ so that $v,w$ lie on the intersection of the boundaries of two faces $F_1$ and $F_2$.
If $v, w$ are non-adjacent vertices in both $F_1$ and $F_2$, then the corresponding components $\Omega_{F_1}$ and $\Omega_{F_2}$ touch at two points corresponding to the two fixed points of the loxodromic element $g_vg_w$ (see $g_Cg_{C''}$ in Figure~\ref{fig:njd1}).
Thus, $\Lambda(G_\mathcal{P})$ is not a gasket.
If $v, w$ are non-adjacent vertices in $F_1$ but adjacent vertices in $F_2$, then $g_vg_w$ gives an accidental parabolic element.
The corresponding component $\Omega_{F_1}$ is not a Jordan domain as the unique fixed point of the parabolic element $g_vg_w$ corresponds to two points on the ideal boundary of $\Omega_{F_1}$ (see $g_Cg_{C'}$ in Figure~\ref{fig:njd1}).
Therefore $\Lambda(G_\mathcal{P})$ is not a gasket.
\end{proof}

In the course of the proof, we have actually derived the following characterization which is worth mentioning.

\begin{prop}\label{why_not_gasket_prop}
Let $\mathcal{P}$ be a circle packing whose contact graph $\Gamma$ is not $3$-connected, and $G_\mathcal{P}$ be the associated kissing reflection group, then either
\begin{itemize}
\item there exists a component of $\Omega(G_\mathcal{P})$ which is not a Jordan domain; or
\item there exist two components of $\Omega(G_\mathcal{P})$ whose closures touch at least at two points.
\end{itemize}
\end{prop}

We are now able to prove the other direction of Proposition \ref{prop:3connected}:
\begin{lemma}\label{lem:bd3c}
If $\Gamma$ is $3$-connected, then $G_\mathcal{P}$ is acylindrical.
\end{lemma}
\begin{proof}
Since $\Gamma$ is $3$-connected, it follows that the closures of any two different components of $\Omega(G_\mathcal{P})$ intersect only at cusps.
This means that there are no essential cylinder other than the pairing cylinders of the rank one cusps.
So $G_\mathcal{P}$ is acylindrical.
\end{proof}

\begin{proof}[Proof of Proposition \ref{prop:3connected}]
This follows from Lemma \ref{lem:fd3c} and \ref{lem:bd3c}.
\end{proof}

We remark that the unique configuration of pairs of circle packings given in Theorem \ref{thm:gcpt} gives a kissing reflection group with {\em totally geodesic boundary}.
This unique point in the deformation space of acylindrical manifolds is guaranteed by a Theorem of McMullen \cite{McM90}.

\subsection{Deformation spaces of kissing reflection groups.}\label{sec:ads}
Throughout this section, we will use bold symbols, such as $\mathbf{G}, \mathbf{G}_\Gamma$, to represent the base point for the corresponding deformation spaces.
We use regular symbols $G$ to represent the image of a representation in the deformation spaces.
If the group is a kissing reflection group, we also use $G_\mathcal{P}$ if we want to emphasize the corresponding circle packing is $\mathcal{P}$.

\subsection*{Definition of $\textrm{AH}(\mathbf{G})$.}
Let $\mathbf{G}$ be a finitely generated discrete subgroup of $\textrm{Aut}^\pm(\widehat{\C})$.
A representation (i.e., a group homomorphism) $\xi: \mathbf{G} \longrightarrow \textrm{Aut}^\pm(\widehat{\C})$ is said to be weakly type preserving if 
\begin{enumerate}
\item $\xi(g) \in \textrm{Aut}^+(\widehat\C)$ if and only if $g\in \textrm{Aut}^+(\widehat\C)$, and
\item if $g \in \textrm{Aut}^+(\widehat\C)$, then $\xi(g)$ is parabolic whenever $g$ is parabolic.
\end{enumerate}
Note that a weakly type preserving representation may send a loxodromic element to a parabolic one.

\begin{definition}\label{AH_def}
Let $\mathbf{G}$ be a finitely generated discrete subgroup of $\textrm{Aut}^\pm(\widehat{\C})$.
\begin{align*}
\textrm{AH}(\mathbf{G})
&:=  \lbrace\xi: \mathbf{G}\longrightarrow G\ \textrm{is a weakly type preserving isomorphism to}
\\
& \textrm{a discrete subgroup}\ G\ \textrm{of}\ \textrm{Aut}^\pm(\widehat{\C})\rbrace / \sim,\;
\end{align*}
where $\xi_1\sim\xi_2$ if there exists a M{\"o}bius transformation $M$ such that $$\xi_2(g)=M\circ\xi_1(g)\circ M^{-1},\ \textrm{for all}\ g\in \mathbf{G}.$$
\end{definition}

$\textrm{AH}(\mathbf{G})$ inherits the quotient topology of algebraic convergence.
Indeed, we say that a sequence of weakly type preserving representations $\{\xi_n\}$ converges to $\xi$ algebraically if $\{\xi_n(g_i)\}$ converges to $\xi(g_i)$ as elements of $\textrm{Aut}^\pm(\widehat{\C})$ for (any) finite generating set $\{g_i\}$.

\subsection*{Quasiconformal deformation space of $\mathbf{G}$.}
Recall that a Kleinian group is said to be {\em geometrically finite} if it has a finite sided fundamental polyhedron.
We say a finitely generated discrete subgroup of $\textrm{Aut}^\pm(\widehat{\C})$ is {\em geometrically finite} if the index $2$ subgroup is geometrically finite.

Let $\mathbf{G}$ be a finitely generated, geometrically finite, discrete subgroup of $\textrm{Aut}^\pm(\widehat{\C})$.
A group $G$ is called a {\em quasiconformal deformation} of $\mathbf{G}$ if there is a quasiconformal map $F: \widehat{\C} \longrightarrow \widehat{\C}$ that induces an isomorphism $\xi: \mathbf{G} \longrightarrow G$ such that $F \circ g(z) = \xi(g) \circ F(z)$ for all $g\in \mathbf{G}$ and $z\in \widehat{\C}$.
Such a group $G$ is necessarily geometrically finite and discrete.
The map $F$ is uniquely determined (up to normalization) by its Beltrami differential on the domain of discontinuity $\Omega(\mathbf{G})$ (See \cite[\S 5.1.2, Theorem~3.13.5]{Marden16}).

We define the {\em quasiconformal deformation space} of $\mathbf{G}$ as
$$
\mathcal{QC}(\mathbf{G}) = \{\xi \in \textrm{AH}(\mathbf{G}): \xi \text{ is induced by a quasiconformal deformation of } \mathbf{G}\}.
$$
By definition, $\mathcal{QC}(\mathbf{G})  \subseteq \textrm{AH}(\mathbf{G})$.

\subsection*{Kissing reflection groups.}\label{sec:krg}
The above discussion is quite general, and applies to any Kleinian (reflection) group.
In the following, we will specialize to the case of a kissing reflection group.

Recall that different realizations of a fixed marked, connected, simple, plane graph $\Gamma$ as circle packings $\mathcal{P}$ give canonically isomorphic kissing reflection groups $G_\mathcal{P}$. Thus, the algebraic/quasiconformal deformation spaces of all such $G_\mathcal{P}$ can be canonically identified. However, in order to study these deformation spaces, it will be convenient to fix a (marked)  circle packing realization $\mathcal{P}$ of the (marked) graph $\Gamma$, and use the associated kissing reflection group
$$
\mathbf{G}_\Gamma:=G_\mathcal{P}
$$ 
as the base point. With this choice, we can and will use the notation $\textrm{AH}(\Gamma)$ and $\mathcal{QC}(\Gamma)$ to denote $\textrm{AH}(\mathbf{G}_\Gamma)$ and $\mathcal{QC}(\mathbf{G}_\Gamma)$, respectively.

Now let $\Gamma_0$ be a simple plane graph. The goal of this section is to describe the algebraic deformation space $\textrm{AH}(\Gamma_0)$ and the closure $\overline{\mathcal{QC}(\Gamma_0)}$ (of the quasiconformal deformation space) in $\textrm{AH}(\Gamma_0)$.

In line with the convention introduced above, we choose a circle packing $\{\mathbf{C}_1,..., \mathbf{C}_n\}$ realizing $\Gamma_0$, and denote the reflection along $\mathbf{C}_i$ as $\rho_i$. Let $\mathbf{G}_{\Gamma_0}$ be the kissing reflection group generated by the reflections $\rho_i$. 
We remark that since we will be working with deformation spaces for several contact graphs, we use $\mathbf{G}_{\Gamma_0}$ to emphasize that the contact graph is $\Gamma_0$.

Let $G$ be a discrete subgroup of $\textrm{Aut}^\pm(\widehat\C)$, and let $\xi: \mathbf{G}_{\Gamma_0}\longrightarrow G$ be a weakly type preserving isomorphism.
Since the union of the circles $\mathbf{C}_1,..., \mathbf{C}_n$ is connected, for each $\rho_i$, there exists $\rho_j$ so that $\rho_i\circ\rho_j$ is parabolic.
This implies that no $\xi(\rho_i)$ is the antipodal map.
Hence, $\xi(\rho_i)$ is also a circular reflection.
Assume that $\xi(\rho_i)$ is the reflection along some circle $C_i$.
If $\rho_i \circ \rho_j$ is parabolic, then $\xi(\rho_i) \circ \xi(\rho_{j})$ is also parabolic.
Thus, $C_i$ is tangent to $C_j$ as well. This motivates the following definition.

\begin{definition}\label{abstract_domination_def}
Let $\Gamma$ be a simple plane graph with the same number of vertices as $\Gamma_0$, we say that $\Gamma$ {\em abstractly dominates} $\Gamma_0$, denoted by $\Gamma \geq_{a} \Gamma_0$, if there exists an embedding 
$\psi: \Gamma_0 \longrightarrow \Gamma$ as abstract graphs (i.e., if there exists a graph isomorphism between $\Gamma_0$ and a subgraph of $\Gamma$).

We say that $\Gamma$ {\em dominates} $\Gamma_0$, denoted by $\Gamma \geq \Gamma_0$, if there exists an embedding 
$\psi: \Gamma_0 \longrightarrow \Gamma$ as plane graphs (i.e., if there exists a graph isomorphism between $\Gamma_0$ and a subgraph of $\Gamma$ that extends to an orientation-preserving homeomorphism of $\widehat{\C}$).

We also define
$$
\mathrm{Emb}_a(\Gamma_0):=\{(\Gamma, \psi): \Gamma\geq_a \Gamma_0 \text{ and } \psi: \Gamma_0 \longrightarrow \Gamma \text{ is an embedding as abstract graphs}\},
$$
and
$$
\mathrm{Emb}_p(\Gamma_0):=\{(\Gamma, \psi): \Gamma\geq \Gamma_0 \text{ and } \psi: \Gamma_0 \longrightarrow \Gamma \text{ is an embedding as plane graphs}\}.
$$
\end{definition}

\begin{figure}[ht!]
\begin{tikzpicture}[thick]
   \node at (1,1) [circle,fill=black,inner sep=3pt] {};
      \node at (3,1) [circle,fill=black,inner sep=3pt] {};
         \node at (1,4) [circle,fill=black,inner sep=3pt] {};
   \node at (3,4) [circle,fill=black,inner sep=3pt] {};
      \node at (4,2.5) [circle,fill=black,inner sep=3pt] {};
         \node at (0,2.5) [circle,fill=black,inner sep=3pt] {};
         \draw[-,line width=1pt] (1,1)-|(3,4);
   \draw[-,line width=1pt] (1,4)-|(3,4);
  \draw[-,line width=1pt] (0,2.5)->(1,1);
  \draw[-,line width=1pt] (0,2.5)->(1,4);
  \draw[-,line width=1pt] (4,2.5)->(3,1);
  \draw[-,line width=1pt] (4,2.5)->(3,4);
  \draw[-,line width=1pt] (1,1)->(3,4);  
  \draw [-] (1,1) to [out=225,in=270] (-1,2.5);
  \draw [-] (-1,2.5) to [out=90,in=135] (1,4);
        
         \node at (5.6,3.2) {\begin{huge}$\geq_a$\end{huge}};
    \node at (5.5,1.8) {\begin{huge}$\ngeq$\end{huge}};
             
   \node at (8,1) [circle,fill=black,inner sep=3pt] {};
      \node at (10,1) [circle,fill=black,inner sep=3pt] {};
         \node at (8,4) [circle,fill=black,inner sep=3pt] {};
   \node at (10,4) [circle,fill=black,inner sep=3pt] {};
      \node at (11,2.5) [circle,fill=black,inner sep=3pt] {};
         \node at (7,2.5) [circle,fill=black,inner sep=3pt] {};
         \draw[-,line width=1pt] (8,1)-|(10,1);
  \draw[-,line width=1pt] (10,1)-|(10,4);
    \draw[-,line width=1pt] (8,1)-|(8,4);
  \draw[-,line width=1pt] (8,1)-|(10,1);
  \draw[-,line width=1pt] (8,4)-|(10,4);
  \draw[-,line width=1pt] (7,2.5)->(8,1);
  \draw[-,line width=1pt] (7,2.5)->(8,4);
  \draw[-,line width=1pt] (11,2.5)->(10,1);
  \draw[-,line width=1pt] (11,2.5)->(10,4);
   \end{tikzpicture}
\caption{The graph on the left abstractly dominates the graph on the right, but no embedding of the right graph into the left graph respects the plane structure.}
\label{abstract_domination_fig}
\end{figure}
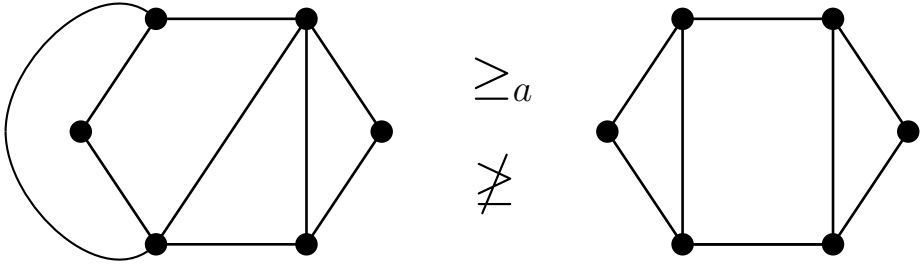

In other words, a simple plane graph $\Gamma$ abstractly dominates $\Gamma_0$ if $\Gamma$ (as an abstract graph) can be constructed from $\Gamma_0$ by introducing new edges; and it dominates $\Gamma_0$ if one can do this while respecting the plane structure.

We emphasize that the graph $\Gamma$ is always assumed to be plane, but the embedding $\psi$ may not respect the plane structure in $\mathrm{Emb}_a(\Gamma_0)$ (See Figure \ref{abstract_domination_fig}).
We also remark that an element $(\Gamma, \psi)$ is identified with $(\Gamma', \psi')$ in $\mathrm{Emb}_a(\Gamma_0)$ or $\mathrm{Emb}_p(\Gamma_0)$ if $\psi'\circ\psi^{-1}$ extends to an isomorphism between $\Gamma$ and $\Gamma'$ as plane graphs.

We have the following lemma.
\begin{lemma}\label{lem:kissingVsAH}
Let $G$ be a discrete subgroup of $\textrm{Aut}^\pm(\widehat{\C})$ and let
$\xi:\mathbf{G}_{\Gamma_0}\longrightarrow G$ be a weakly type preserving isomorphism.
Then the simple plane graph $\Gamma$ associated with $G$ abstractly dominates $\Gamma_0$.

Conversely, if $\Gamma$ is a simple plane graph abstractly dominating $\Gamma_0$ and $G$ is a kissing reflection group associated with $\Gamma$, then there exists a weakly type preserving isomorphism $\xi:\mathbf{G}_{\Gamma_0}\longrightarrow G$.
\end{lemma}
\begin{proof}
Let $G= \xi(\mathbf{G}_{\Gamma_0})$ be a discrete faithful weakly type preserving representation, 
then the reflections $\xi(\rho_i)$ along $C_i$ generate $G$.
If there were a non-tangential intersection between some $C_i$ and $C_j$, it would introduce a new relation between $\xi(\rho_i)$ and $\xi(\rho_j)$ by discreteness of $G$ (cf. \cite[Part~II, Chapter~5, \S 1.1]{VS93}), which would contradict the assumption that $\xi$ is an isomorphism.
Similarly, if there were circles $C_i, C_j, C_k$ touching at a point, then discreteness would give a new relation among $\xi(\rho_i)$, $\xi(\rho_j)$ and $\xi(\rho_k)$, which would again lead to a contradiction.
The above observations combined with the discussion preceding Definition~\ref{abstract_domination_def} imply that $\{C_1,..., C_n\}$ is a circle packing with associated contact graph $\Gamma$ abstractly dominating $\Gamma_0$.

Conversely, assume that $G$ is a kissing reflection group associated with a simple plane graph $\Gamma$ abstractly dominating $\Gamma_0$.
Let $C_i$ be the circle corresponding to $\mathbf{C}_i$ under a particular embedding of $\Gamma_0$ into $\Gamma$ (as abstract graphs), and $g_i$ be the reflection along $C_i$.
Defining $\xi:\mathbf{G}_{\Gamma_0}\longrightarrow G$ by $\xi(\rho_i) = g_i$, it is easy to check that $\xi$ is a weakly type preserving isomorphism.
\end{proof}

Note that for kissing reflection groups, the double of the polyhedron bounded by the half-planes associated to the circles $C_i$ is a fundamental polyhedron for the action of the index two Kleinian group on $\mathbb{H}^3$.
Thus, we have the following corollary which is worth mentioning.
\begin{cor}
Every group in $\textrm{AH}(\Gamma_0)$ is geometrically finite.
\end{cor}

Let $(\Gamma, \psi) \in \mathrm{Emb}_p(\Gamma_0)$ (respectively, in $\mathrm{Emb}_a(\Gamma_0)$).
Let us fix a circle packing $C_{1,\Gamma}, ..., C_{n, \Gamma}$ realizing $\Gamma$, where $C_{i,\Gamma}$ corresponds to $\mathbf{C}_i$ under the embedding $\psi$ of $\Gamma_0$.
Let $g_i$ be the associated reflection along $C_{i, \Gamma}$, and $\mathbf{G}_\Gamma = \langle g_1,..., g_{n}\rangle$.
Then Lemma \ref{lem:kissingVsAH} shows that 
\begin{align*}
\xi_{(\Gamma,\psi)}: \mathbf{G}_{\Gamma_0} &\longrightarrow \mathbf{G}_\Gamma\\
\rho_i &\mapsto g_i
\end{align*}
is a weakly type preserving isomorphism.
Thus, $\mathcal{QC}(\Gamma)= \mathcal{QC}(\mathbf{G}_\Gamma)$ can be embedded in $\textrm{AH}(\Gamma_0)= \textrm{AH}(\mathbf{G}_{\Gamma_0})$.
Indeed, if $\xi: \mathbf{G}_\Gamma \longrightarrow G$ represents an element in $\mathcal{QC}(\Gamma)$, then 
$$
\xi\circ \xi_{(\Gamma,\psi)}: \mathbf{G}_{\Gamma_0} \longrightarrow G
$$
is a weakly type preserving isomorphism.
It can be checked that the map $\xi \mapsto \xi\circ \xi_{(\Gamma,\psi)}$ gives an embedding of $\mathcal{QC}(\Gamma)$ into $\textrm{AH}(\Gamma_0)$.
We shall identify the space $\mathcal{QC}(\Gamma)$ with its image in $\textrm{AH}(\Gamma_0)$ under this embedding. 

\begin{prop}\label{thm:closure}
$$
\textrm{AH}(\Gamma_0) =  \bigcup_{(\Gamma, \psi) \in \mathrm{Emb}_a(\Gamma_0)} \mathcal{QC}(\Gamma),
$$
and
$$
\overline{\mathcal{QC}(\Gamma_0)} = \bigcup_{(\Gamma, \psi) \in \mathrm{Emb}_p(\Gamma_0)} \mathcal{QC}(\Gamma).
$$
\end{prop}
The proof of this proposition will be furnished after a discussion of pinching deformations.
Once the connection with pinching deformation is established, the result can be derived from \cite{O98}.

\subsection*{The perspective of pinching deformation.}
For the general discussion of pinching deformations, let us fix an arbitrary $2$-connected simple plane graph $\Gamma$, a circle packing $\mathcal{P}$ with contact graph isomorphic to $\Gamma$, and the associated kissing reflection group $\mathbf{G}= G_\mathcal{P}$, which we think of as the base point of the quasiconformal deformation space $\mathcal{QC}(\Gamma)$. Since $\Gamma$ is the only graph appearing in this discussion (until the proof Proposition~\ref{thm:closure}), we omit the subscript `$\Gamma$' from the base point $\mathbf{G}$.

Let $F$ be a face of $\Gamma$, and $R_F$ be the associated component of $\partial\mathcal{M}(\mathbf{G})$.
Let $C\in \mathcal{P}$ be a circle on the boundary of $\Pi_F$.
Then the reflection along $C$ descends to an anti-conformal involution on $R_F$, which we shall denote as
$$
\sigma_F: R_F \longrightarrow R_F.
$$
Note that different choices of the circle descend to the same involution.
It is known that for boundary incompressible geometrically finite Kleinian groups, the quasiconformal deformation space is the product of the Teichm\"uller spaces of the components of the conformal boundary (see \cite[Theorem~5.1.3]{Marden16}):
$$
\mathcal{QC}(\widetilde{\mathbf{G}}) = \prod_{F \text{ face of } \Gamma} \textrm{Teich} (R_F),
$$
where $\widetilde{\mathbf{G}}$ is the index two subgroup of $\mathbf{G}$ consisting of orientation-preserving elements.

We denote by $\textrm{Teich}^{\sigma_F} (R_F) \subseteq \textrm{Teich} (R_F)$ those elements corresponding to $\sigma_F$-invariant quasiconformal deformations of $R_F$.
Then
$$
\mathcal{QC}(\mathbf{G}) = \prod_{F \text{ face of } \Gamma} \textrm{Teich}^{\sigma_F} (R_F).
$$
Indeed, any element in $\prod_{F \text{ face of } \Gamma} \textrm{Teich}^{\sigma_F} (R_F)$ is uniquely determined by the associated Beltrami differential on $\Pi_F$, which can be pulled back by $\mathbf{G}$ to produce a $\mathbf{G}$-invariant Beltrami differential on $\widehat\C$.
Such a Beltrami differential can be uniformized by the Measurable Riemann mapping Theorem.

Recall that $\Pi_F$ is an ideal polygon.
Note that the component of $\Omega(\mathbf{G})$ containing $\Pi_F$ is simply connected and hence $\Pi_F$ inherits the hyperbolic metric from the corresponding component of $\Omega(\mathbf{G})$.
Then $R_F$ is simply the double of $\Pi_F$.
We claim that the only $\sigma_F$-invariant geodesics are those simple closed curves $\widetilde{\gamma}^F_{vw}$ on $R_F$ (see the discussion following Proposition~\ref{prop:3connected} for the definition of $\widetilde{\gamma}^F_{vw}$), where $v,w$ are two non-adjacent vertices on the boundary of $F$.
Indeed, any $\sigma_F$-invariant geodesic would intersect the boundary of the ideal polygon perpendicularly, and the $\widetilde{\gamma}^F_{vw}$ are the only geodesics satisfying this property.
We denote the associated geodesic arc in $\Pi_F$ by $\gamma^F_{vw}$.

We define a {\em multicurve} on a surface as a disjoint union of simple closed curves, such that no two components are homotopic.
It is said to be {\em weighted} if a non-negative number is assigned to each component.
We shall identify two multicurves if they are homotopic to each other.

Let $\mathcal{T}$ be a triangulation of $F$ obtained by adding new edges connecting the vertices of $F$.
Since each additional edge in this triangulation connects two non-adjacent vertices of $F$, we can associate a multicurve $\widetilde{\alpha}^F_\mathcal{T}$ on $R_F$ consisting of all $\widetilde{\gamma}^F_{vw}$, where $vw$ is a new edge in $\mathcal{T}$.
A marking of the graph $\Gamma$ also gives a marking of the multicurve.
We use $\alpha^F_\mathcal{T}$ to denote the multi-arc in the hyperbolic ideal polygon $\Pi_F$.
Since $\mathcal{T}$ is a triangulation, the multicurve $\widetilde{\alpha}^F_\mathcal{T}$ yields a pants decomposition of the punctured sphere $R_F$ such that the boundary components of every pair of pants are punctures or simple closed curves in $\widetilde{\alpha}^F_\mathcal{T}$. Thus, the complex structures of these pairs of pants are uniquely determined by the lengths of the 
marked multicurve $\widetilde{\alpha}^F_\mathcal{T}$ (cf. \cite[Theorem~3.5]{TY}). Since $R_F$ is invariant under the anti-conformal involution $\sigma_F$, the surface $R_F$ is uniquely determined by the complex structures of the above pairs of pants, and hence by the lengths of the marked 
multicurve $\widetilde{\alpha}^F_\mathcal{T}$.
If $F$ has $m$ sides, then any triangulation of $F$ has $m-3$ additional edges.
Thus, we have the analogue of Fenchel-Nielsen coordinates on
$$
 \textrm{Teich}^{\sigma_F} (R_F) = \R_+^{m-3}
$$
by assigning to each element of $\textrm{Teich}^{\sigma_F} (R_F)$ the lengths of the marked multicurve $\widetilde{\alpha}^F_\mathcal{T}$ (or equivalently the lengths of the marked multi-arc $\alpha^F_\mathcal{T}$).
Note that with these lengths, the multi-arcs and the multicurves become weighted.
Note that different triangulations yield different coordinates. We also remark that unlike classical Fenchel-Nielsen coordinates on Teichm{\"u}ller spaces, Fenchel-Nielsen coordinates on $\textrm{Teich}^{\sigma_F} (R_F)$ do not contain \emph{twist parameters} precisely due to the $\sigma_F-$invariance of quasiconformal deformations (cf. \cite[\S 3.2]{TY}).

In order to degenerate in $\textrm{Teich}^{\sigma_F} (R_F)$, the length of some arc $\gamma^F_{vw}$ must shrink to $0$.
The above discussion implies that the closure of $\textrm{Teich}^{\sigma_F} (R_F)$ in the Thurston compactification consists only of weighted $\sigma_F$-invariant multicurves $\widetilde{\alpha}^F$.
If $S_k \in \textrm{Teich}^{\sigma_F} (R_F)$ converges to a weighted $\sigma_F$-invariant multicurve $\widetilde{\alpha}^F$, then we say that $S_k$ is a pinching deformation on $\widetilde{\alpha}^F$.
More generally, we say that $G_k \in \mathcal{QC}(\mathbf{G})$ is a pinching deformation on $\widetilde{\alpha} = \bigcup\widetilde{\alpha}^F$ if for each face $F$, the conformal boundary associated to $F$ is a pinching deformation on $\widetilde{\alpha}^F$.

Given two curves $\widetilde{\gamma}^F_{vw} \subseteq R_F$ and $\widetilde{\gamma}^{F'}_{ut} \subseteq R_{F'}$, we say that they are {\em parallel} if they are homotopic in the $3$-manifold $\mathcal{M}(G_{\mathcal{P}})$.
Since the curve $\widetilde{\gamma}^F_{vw}$ corresponds to the element $g_vg_w$, we get that $\widetilde{\gamma}^F_{vw}$ and $\widetilde{\gamma}^{F'}_{ut}$ are parallel if and only if (possibly after switching the order) $v=u$ and $w=t$.
Note that in this case, $g_vg_w \in \text{stab}(\Omega_F) \cap \text{stab}(\Omega_{F'})$, and $v, w$ are on the common boundaries of $F$ and $F'$. (See, for example Figure~\ref{fig:njd1} where $v, w$ correspond to the circles $C, C''$, where the curve $\widetilde{\gamma}^{F_1}_{vw}$ is parallel with $\widetilde{\gamma}^{F_2}_{vw}$).

Let $\widetilde{\alpha} = \bigcup \widetilde{\alpha}^F$ be a union of $\sigma_F$-invariant multicurves.
We say that $\widetilde{\alpha}$ is a {\em non-parallel} multicurve if no two components are parallel.
We note that non-parallel multicurves $\widetilde{\alpha}$ for $\mathbf{G}$ are in one-to-one correspondence with the simple plane graphs that dominate $\Gamma$.
Indeed, any multicurve $\widetilde{\alpha} = \bigcup \widetilde{\alpha}^F$ corresponds to a plane graph that dominates $\Gamma$.
The non-parallel condition is equivalent to the condition that the graph is simple.

We now apply the general discussion carried out above to prove Proposition \ref{thm:closure}.

\begin{proof}[Proof of Proposition \ref{thm:closure}]
The first equality follows from Lemma \ref{lem:kissingVsAH}.

We first show $ \bigcup_{(\Gamma, \psi) \in \mathrm{Emb}_p(\Gamma_0)} \mathcal{QC}(\Gamma) \subseteq \overline{\mathcal{QC}(\Gamma_0)}$.
Let $\Gamma$ be a simple plane graph that dominates $\Gamma_0$ with embedding $\psi:\Gamma_0 \longrightarrow \Gamma$ as plane graphs. 
It follows from the discussion on pinching deformations that there exists a non-parallel multicurve $\widetilde\alpha\subseteq\partial\mathcal{M}(\mathbf{G}_{\Gamma_0})$ associated to $\Gamma$.
We complete $\Gamma$ to a triangulation $\mathcal{T}$ by adding edges.
As before, $\mathcal{T}$ gives a $\sigma$-invariant multicurve $\widetilde{\beta}\subseteq\partial\mathcal{M}(\mathbf{G}_{\Gamma_0})$ which contains $\widetilde{\alpha}$.
We set $\widetilde{\beta} = \widetilde{\alpha} \sqcup \widetilde{\alpha}'$.
Then the lengths of the multicurve $\widetilde{\alpha}'$ gives a parameterization of $\mathcal{QC}(\Gamma)$, and the lengths of the multicurve $\widetilde{\beta}$ gives a parameterization of $\mathcal{QC}(\Gamma_0)$.
Now given any element $\vec{l}=(l_\gamma: \gamma \in \widetilde{\alpha}') \in \mathcal{QC}(\Gamma)$, we show that it can be realized as the algebraic limit of a sequence in $\mathcal{QC}(\Gamma_0)$.

Such a construction is standard, and follows directly from \cite[Theorem 5.1]{O98} (see also \cite{Maskit83}).
For completeness, we sketch the proof here.
Let $G \in \mathcal{QC}(\Gamma_0)$ be so that the length of the multicurve $\widetilde{\alpha}'$ is $\vec{l}$.
We assume that $\widetilde{\beta}$ is realized by hyperbolic geodesics in $\partial\mathcal{M}(G)$.
Let $A \subseteq \partial\mathcal{M}(G)$ be an $\epsilon$-neighborhood of $\widetilde{\alpha}$.
We choose $\epsilon$ small enough so that $A$ is disjoint from $\widetilde{\alpha}'$, and each component contains only one component of $\widetilde{\alpha}$.
We construct a sequence of quasiconformal deformations supported on $A$ so that the modulus of each annulus in $A$ tends to infinity while $\widetilde{\alpha}$ remains as core curves.
Since the multicurve $\widetilde{\alpha}$ is non-parallel, Thurston's Hyperbolization Theorem guarantees a convergent subsequence (see \cite{O98} for more details).
This algebraic limit is a kissing reflection group. 
By construction, the contact graph of such a limiting kissing reflection group is $\Gamma$, and the lengths of $\widetilde{\alpha}'$ is $\vec{l}$.

Conversely, if a sequence of quasiconformal deformations $\xi_n:\mathbf{G}_{\Gamma_0}\longrightarrow G_n$ of $\mathbf{G}_{\Gamma_0}$ converges to $\xi_\infty:\mathbf{G}_{\Gamma_0} \longrightarrow G$ algebraically, then $G$ is a kissing reflection group.
Since $\{\xi_n(g_i)\}$ converges, where $g_i$ are the standard generators, the contact graph for $G$ dominates $\Gamma_0$. The embedding $\psi: \Gamma_0 \longrightarrow \Gamma$ (respecting the plane structure) comes from the identification of the generators.
\end{proof}

\subsection{Quasi-Fuchsian space and mating locus.}\label{sec:qfs}
The above discussion applies to the special case of quasiFuchsian space.
This space is related to the mating locus for critically fixed anti-rational maps.

Let $\Gamma_d$ be the marked $d+1$-sided polygonal graph, i.e., $\Gamma_d$ contains $d+1$ vertices $v_1,..., v_{d+1}$ with edges $v_iv_{i+1}$ where indices are understood modulo $d+1$.
We choose the most symmetric circle packing $\mathcal{P}_d$ realizing $\Gamma_d$.
More precisely, consider the ideal $(d+1)$-gon in $\D\cong \Hyp^2$ with vertices at the $(d+1)$-st roots of unity.
The edges of this ideal $(d+1)$-gon are arcs of $d+1$ circles, and we label these circles as $\mathcal{P}_d := \{\mathbf{C}_1,..., \mathbf{C}_{d+1}\}$, where we index them counter-clockwise such that $\mathbf{C}_1$ passes through $1$ and $e^{2\pi i/(d+1)}$.
Let $\mathbf{G}_d$ be the kissing reflection group associated to $\mathcal{P}_d$.
Note that $\widetilde{\mathbf{G}}_d$ is a Fuchsian group. We remark that since any embedding of the polygonal graph $\Gamma_d$ into a graph $\Gamma$ abstractly dominating $\Gamma_d$ respects the plane structure, we have that $\textrm{AH}(\Gamma_d)=\overline{\mathcal{QC}(\Gamma_d)}$.

Recall that a graph $\Gamma$ is said to be Hamiltonian if there exists a cycle which passes through each vertex exactly once.
Since a simple plane graph with $d+1$ vertices is Hamiltonian if and only if it dominates $\Gamma_d$, the following proposition follows immediately from Proposition \ref{thm:closure}.

\begin{prop}\label{prop:qb}
Let $\Gamma$ be a simple plane graph with $d+1$ vertices, then any kissing reflection group $G$ with contact graph $\Gamma$ is in the closure $\overline{\mathcal{QC}(\Gamma_d)}$ if and only if $\Gamma$ is Hamiltonian.
\end{prop}

Let $\Gamma$ be a marked Hamiltonian simple plane graph with a Hamiltonian cycle $C = (v_1,...,v_{d+1})$.
Let $G:=\langle g_1,..., g_{d+1}\rangle$ be a kissing reflection group with contact graph $\Gamma$.
The Hamiltonian cycle divides the fundamental domain $\Pi \subseteq \widehat\C$ into two parts, and we denote them as $\Pi^+$ and $\Pi^-$, where we assume that the Hamiltonian cycle is positively oriented on the boundary of $\Pi^+$ and negatively oriented on the boundary of $\Pi^-$ (see Figure~\ref{fig:njd1}).
We denote
$$
\Omega^\pm := \bigcup_{g\in G} g\cdot \Pi^\pm.
$$
Since each $\Omega^\pm$ is $G$-invariant, we have $\Lambda(G) = \partial \Omega^+ = \partial\Omega^-$.

As in \S \ref{limit_connected_subsec}, set $$\Pi^{1,\pm} = \bigcup_{i=1}^{d+1} g_i \cdot \Pi^\pm,\ \textrm{and}\quad \Pi^{j+1,\pm} = \bigcup_{i=1}^{d+1} g_i \cdot \left(\Pi^{j,\pm} \setminus \overline{D}_i\right).$$
For consistency, we also set $\Pi^{0, \pm} = \Pi^\pm$.
Then the arguments of the proof of Lemma \ref{lem:genl} show that
$$
\Pi^{l,\pm} = \bigcup_{\vert g\vert = l} g\cdot \Pi^\pm.
$$

\begin{lemma}\label{lem:HC}
The closures $\overline{\Omega^+}$ and $\overline{\Omega^-}$ are connected.
\end{lemma}
\begin{proof}
We shall prove by induction that the closures $\overline{\bigcup_{i=0}^n \Pi^{i,+}}$ and $\overline{\bigcup_{i=0}^n \Pi^{i,+}\setminus \overline{D}_j}$ are connected for all $n$ and $j$.

Indeed, the base case is true as $\overline{\Pi^+}$ and $\overline{\Pi^+ \setminus \overline{D}_j}$ are connected for all $j$.
Assume that $\overline{\bigcup_{i=0}^n \Pi^{i,+}}$ and $\overline{\bigcup_{i=0}^n \Pi^{i,+}\setminus \overline{D}_j}$ connected for all $j$. 
Note that
$$
\overline{\bigcup_{i=0}^{n+1} \Pi^{i,+}} = \overline{\Pi^{+} \cup \bigcup_{j=1}^{d+1} g_j \cdot \left(\bigcup_{i=0}^n \Pi^{i,+} \setminus \overline{D}_j\right)}.
$$
By induction hypothesis, $\overline{g_j \cdot (\bigcup_{i=0}^n \Pi^{i,+} \setminus \overline{D}_j)}$ is connected.
Since each 
$$
\overline{g_j \cdot \left(\bigcup_{i=0}^n \Pi^{i,+} \setminus \overline{D}_j\right)}
$$ 
intersects $\overline{\Pi^{+}}$ along an arc of $C_j = \partial D_j$, we have that $\overline{\bigcup_{i=0}^{n+1} \Pi^{i,+}}$ is also connected.

Similarly, 
$$
\overline{\bigcup_{i=0}^{n+1} \Pi^{i,+}\setminus \overline{D}_j} = \overline{\Pi^+ \cup \bigcup_{k\neq j} g_k \cdot \left(\bigcup_{i=0}^n \Pi^{i,+} \setminus \overline{D}_k\right)}
$$ 
is connected for any $j$.

Connectedness of $\overline{\Omega^-}$ is proved in the same way.
\end{proof}

\subsection*{Matings of function kissing reflection groups.}
We say that a kissing reflection group $G$ with connected limit set is a {\em function kissing reflection group} if there is a component $\Omega_0$ of $\Omega(G)$ invariant under $G$.
This terminology was traditionally used in complex analysis as one can construct Poincar\'e series, differentials and functions on it.

We say that a simple plane graph $\Gamma$ with $n$ vertices is {\em outerplanar} if it has a face with all $n$ vertices on its boundary.
We also call this face the {\em outer face}.
Note that a $2$-connected outerplanar graph is Hamiltonian with a unique Hamiltonian cycle. The following proposition characterizes function kissing reflection groups in terms of their contact graphs.
\begin{prop}\label{prop:fkrg}
A kissing reflection group $G$ is a function group if and only if its contact graph $\Gamma$ is 2-connected and outerplanar.
\end{prop}
\begin{proof}
If $\Gamma$ is outerplanar, then let $\Omega_F$ be the component of $\Omega(G)$ associated to the outer face $F$. It is easy to see that the standard generating set fixes $\Omega_F$, so $\Omega_F$ is invariant under $G$.

Conversely, assume that $\Gamma$ is not outerplanar. Since $\Pi$ is a fundamental domain of the action of $G$ on $\Omega(G)$, any $G$-invariant component of $\Omega(G)$ must correspond to some face of $\Gamma$. 
Let $F$ be a face of $\Gamma$.
Since $\Gamma$ is not outerplanar, there exists a vertex $v$ which is not on the boundary of $F$. Thus, $g_v \cdot \Pi_F$ is not in $\Omega_F$, and hence $\Omega_F$ is not invariant under $G$.
\end{proof}

  \begin{figure}[h!]
  \centering
   \includegraphics[width=.54\linewidth]{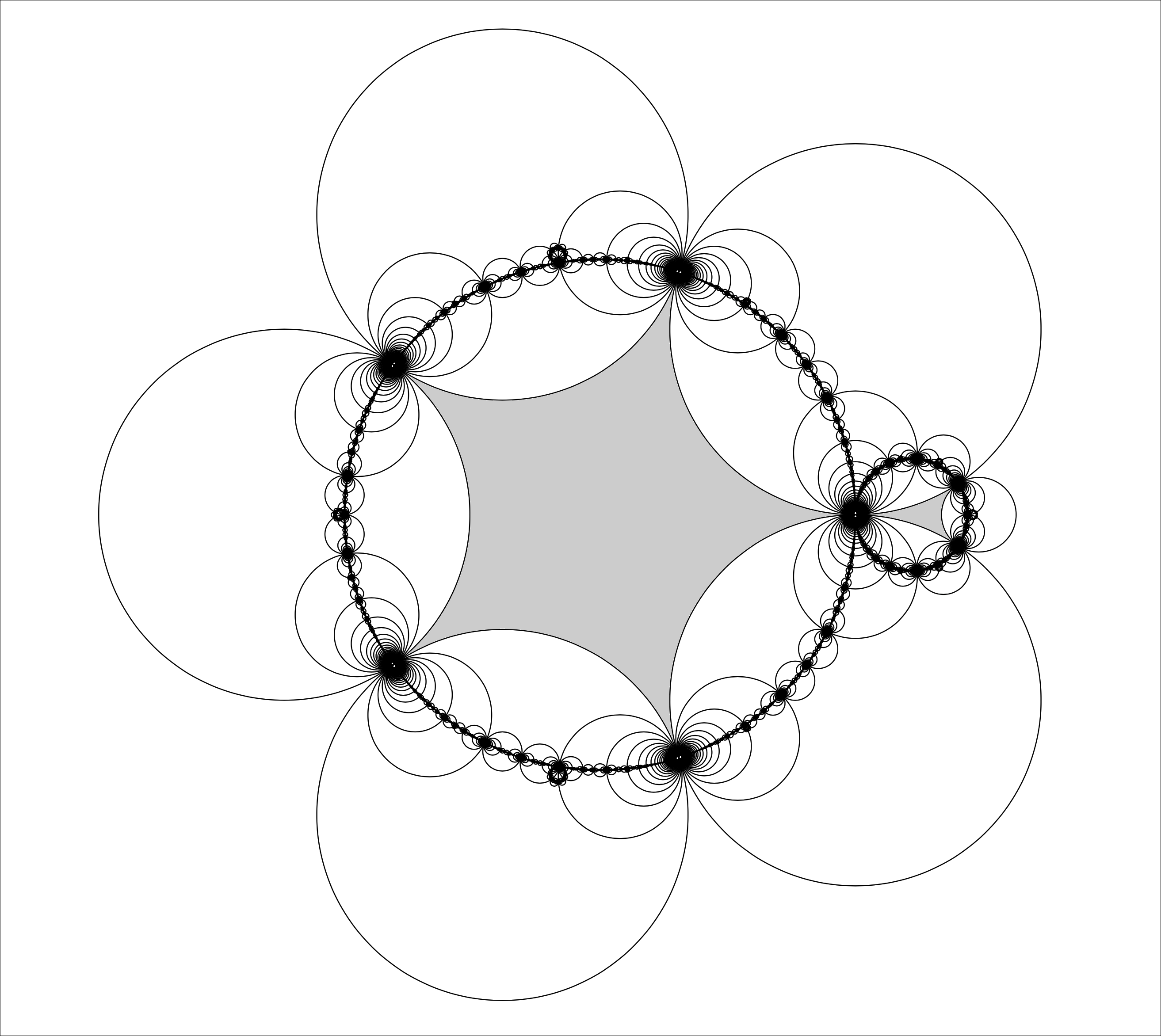}  
  \caption{The limit set of a function kissing reflection group $H$ with an outerplanar contact graph. The grey region represents the part $\Pi_b$ of the fundamental domain $\Pi(H)$ corresponding to the non-outer faces of $\Gamma(H)$. The kissing reflection group $G$ shown in Figure~\ref{fig:njd1} can be constructed by pinching a simple closed curve for $H$.}
\label{fig:njd}
\end{figure}

For a function kissing reflection group $G$, we set $\Omega_b := \Omega(G)\setminus \Omega_0$, where $\Omega_0$ is $G$-invariant.
Similarly, we shall use the notation $\Pi_b := \Pi \cap \Omega_b$ (see Figure~\ref{fig:njd}), and $\Pi^i_b := \Pi^i \cap \Omega_b$, for $i\geq0$ (see Lemma~\ref{lem:genl}).
We call the closure $\overline{\Omega_b} = \widehat\C \setminus \Omega_0$ the {\em filled limit set} for the function kissing reflection group $G$, and denote it by $\mathcal{K}(G)$.

Let $G^\pm$ be two function kissing reflection groups with the same number of vertices in their contact graphs.
We say that a kissing reflection group $G$ is a {\em geometric mating} of $G^+$ and $G^-$ if we have 
\begin{itemize}
\item a decomposition $\Omega(G) = \Omega^+(G) \sqcup \Omega^-(G)$ with $\Lambda(G) = \partial \Omega^+(G) = \partial \Omega^-(G)$;
\item weakly type preserving isomorphisms $\phi^\pm : G^\pm \longrightarrow G$;
\item continuous surjections $\psi^{\pm}: \mathcal{K}(G^\pm) \longrightarrow \overline{\Omega^\pm(G)}$ which are conformal between the interior $\mathring{\mathcal{K}}(G^\pm)$ and $\Omega^\pm(G)$
such that for any $g\in G^\pm$, $\psi^\pm \circ g|_{\mathcal{K}(G^\pm)} = \phi^\pm(g) \circ \psi^\pm$.
\end{itemize}

Note that by the semi-conjugacy relation, the sets $\Omega^\pm(G)$ are $G-$invariant.
We shall now complete the proof of the group part of Theorem~\ref{thm:gm}.

\begin{prop}\label{prop:mg}
A kissing reflection group $G$ is a geometric mating of two function kissing reflection groups if and only if the contact graph $\Gamma(G)$ of $G$ is Hamiltonian.
\end{prop}
\begin{proof}
If $G$ is a geometric mating of $G^\pm$, then we have a decomposition $\Omega(G) =  \Omega^+(G) \sqcup \Omega^-(G)$ with $\Lambda(G) = \partial \Omega^+(G) = \partial \Omega^-(G)$.
This gives a decomposition of $\Pi(G) = \Pi^+ \sqcup \Pi^-$, and thus a decomposition of the contact graph $\Gamma(G) = \Gamma^+ \cup \Gamma^-$. More precisely, the vertex sets of $\Gamma^\pm$ coincide with that of $\Gamma(G)$, and there is an edge in $\Gamma^\pm$ connecting two vertices if and only if the point of intersection of the corresponding two circles lies in $\partial\Pi^\pm$.
Since the action of $G^\pm$ on $\mathcal{K}(G^\pm)$ is semi-conjugate to $G$ on $\overline{\Omega^\pm(G)}$, it follows that $\Gamma^\pm$ is isomorphic to $\Gamma(G^\pm)$ as plane graphs.
Thus, the intersection of $\Gamma^+$ and $\Gamma^-$ (which is the common boundary of the outer faces of $\Gamma^+$ and $\Gamma^-$) gives a Hamiltonian cycle for $\Gamma(G)$.

Conversely, if $\Gamma(G)$ is Hamiltonian, then a Hamiltonian cycle yields a decomposition $\Gamma(G) = \Gamma^+ \cup \Gamma^-$, where $\Gamma^+$ and $\Gamma^-$ are outerplanar graphs with vertex sets equal to that of $\Gamma(G)$.
We construct function kissing reflection groups $G^\pm$ with contact graphs $\Gamma(G^\pm) = \Gamma^\pm$.
Note that we have a natural embedding of $\Gamma(G^\pm)$ into $\Gamma(G)$ (as plane graphs), which gives an identification of vertices and non-outer faces.
Thus, we have weakly type preserving isomorphisms $\phi^\pm: G^\pm \longrightarrow G$ coming from the identification of vertices and Proposition \ref{thm:closure}.
By quasiconformal deformation, we may assume that for each non-outer face $F\in \Gamma(G^\pm)$, the associated conformal boundary component $R_F(G^\pm)\subseteq\partial\mathcal{M}(G^\pm)$ is conformally equivalent to the corresponding conformal boundary component $R_F(G)\subseteq\partial\mathcal{M}(G)$.

The existence of the desired continuous map $\psi^+$ can be derived from a more general statement on the existence of Cannon-Thurston maps (see \cite[Theorem 4.2]{MS13}).
For completeness and making the proof self-contained, we give an explicit construction of $\psi^+: \mathcal{K}(G^+) \longrightarrow \overline{\Omega^+(G)}$.
This construction is done in levels: we start with a homeomorphism $$\psi^+_0: \overline{\Pi^{0}_b(G^+)} \longrightarrow \overline{\Pi^{0, +}}.$$
This can be chosen to be conformal on the interior as the associated conformal boundaries are assumed to be conformally equivalent.
Assuming that $$\psi^+_i: \overline{\bigcup_{j=0}^i \Pi^{j}_b(G^+)} \longrightarrow \overline{\bigcup_{j=0}^i \Pi^{j,+}}$$ is constructed, we extend $\psi^+_i$ to $\psi^+_{i+1}$ by setting
$$
\psi^+_{i+1} (z) := \phi^+(g_k)\circ \psi^+_i \circ g_k(z)
$$
if $z\in \Pi^{i+1}_b(G^+) \cap D_k$.
It is easy to check by induction that $\psi^+_{i+1}$ is a homeomorphism which is conformal on the interior (see Figure \ref{fig:njd1} and Figure \ref{fig:njd}).

Let $P^{0}(G^+) := \overline{\Pi^{0}_b(G^+)}\setminus \Pi^{0}_b(G^+)$, then $P^{0}(G^+)$ consists of cusps where various components of $\Pi^{0}_b(G^+)$ touch.
Let $P^{\infty}(G^+):= \bigcup_{g\in G^+} g\cdot P^{0}(G^+)$.
Then, $P^{\infty}(G^+)$ is dense in $\Lambda(G^+)$.
We define $P^{\infty, +}$ similarly (where, $P^{0,+}:=\overline{\Pi^{0,+}}\setminus\Pi^{0,+}$), and note that $P^{\infty, +}$ is dense in $\partial \Omega^+(G) = \Lambda(G)$.

Note that $\psi^+_i = \psi^+_j$ on $P^{i}(G^+)$ for all $j \geq i$. Thus, we have a well defined limit $\psi^+ : P^{\infty}(G^+) \longrightarrow P^{\infty, +}$.
We claim that $\psi^+$ is uniformly continuous on $P^{\infty}(G^+)$.
For an arbitrary $\epsilon >0$, there exists $N$ such that all disks in $\mathcal{D}^N(G)$ have spherical diameter $< \epsilon$ by Lemma \ref{lem:maxr}.
Choose $\delta$ so that any two non-adjacent disks in $\mathcal{D}^N(G^+)$ are separated in spherical metric by $\delta$.
Then if $x,y \in P^{\infty}(G^+)$ with $d(x,y) < \delta$, they must lie in two adjacent disks of $\mathcal{D}^N(G^+)$.
Thus, $\psi^+_i(x)$ and $\psi^+_i(y)$ (whenever defined) lie in two adjacent disks of $\mathcal{D}^N(G)$ for all $i\geq N$.
Hence, $d(\psi^+_i(x), \psi^+_i(y)) < 2\epsilon$ for all $i \geq N$.
This shows that $\psi^+$ is uniformly continuous on $P^{\infty}(G^+)$.

Thus, we have a continuous extension $\psi^+:  \mathcal{K}(G^+) \longrightarrow \overline{\Omega^+(G)}$.
It is easy to check that $\psi^+$ is surjective, conformal on the interior, and equivariant with respect to the actions of $G^+$ and $G$.

The same proof gives the construction for $\psi^-$. This shows that $G$ is a geometric mating of $G^+$ and $G^-$.
\end{proof}

Note that from the proof of Proposition \ref{prop:mg}, we see that each (marked) Hamiltonian cycle $H$ gives a decomposition of the graph $\Gamma = \Gamma^+ \cup \Gamma^-$ and hence an unmating of a kissing reflection group. 
From the pinching perspective as in the proof of Proposition \ref{thm:closure}, this (marked) Hamiltonian cycle also gives a pair of non-parallel multicurves.
This pair of multicurves can be constructed explicitly from the decomposition of $\Gamma$.
Indeed, each edge in $\Gamma^\pm - H$ gives a pair of non-adjacent vertices in $H$.
The corresponding $\sigma$-invariant curve comes from these non-adjacent vertices in $H$.
If a marking is not specified on $H$, this pair of multicurves is defined only up to simultaneous change of coordinates on the two conformal boundaries.

\subsection{Nielsen maps for kissing reflection groups.}\label{br_cov_nielsen}
Let $\Gamma$ be a simple plane graph, and $G_\Gamma$ be a kissing reflection group with contact graph $\Gamma$.
Recall that $\overline{\mathcal{D}} = \bigcup_{j=1}^n \overline{D_j}$ is defined as the union of the closure of the disks for the associated circle packing.
We define the {\em Nielsen map} $\mathcal{N}_\Gamma: \overline{\mathcal{D}} \longrightarrow \widehat\C$ by
$$
\mathcal{N}_\Gamma(z) = g_j(z) \text{ if } z\in \overline{D_j}.
$$

Let us now assume that $\Gamma$ is $2$-connected, then the limit set $\Lambda(G_\Gamma)$ is connected.
Thus, each component of $\Omega(G_\Gamma)$ is a topological disk.
Let $F$ be a face with $d+1$ sides of $\Gamma$, and let $\Omega_F$ be the component of $\Omega(G)$ containing $\Pi_F$.
By a quasiconformal deformation, we can assume that the restriction of $G_\Gamma$ on $\Omega_F$ is conformally conjugate to the regular ideal $d+1$-gon reflection group on $\mathbb{D} \cong\Hyp^2$.

For the regular ideal $d+1$-gon reflection group $\mathbf{G}_d$ (whose associated contact graph is $\Gamma_d$), the $d+1$ disks yield a Markov partition for the action of the Nielsen map $\mathcal{N}_d\equiv \mathcal{N}_{\Gamma_d}$ on the limit set $\Lambda(\mathbf{G}_d)=\mathbb{S}^1$.
The diameters of the preimages of these disks under $\mathcal{N}_d$ shrink to $0$ uniformly by Lemma \ref{lem:maxr}, and hence the Nielsen map $\mathcal{N}_d$ is topologically conjugate to $z\mapsto \overline{z}^d$ on $\mathbb{S}^1$ (cf. \cite[\S 4]{LLMM19}).

Therefore, the Nielsen map $\mathcal{N}_\Gamma$ is topologically conjugate to $\overline{z}^d$ on the ideal boundary of $\Omega_F$.
This allows us to replace the dynamics of $\mathcal{N}_\Gamma$ on $\Omega_F$ by $\overline{z}^d$ for every face of $\Gamma$, and obtain a globally defined orientation reversing branched covering 
$\mathcal{G}_\Gamma: \widehat\C\longrightarrow \widehat\C$.

More precisely, let $\phi:\mathbb{D} \longrightarrow \Omega_F$ be a conformal conjugacy between $\mathcal{N}_d$ and $\mathcal{N}_\Gamma$, which extends to a topological conjugacy from the ideal boundary $\mathbb{S}^1 = I(\mathbb{D}) = \partial \mathbb{D}$ onto $I(\Omega_F)$.
Let $\psi: \overline{\mathbb{D}} \longrightarrow \overline{\mathbb{D}}$ be an arbitrary homeomorphic extension of the topological conjugacy between $\overline{z}^d\vert_{\mathbb{S}^1}$ and $\mathcal{N}_d\vert_{\mathbb{S}^1}$ fixing $1$.
We define 
\begin{align}\label{eqn:2}
\mathcal{G}_\Gamma :=\left\{\begin{array}{ll}
                     \mathcal{N}_\Gamma & \mbox{on}\ \overline{\mathcal{D}}\setminus \bigcup_{F} \Omega_F, \\
                     (\phi\circ\psi) \circ m_{-d} \circ (\phi\circ\psi)^{-1} & \mbox{on}\  \Omega_F,                                         \end{array}\right. 
\end{align}
where $m_{-d}(z) = \overline{z}^d$.

The map $\overline{z}^d$ has $d+1$ invariant rays in $\mathbb{D}$ connecting $0$ with $e^{2\pi i \cdot \frac{j}{d+1}}$ by radial line segments.
We shall refer to these rays as \emph{internal rays}.
For each $\Omega_F$, we let $\mathscr{T}_F$ be the image of the union of these $d+1$ internal rays (under $\phi\circ\psi$),
and we call the image of $0$ the {\em center} of $\Omega_F$.
This graph $\mathscr{T}_F$ is a deformation retract of $\Pi_F$ fixing the ideal points.
Thus, each ray of $\Omega_F$ lands exactly at a cusp where two circles of the circle packing touch.
We define 
\begin{align}\label{eqn:3}
\mathscr{T}(\mathcal{G}_\Gamma) := \bigcup_F \overline{\mathscr{T}_F},
\end{align}
and endow it with a simplicial structure such that the centers of the faces are vertices.
Then $\mathscr{T}(\mathcal{G}_\Gamma)$ is the planar dual to the plane graph $\Gamma$.
We will now see that $\mathcal{G}_\Gamma$ acts as a `reflection' on each face of $\mathscr{T}(\mathcal{G}_\Gamma)$.

\begin{lemma}\label{lem:uc}
Let $P$ be an open face of $\mathscr{T}(\mathcal{G}_\Gamma)$, then $\mathcal{G}_\Gamma$ is a homeomorphism sending $P$ to the interior of its complement.
\end{lemma}
\begin{proof}
Since $\Gamma$ is $2$-connected, the complement of $\overline{P}$ is connected, so $P$ is a Jordan domain.
Since $P$ contains no critical point, and $\mathcal{G}_\Gamma(\partial P) = \partial P$, it follows that $\mathcal{G}_\Gamma$ sends $P$ homeomorphically to the interior of its complement. 
\end{proof}

We shall see in the next section that this branched covering $\mathcal{G}_\Gamma$ is topologically conjugate to a critically fixed anti-rational map $R$, and this dual graph $\mathscr{T}(\mathcal{G}_\Gamma)$ is the \emph{Tischler graph} associated with $R$.

\section{Critically Fixed Anti-rational Maps}\label{sec:cf}
In this section, we shall show how critically fixed anti-rational maps are related to circle packings and kissing reflection groups. As a corollary of our study, we obtain a classification of critically fixed anti-rational maps in terms of the combinatorics of planar duals of Tischler graphs. A classification of these maps in terms of combinatorial properties of Tischler graphs was given independently by Geyer in a recent work \cite{Gey20}. For the purpose of establishing a dynamical correspondence between critically fixed anti-rational maps and kissing reflection groups, planar duals of Tischler graphs turn out to be a more natural combinatorial invariant.

The arguments employed in Section~\ref{comb_prop_subsec}, where we investigate the structure (of planar duals) of Tischler graphs, parallel those used in the proof of \cite[Theorem~4.3]{Gey20}. On the other hand, \cite[\S 6]{Gey20} uses purely topological means to construct Thurston maps for the realization part of the classification theorem; while we use the branched cover $\mathcal{G}_\Gamma$ of Section~\ref{br_cov_nielsen}, which is cooked up from Nielsen maps of kissing reflection groups (see Section~\ref{sec:to}).

We define an anti-polynomial $P$ of degree $d$ as
$$
P(z) = a_d \overline{z}^d + a_{d-1} \overline{z}^{d-1} + ... + a_0
$$
where $a_i\in \C$ and $a_d \neq 0$.
An anti-rational map $R$ of degree $d$ is the ratio of two anti-polynomials
$$
R(z) = \frac{P(z)}{Q(z)}
$$
where $P$ and $Q$ have no common zeroes and the maximum degree of $P$ and $Q$ is $d$.
An anti-rational map of degree $d$ is an orientation reversing branched covering of $\widehat\C$.
It is said to be {\em critically fixed} if all of its critical points are fixed.
The Julia set and Fatou set of anti-rational maps can be defined as in the rational setting.

\subsection{Tischler graph of critically fixed anti-rational maps.}\label{comb_prop_subsec}
Let $R$ be a critically fixed anti-rational map of degree $d$. Let $c_1,\cdots, c_k$ be the distinct critical points of $R$, and the local degree of $R$ at $c_i$ be $m_i$ ($i=1,\cdots, k$). Since $R$ has $(2d-2)$ critical points counting multiplicity, we have that 
$$
\displaystyle\sum_{i=1}^k (m_i-1)=2d-2\ \implies\ \displaystyle\sum_{i=1}^k m_i=2d+k-2.
$$

Suppose that $U_i$ is the invariant Fatou component containing $c_i$ ($i=1,\cdots, k$). Then, $U_i$ is a simply connected domain such that $R\vert_{U_i}$ is conformally conjugate to $\overline{z}^{m_i}\vert_\D$ \cite[Theorem~9.3]{Milnor06}. This defines internal rays in $U_i$, and $R$ maps the internal ray at angle $\theta\in\R/\Z$ to the one at angle $-m_i\theta\in\R/\Z$. It follows that there are $(m_i+1)$ fixed internal rays in $U_i$. A straightforward adaptation of \cite[Theorem~18.10]{Milnor06} now implies that all these fixed internal rays land at repelling fixed points on $\partial U_i$.

We define the \emph{Tischler graph} $\mathscr{T}$ of $R$ as the union of the closures of the fixed internal rays of $R$.

\begin{lemma}\label{fixed_points_count}
$R$ has exactly $(d+2k-1)$ distinct fixed points in $\widehat{\C}$ of which $(d+k-1)$ lie on the Julia set of $R$. 
\end{lemma}
\begin{proof}
Note that $R$ has no neutral fixed point and exactly $k$ attracting fixed points. The count of the total number of fixed points of $R$ now follows from the Lefschetz fixed point theorem (see \cite[Lemma~6.1]{LM14}).
\end{proof}

\begin{lemma}\label{pairwise_landing}
The fixed internal rays of $R$ land pairwise.
\end{lemma}
\begin{proof}
As $R$ is a local orientation reversing diffeomorphism in a neighborhood of the landing point of each internal ray, it follows that at most two distinct fixed internal rays may land at a common point.

Note that the total number of fixed internal rays of $R$ is 
$$
\displaystyle\sum_{i=1}^k (m_i+1)=2(d+k-1),
$$ 
while there are only $(d+k-1)$ landing points available for these rays by Lemma \ref{fixed_points_count}. Since no more than two fixed internal rays can land at a common fixed point, the result follows.
\end{proof}

In our setting, it is more natural to put a simplicial structure on $\mathscr{T}$ so that the vertices correspond to the critical points of $R$.
We will refer to the repelling fixed points on $\mathscr{T}$ as the {\em midpoints} of the edges.
To distinguish an edge of $\mathscr{T}$ from an arc connecting a vertex and a midpoint, we will call the latter an  {\em internal ray} of $\mathscr{T}$.

\begin{cor}\label{valence_restriction}
The valence of the critical point $c_i$ (as a vertex of $\mathscr{T}$) is $(m_i+1)$ ($i=1,\cdots, k$). The repelling fixed points of $R$ are in bijective correspondence with the edges of $\mathscr{T}$. 
\end{cor}

The proof of Lemma~\ref{pairwise_landing} implies the following (see \cite[Corollary~6]{H19} for the same statement in the holomorphic setting).
\begin{cor}\label{fixed_are_vertices}
Each fixed point of $R$ lies on (the closure of) a fixed internal ray.
\end{cor}

We are now ready to establish the key properties of the Tischler graph of a critically fixed anti-rational map that will be used in the combinatorial classification of such maps.

\begin{lemma}\label{lem:jd}
The faces of $\mathscr{T}$ are Jordan domains.
\end{lemma}
\begin{proof}
Let $F$ be a face of the Tischler graph $\mathscr{T}$.
Then a component of the ideal boundary $I(F)$ consists of a sequence of edges $e_1,..., e_m$ of $\mathscr{T}$, oriented counter clockwise viewed from $F$.
Note that a priori, $e_i$ may be equal to $e_j$ in $\mathscr{T}$ for different $i$ and $j$.

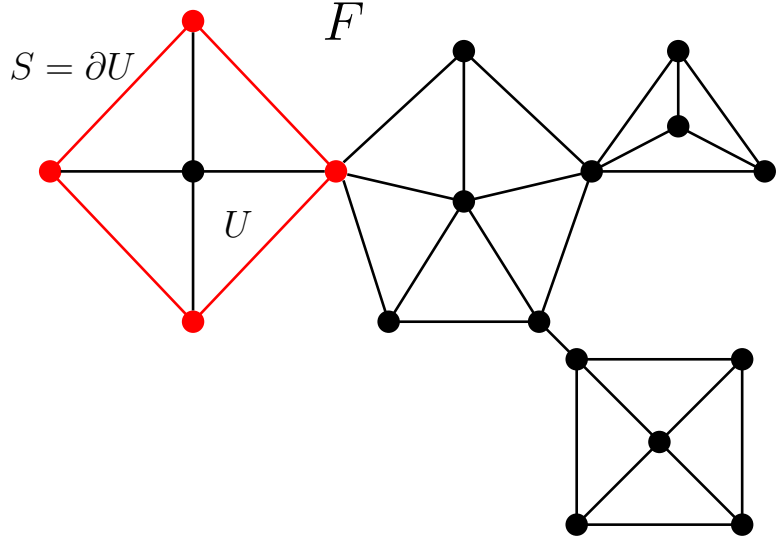
\begin{figure}[ht!]
\begin{tikzpicture}[thick]
  \node at (4.5,4) [circle,fill=red,inner sep=3pt] {};
  \node at (6.4,6) [circle,fill=red,inner sep=3pt] {};
     \node at (6.4,2) [circle,fill=red,inner sep=3pt] {};
  \node at (8.3,4) [circle,fill=red,inner sep=3pt] {};
  \node at (6.4,4) [circle,fill=black,inner sep=3pt] {};
  \node at (10,5.6) [circle,fill=black,inner sep=3pt] {}; 
  \node at (11.7,4) [circle,fill=black,inner sep=3pt] {}; 
   \node at (9,2) [circle,fill=black,inner sep=3pt] {};
    \node at (11,2) [circle,fill=black,inner sep=3pt] {};
    \node at (11.5,1.5) [circle,fill=black,inner sep=3pt] {};
    \node at (10,3.6) [circle,fill=black,inner sep=3pt] {};
        \node at (14,4) [circle,fill=black,inner sep=3pt] {};
         \node at (12.85,5.6) [circle,fill=black,inner sep=3pt] {};
         \node at (12.85,4.6) [circle,fill=black,inner sep=3pt] {};  
  \node at (13.7,1.5) [circle,fill=black,inner sep=3pt] {};
 \node at (13.7,-0.7) [circle,fill=black,inner sep=3pt] {};
  \node at (11.5,-0.7) [circle,fill=black,inner sep=3pt] {};
  \node at (12.6,0.4) [circle,fill=black,inner sep=3pt] {};
  \node at (2.5,0) {}; 
  
  \draw[-,line width=1pt,red] (4.5,4)->(6.4,6);
   \draw[-,line width=1pt,red] (6.4,6)->(8.3,4);
  \draw[-,line width=1pt,red] (8.3,4)->(6.4,2);
  \draw[-,line width=1pt,red] (6.4,2)->(4.5,4);
   \draw[-,line width=1pt] (6.4,4)->(6.4,5.85);
   \draw[-,line width=1pt] (6.4,4)->(8.15,4);
  \draw[-,line width=1pt] (6.4,4)->(6.4,2.15);
  \draw[-,line width=1pt] (6.4,4)->(4.64,4);
 \draw[-,line width=1pt] (8.4,4.12)->(10,5.6);
   \draw[-,line width=1pt] (10,5.6)->(11.7,4);
  \draw[-,line width=1pt] (11.7,4)->(11,2);
  \draw[-,line width=1pt] (11,2)->(9,2);
 \draw[-,line width=1pt] (9,2)->(8.4,3.88);
  \draw[-,line width=1pt] (10,3.6)->(10,5.6);
   \draw[-,line width=1pt] (10,3.6)->(11.7,4);
  \draw[-,line width=1pt] (10,3.6)->(11,2);
  \draw[-,line width=1pt] (10,3.6)->(9,2);
 \draw[-,line width=1pt] (10,3.6)->(8.45,3.96);
 \draw[-,line width=1pt] (11.7,4)->(14,4);
  \draw[-,line width=1pt] (14,4)->(12.85,5.6);
 \draw[-,line width=1pt] (12.85,5.6)->(11.7,4); 
  \draw[-,line width=1pt] (12.85,4.6)->(14,4);
  \draw[-,line width=1pt] (12.85,4.6)->(12.85,5.6);
 \draw[-,line width=1pt] (12.85,4.6)->(11.7,4);         
  \draw[-,line width=1pt] (11.5,1.5)->(13.7,1.5);
  \draw[-,line width=1pt] (13.7,1.5)->(13.7,-0.7);
  \draw[-,line width=1pt] (13.7,-0.7)->(11.5,-0.7);
  \draw[-,line width=1pt] (11.5,-0.7)->(11.5,1.5);  
    \draw[-,line width=1pt] (12.6,0.4)->(13.7,1.5);
  \draw[-,line width=1pt] (12.6,0.4)->(13.7,-0.7);
  \draw[-,line width=1pt] (12.6,0.4)->(11.5,-0.7);
  \draw[-,line width=1pt] (12.6,0.4)->(11,2);  
   \node at (8.4,6)  {\begin{huge}$F$\end{huge}};  
     \node at (7,3.3)  {\begin{Large}$U$\end{Large}};
 \node at (4.8,5.4)  {\begin{Large}$S=\partial U$\end{Large}}; 
   \end{tikzpicture}
   \caption{An a priori possible schematic picture of a component of the boundary of $F$.}
   \label{face_jordan_fig}
   \end{figure}

Consider the graph $T$ whose vertices are components of $\widehat\C\setminus\overline{F}$, and two vertices $U, V$ are connected by an edge if there is a path in $\widehat\C\setminus F$ connecting $U,V$ and not passing through other components.
Then $T$ is a finite union of trees.
Since each vertex (of $\mathscr{T}$) has valence at least $3$, we have that $e_i\neq e_{i+1}$.
This also implies that two adjacent components $U,V$ either share a common boundary vertex, or there exists an edge of $\mathscr{T}$ connecting them.
Let $U$ be a component of $\widehat\C\setminus\overline{F}$.
If $e_i, e_j$ are a pair of adjacent edges of $\partial U$, but $e_i, e_j$ are not adjacent on the ideal boundary $I(F)$, then there exists a component $V$ of $\widehat\C\setminus\overline{F}$ `attached' to $U$ through the vertex $v = e_i \cap e_j$ (see Figure~\ref{face_jordan_fig}).
Therefore, the number of pairs of adjacent edges of $\partial U$ that are not adjacent on the ideal boundary $I(F)$ equals to the valence of $U$ in $T$.
Choose $U$ to be an end point of $T$, and let $S = \partial U$.

Since $R$ is orientation reversing and each edge of $\mathscr{T}$ is invariant under the map, $R$ reflects the two sides near each open edge of $S$.
Thus, there is a connected component $V$ of $R^{-1}(U)$ with $\partial V \cap \partial U \neq \emptyset$ and $V \cap F \neq \emptyset$. It is easy to see from the local dynamics of $\overline{z}^{m_i}$ at the origin that near a critical point, $R$ sends the region bounded by two adjacent edges of $\mathscr{T}$ to its complement.
Therefore, by our choice of $U$, we have $\partial U \subseteq \partial V$. 
Since $\partial F$ is $R$-invariant, $V \cap \partial F = \emptyset$.
Thus, $V \subseteq F$ and $V$ contains no critical points of $R$.
As $U$ is a disk, the Riemann-Hurwitz formula now implies that $V$ is a disk, and $R$ is a homeomorphism from $\partial V$ to $\partial U$.
Therefore, $\partial U = \partial V$ and $V = F$; so $F$ is a Jordan domain.
\end{proof}

In particular, we have the following.
\begin{cor}
The Tischler graph $\mathscr{T}$ is connected.
\end{cor}

In fact, the Tischler graph gives a topological model for the map $R$.
\begin{cor}\label{cor:tm}
Let $F$ be a face of $\mathscr{T}$, and $F^c = \overline{\widehat\C\setminus F}$ be the closure of its complement. Then $R: \overline{F} \longrightarrow F^c$ is an orientation reversing homeomorphism.
\end{cor}

\begin{lemma}\label{lem:ns}
Let $F_1, F_2$ be two faces of $\mathscr{T}$, then the boundaries share at most one edge.
\end{lemma}
\begin{proof}
Suppose that $F_1$ and $F_2$ share two or more edges.
For $i=1,2$, let $\gamma_i$ be the hyperbolic geodesic arc in $\mathring{F_i}$ connecting the two midpoints of the edges, and let $\gamma = \gamma_1 \cup \gamma_2$.
Since each vertex has valence at least $3$, we see that $\gamma$ is essential in $\widehat\C\setminus V(\mathscr{T}) = \widehat\C\setminus P(R)$ (i.e., $\gamma$ is not homotopic to a point or a puncture of the surface $\widehat\C\setminus P(R)$), where $P(R)$ stands for the post-critical set of $R$ and $V(\mathscr{T})$ denotes the vertex set of the graph $\mathscr{T}$.
By Corollary \ref{cor:tm}, there exists $\gamma'$ homotopic to $\gamma$ in $\widehat\C\setminus P(R)$ such that $R: \gamma' \longrightarrow \gamma$ is a homeomorphism.
This gives a Thurston's obstruction (more precisely, a Levy cycle) for the anti-rational map $R$ (see the discussion in \S \ref{sec:to} on Thurston's theory for rational maps), which is a contradiction. Hence, the supposition that $F_1$ and $F_2$ share two or more edges is false.
\end{proof}

The following result, where we translate the above properties of the Tischler graph to a simple graph theoretic property of its planar dual, plays a crucial role in the combinatorial classification of critically fixed anti-rational maps.

\begin{lemma}\label{lem:d2s}
Let $\Gamma$ be the planar dual of the Tischler graph $\mathscr{T}$ of a critically fixed anti-rational map $R$. Then $\Gamma$ is simple and $2$-connected.
\end{lemma}
\begin{proof}
By Lemma \ref{lem:ns}, no two faces of $\mathscr{T}$ share two edges on their boundary. Hence, the dual graph contains no multi-edge.
Again, as each face of $\mathscr{T}$ is a Jordan domain by Lemma \ref{lem:jd}, the dual graph contains no self-loop.
Therefore $\Gamma$ is simple.

The fact that each face of $\mathscr{T}$ is a Jordan domain also implies that the complement of the closure of each face is connected.
So the dual graph $\Gamma$ remains connected upon deletion of any vertex. In other words, $\Gamma$ is $2$-connected.
\end{proof}

\subsection{Constructing critically fixed anti-rational maps from graphs.}\label{sec:to}
Let $\Gamma$ be a $2$-connected simple plane graph, and $G_\Gamma$ be an associated kissing reflection group.
In Section~\ref{br_cov_nielsen}, we constructed a topological branched covering $\mathcal{G}_\Gamma$ from the Nielsen map of $G_\Gamma$.
In the remainder of this section, we will use this branched covering $\mathcal{G}_\Gamma$ to promote Lemma~\ref{lem:d2s} to a characterization of Tischler graphs of critically fixed anti-rational maps.

\begin{prop}\label{prop:tc}
A plane graph $\mathcal{T}$ is the Tischler graph of a critically fixed anti-rational map $R$ if and only if the dual (plane) graph $\Gamma$ is simple and $2$-connected.
Moreover, $R$ is topologically conjugate to $\mathcal{G}_\Gamma$.
\end{prop}

We will first introduce some terminology.
A post-critically finite branched covering (possibly orientation reversing) of a topological $2$-sphere $\mathbb{S}^2$ is called a {\em Thurston map}. 
We denote the post-critical set of a Thurston map $f$ by $P(f)$. 
Two Thurston maps $f$ and $g$ are \emph{equivalent} if there exist two orientation-preserving homeomorphisms $h_0,h_1:(\mathbb{S}^2,P(f))\to(\mathbb{S}^2,P(g))$ so that $h_0\circ f = g\circ h_1$ where $h_0$ and $h_1$ are isotopic relative to $P(f)$.

A set of pairwise disjoint, non-isotopic, essential, simple, closed curves $\Sigma$ on $\mathbb{S}^2\setminus P(f)$ is called a {\em curve system}. A curve system $\Sigma$ is called $f$-stable if for every curve $\sigma\in\Sigma$, all the essential components of $f^{-1}(\sigma)$ are homotopic rel $P(f)$ to curves in $\Sigma$. We associate to an $f$-stable curve system $\Sigma$ the {\em Thurston linear transformation} 
$$
f_\Sigma: \R^\Sigma \longrightarrow \R^{\Sigma}
$$
defined as 
$$
f_\Sigma(\sigma) = \sum_{\sigma' \subseteq f^{-1}(\sigma)} \frac{1}{\deg(f: \sigma' \rightarrow \sigma)}[\sigma']_\Sigma,
$$
where $\sigma\in\Sigma$, and $[\sigma']_\Sigma$ denotes the element of $\Sigma$ isotopic to $\sigma'$, if it exists.
The curve system is called {\em irrreducible} if $f_\Sigma$ is irreducible as a linear transformation.
It is said to be a {\em Thurston obstruction} if the spectral radius $\lambda(f_\Sigma) \geq 1$.

Similarly, an arc $\lambda$ in $\mathbb{S}^2\setminus P(f)$ is an embedding of $(0,1)$ in $\mathbb{S}^2\setminus P(f)$ with end-points in $P(f)$.
It is said to be {\em essential} if it is not contractible in $\mathbb{S}^2$ fixing the two end-points.
A set of pairwise non-isotopic essential arcs $\Lambda$ is called an {\em arc system}.
The Thurston linear transformation $f_\Lambda$ is defined in a similar way, and we say that it is irreducible if $f_\Lambda$ is irreducible as a linear transformation.

The next proposition allows us to directly apply Thurston's topological characterization theorem for rational maps (see \cite[Theorem~1]{DH93}) to the study of orientation-reversing Thurston maps (see \cite[Theorem~3.9]{Gey20} for an intrinsic topological characterization theorem for anti-rational maps). It is proved by considering the second iterate of the Thurston's pullback map on the Teichm\"uller space of $\mathbb{S}^2\setminus P(f)$. The reader is referred to \cite{DH93} for the definition of hyperbolic orbifold, but this is the typical case as any map with more than four postcritical points has hyperbolic orbifold.

\begin{prop}\cite[Proposition~6.1]{LLMM19}\label{prop:ttar}
Let $f$ be an orientation reversing Thurston map so that $f\circ f$ has hyperbolic orbifold. Then $f$ is equivalent to an anti-rational map if and only if $f\circ f$ is equivalent to a rational map if and only if $f\circ f$ has no Thurston's obstruction.
Moreover, if $f$ is equivalent to an anti-rational map, the map is unique up to M\"obius conjugacy.
\end{prop}

For a curve system $\Sigma$ (respectively, an arc system $\Lambda$), we set $\widetilde{\Sigma}$ (respectively, $\widetilde{\Lambda}$) as the union of those components of $f^{-1}(\Sigma)$ (respectively, $f^{-1}(\Lambda)$) which are isotopic relative to $P(f)$ to elements of $\Sigma$ (respectively, $\Lambda$).
We will use $\Sigma \cdot \Lambda$ to denote the minimal intersection number between them.
We will be using the following theorem excerpted and paraphrased from \cite[Theorem 3.2]{PT98}.
\begin{theorem}\cite[Theorem 3.2]{PT98}\label{thm:pt}
Let $f$ be an orientation preserving Thurston map, $\Sigma$ an irreducible Thurston obstruction in $(\mathbb{S}^2, P(f))$, and $\Lambda$ an irreducible arc system in $(\mathbb{S}^2, P(f))$.
Assume that $\Sigma$ intersect $\Lambda$ minimally,
then either
\begin{itemize}
\item $\Sigma \cdot \Lambda = 0$; or
\item $\Sigma \cdot \Lambda \neq 0$ and for each $\lambda \in \Lambda$, there is exactly one connected component $\lambda'$ of $f^{-1}(\lambda)$ such that $\lambda' \cap \widetilde{\Sigma} \neq \emptyset$. Moreover, the arc $\lambda'$ is the unique component of $f^{-1}(\lambda)$ which is isotopic to an element of $\Lambda$.
\end{itemize}
\end{theorem}

With these preparations, we are now ready to show that $\mathcal{G}_\Gamma$ is equivalent to an anti-rational map.
The proof is similar to \cite[Proposition 6.2]{LLMM19}.
\begin{lemma}\label{lem:te}
Let $\Gamma$ be a $2$-connected, simple, plane graph. Then
$\mathcal{G}_\Gamma$ is equivalent to a critically fixed anti-rational map $\mathcal{R}_\Gamma$.
\end{lemma}
\begin{proof}
It is easy to verify that $\mathcal{G}_\Gamma\circ\mathcal{G}_\Gamma$ has hyperbolic orbifold, except when $\Gamma=\Gamma_d$, in which case $\mathcal{G}_\Gamma$ has only two fixed critical points each of which is fully branched, and $\mathcal{G}_\Gamma$ is equivalent to $\overline{z}^d$.
Thus, by  Proposition \ref{prop:ttar}, we will prove the lemma by showing that there is no Thurston's obstruction for $\mathcal{G}_\Gamma\circ\mathcal{G}_\Gamma$.

We will assume that there is a Thurston obstruction $\Sigma$, and arrive at a contradiction. After passing to a subset, we may assume that $\Sigma$ is irreducible.
Isotoping the curve system $\Sigma$, we may assume that $\Sigma$ intersects the graph $\mathscr{T}(\mathcal{G}_\Gamma)$ minimally (see Section~\ref{br_cov_nielsen} for the definition of $\mathscr{T}(\mathcal{G}_\Gamma)$).
Let $\lambda$ be an edge in $\mathscr{T}(\mathcal{G}_\Gamma)$, then $\Lambda = \{\lambda\}$ is an irreducible arc system. 
Let $\widetilde{\Sigma}$ be the union of those components of $\mathcal{G}_\Gamma^{-2}(\Sigma)$ which are isotopic to elements of $\Sigma$.

We claim that $\widetilde{\Sigma}$ does not intersect $\mathcal{G}_\Gamma ^{-2}(\lambda) \setminus \lambda$.
Indeed, applying Theorem~\ref{thm:pt}, we are led to the following two cases.
In the first case, we have $\Sigma \cdot \lambda = 0$; hence $\mathcal{G}_\Gamma ^{-2}(\Sigma)$ cannot intersect $\mathcal{G}_\Gamma ^{-2}(\lambda)$, and the claim follows.
In the second case, since $\mathcal{G}_\Gamma (\lambda) = \lambda$, we conclude that $\lambda$ is the unique component of $\mathcal{G}_\Gamma^{-2}(\lambda)$ isotopic to $\lambda$.
Thus, the only component of $\mathcal{G}_\Gamma^{-2}(\lambda)$ intersecting $\widetilde{\Sigma}$ is $\lambda$, so the claim follows.

Applying this argument on all the edges of $\mathscr{T}(\mathcal{G}_\Gamma)$, we conclude that $\widetilde{\Sigma}$ does not intersect $\mathcal{G}_\Gamma ^{-2}(\mathscr{T}(\mathcal{G}_\Gamma)) \setminus \mathscr{T}(\mathcal{G}_\Gamma)$.

Let $F$ be a face of $\mathscr{T}(\mathcal{G}_\Gamma)$ with vertices of $\partial F$ denoted by $v_1, ..., v_k$ counterclockwise.
Denote $e_i \subseteq \partial F$ be the open edge connecting $v_i, v_{i+1}$.
We claim that the graph obtained by removing the boundary edges of $F$ 
$$
\mathscr{T}_{F}:=\mathscr{T}(\mathcal{G}_\Gamma) \setminus \bigcup_{i=1}^k e_i
$$ 
is still connected.
Since $\mathscr{T}(\mathcal{G}_\Gamma)$ is connected, any vertex $v$ is connected to the set $\{v_1,..., v_k\}$ in $\mathscr{T}_{F}$.
Thus, it suffices to show that $v_i$ and $v_{i+1}$ are connected in $\mathscr{T}_{F}$ for all $i$.
As the planar dual $\Gamma$ of $\mathscr{T}(\mathcal{G}_\Gamma)$ is $2$-connected and has no self-loop, each face of $\mathscr{T}(\mathcal{G}_\Gamma)$ is a Jordan domain. Hence, there exists a face $F_i\neq F$ of $\mathscr{T}(\mathcal{G}_\Gamma)$ sharing $e_i$ on the boundary with $F$. Moreover, since $\Gamma$ has no multi-edge, the boundaries of any two faces of $\mathscr{T}(\mathcal{G}_\Gamma)$ intersect at most at one edge. Thus, $\partial F_i \cap \partial F = e_i \cup \{v_i, v_{i+1}\}$.
Hence, $v_i$ and $v_{i+1}$ are connected by a path in $\partial F_i \setminus e_i \subseteq \mathscr{T}_{F}$, and this proves the claim.

Therefore, by Lemma~\ref{lem:uc}, we deduce that 
$$
(\mathcal{G}_\Gamma^{-1} (\mathscr{T}(\mathcal{G}_\Gamma)) \setminus \mathscr{T}(\mathcal{G}_\Gamma)) \cup  V(\mathscr{T}(\mathcal{G}_\Gamma))
$$ 
is connected.
Thus, $(\mathcal{G}_\Gamma ^{-2}(\mathscr{T}(\mathcal{G}_\Gamma)) \setminus \mathscr{T}(\mathcal{G}_\Gamma)) \cup  V(\mathscr{T}(\mathcal{G}_\Gamma))$ is a connected graph containing the post-critical set $P(\mathcal{G}_\Gamma \circ \mathcal{G}_\Gamma) =  V(\mathscr{T}(\mathcal{G}_\Gamma))$.
This forces $\Sigma$ to be empty, which is a contradiction.

Therefore, $\mathcal{G}_\Gamma$ is equivalent to an anti-rational map $\mathcal{R}_\Gamma$. Since $\mathcal{G}_\Gamma$ is critically fixed, so is $\mathcal{R}_\Gamma$.
\end{proof}

Let $h_0, h_1:(\mathbb{S}^2,P(\mathcal{G}_\Gamma))\to(\mathbb{S}^2,P(\mathcal{R}_\Gamma))$ be two orientation preserving homeomorphisms so that $h_0\circ \mathcal{G}_\Gamma = \mathcal{R}_\Gamma\circ h_1$ where $h_0$ and $h_1$ are isotopic relative to $P(\mathcal{G}_\Gamma) = V(\mathscr{T}(\mathcal{G}_\Gamma))$.

We now use a standard pullback argument to prove the following.

\begin{lemma}\label{lem:it}
$h_0(\mathscr{T}(\mathcal{G}_\Gamma))$ is isotopic to the Tischler graph $\mathscr{T}(\mathcal{R}_\Gamma)$.
\end{lemma}
\begin{proof}
Let $\alpha$ be an edge of $\mathscr{T}(\mathcal{G}_\Gamma)$, then $\mathcal{G}_\Gamma (\alpha) = \alpha$.
Consider the sequence of homeomorphisms $h_i$ with
$$
h_{i-1} \circ \mathcal{G}_\Gamma = \mathcal{R}_\Gamma \circ h_i.
$$
Each $h_i$ is normalized so that it carries $P(\mathcal{G}_\Gamma)$ to $P(\mathcal{R}_\Gamma)$. Note by induction, $\beta_i = h_i(\alpha)$ is isotopic to $\beta_0$ relative to the end-points.
Applying isotopy, we may assume that there is a decomposition $\beta_0 = \beta_0^- + \gamma_0 + \beta_0^+$, where $\beta_0^\pm$ are two internal rays and $\gamma_0$ does not intersect any invariant Fatou component.
Then, every $\beta_i = \beta_i^- + \gamma_i + \beta_i^+$ has such a decomposition too.

Note that $\mathcal{R}_\Gamma(\gamma_{i+1}) = \gamma_{i}$ and $\mathcal{R}_\Gamma:\gamma_{i+1}\to \gamma_{i}$ is injective for each $i$.
Since $\mathcal{R}_\Gamma$ is hyperbolic, and the curves $\gamma_i$ are uniformly bounded away from the post-critical point set of $\mathcal{R}_\Gamma$, we conclude that the sequence of curves $\{\gamma_i\}$ shrinks to a point.
Thus, if $\beta$ is a limit of $\beta_i$ in the Hausdorff topology, then $\beta$ is a union of two internal rays and it is isotopic to $\beta_0$.
Therefore, $h_0(\alpha)$ is isotopic to an edge of $\mathscr{T}(\mathcal{R}_\Gamma)$.

Since this is true for every edge of $\mathscr{T}(\mathcal{G}_\Gamma)$ and $h_0$ is an orientation preserving homeomorphism, by counting the valence at every critical point, we conclude that $h_0(\mathscr{T}(\mathcal{G}_\Gamma))$ is isotopic to $\mathscr{T}(\mathcal{R}_\Gamma)$.
\end{proof}

Therefore, after performing isotopy, we may assume that $h_0$ restricts to a homeomorphism between the graphs $h_0: \mathscr{T}(\mathcal{G}_\Gamma) \longrightarrow \mathscr{T}(\mathcal{R}_\Gamma)$.
By lifting, $h_1: \mathscr{T}(\mathcal{G}_\Gamma) \longrightarrow \mathscr{T}(\mathcal{R}_\Gamma)$ is also a homeomorphism.
Let $F$ be a face of $\Gamma$. By our construction of $\mathcal{G}_\Gamma$ (see Equation~\ref{eqn:2}), $\mathcal{G}_\Gamma$ is conjugate to $\bar{z}^{d_F}$ on the corresponding component $\Omega_F$, where $d_F + 1$ is the number of edges in the ideal boundary of the face $F$.
Since $\mathscr{T}(\mathcal{G}_\Gamma)$ is constructed by taking the union over all faces of the fixed internal rays under such conjugacies (see Equation~\ref{eqn:3}), the dynamics of $\mathcal{G}_\Gamma$ and $\mathcal{R}_\Gamma$ on each edge of $\mathscr{T}(\mathcal{G}_\Gamma)$ and $\mathscr{T}(\mathcal{R}_\Gamma)$ are conjugate.
Therefore, after further isotoping $h_0$, we may assume that
$$
h_0 = h_1 : \mathscr{T}(\mathcal{G}_\Gamma) \longrightarrow \mathscr{T}(\mathcal{R}_\Gamma)
$$ 
gives a conjugacy between $\mathcal{G}_\Gamma$ and $\mathcal{R}_\Gamma$ from $\mathscr{T}(\mathcal{G}_\Gamma)$ to $\mathscr{T}(\mathcal{R}_\Gamma)$.
Hence, both $h_0$ and $h_1$ send a face of $\mathscr{T}(\mathcal{G}_\Gamma)$ homomorphically to the corresponding face of $\mathscr{T}(\mathcal{R}_\Gamma)$.

As an application of the above, we will show that the dynamics on the limit set and Julia set are topologically conjugate (cf. \cite[Theorem~6.11]{LLMM19}).

\begin{lemma}\label{lem:tcl}
There is a homeomorphism $h: \Lambda(G_\Gamma) \longrightarrow \mathcal{J}(\mathcal{R}_\Gamma)$ satisfying 
$h\circ \mathcal{G}_\Gamma = \mathcal{R}_\Gamma \circ h$.
\end{lemma}
\begin{proof}
Let $P^0$ consist of the points of tangency of the circle packing, $P^{i+1} := \mathcal{G}_\Gamma^{-1}(P^i)$, and $P^\infty := \bigcup_{i=0}^\infty P^i$.
Then $P^\infty$ corresponds to the $G_\Gamma$ orbit of the cusps, and hence is dense in $\Lambda(G_\Gamma)$.
Let $Q^0 := h_0(P^0)$, which is also the set of repelling fixed points of $\mathcal{R}_\Gamma$, and $Q^{i+1} := \mathcal{R}_\Gamma^{-1}(Q^i)$.
Then $Q^\infty :=  \bigcup_{i=0}^\infty Q^i$ is dense in $\mathcal{J}(\mathcal{R}_\Gamma)$.

Recall that for each face $F$ of $\Gamma$, the dynamics of $\mathcal{G}_\Gamma$ (respectively, of $\mathcal{R}_\Gamma$) on the corresponding component of $\Omega(G_\Gamma)$ (respectively, on the corresponding critically fixed Fatou component of $\mathcal{R}_\Gamma$) can be uniformized to $\overline{z}^d$ on $\D$.
Let $\Pi$ be the closed ideal $d+1$-gon in $\D$ with ideal vertices at the fixed points of $\overline{z}^d$ on $\mathbb{S}^1$.
Let $\mathcal{L}_{\mathcal{G}, F}$ and $\mathcal{L}_{\mathcal{R},F}$ be the image of $\Pi$ under the uniformizing map in the dynamical plane of $\mathcal{G}_\Gamma$ and $\mathcal{R}_\Gamma$ respectively, and $\mathcal{L}_\mathcal{G}$ and $\mathcal{L}_\mathcal{R}$ be the union over all faces.

Let $\mathcal{E}^0 := \widehat\C\setminus\overline{\mathcal{L}_\mathcal{G}}$ and $\mathcal{H}^0 := \widehat\C\setminus\overline{\mathcal{L}_\mathcal{R}}$.
Inductively, let $\mathcal{E}^{i+1} := \mathcal{G}_\Gamma^{-1}(\mathcal{E}^i)$ and $\mathcal{H}^{i+1} := \mathcal{R}_\Gamma^{-1}(\mathcal{H}^i)$.
Note that these sets are the analogues of $\mathcal{D}^{i}$ for kissing reflection groups. 
Since $\mathcal{G}_\Gamma$ (respectively, $\mathcal{R}_\Gamma$) sends a face of $\mathscr{T}(\mathcal{G}_\Gamma)$ (respectively, a face of $\mathscr{T}(\mathcal{R}_\Gamma)$) to its complement univalently, our construction guarantees that $\overline{\mathcal{E}^1} \subseteq \overline{\mathcal{E}^0}$ and $\overline{\mathcal{H}^1} \subseteq \overline{\mathcal{H}^0}$.
Inductively, we see that $\{\overline{\mathcal{E}^{i}}\}$ and $\{\overline{\mathcal{H}^i}\}$ are nested sequences of closed sets.
Thus, we have 
$$
\Lambda(G_\Gamma) = \bigcap_{i=0}^\infty \overline{\mathcal{E}^{i}}
$$
and
$$
\mathcal{J}(\mathcal{R}_\Gamma) = \bigcap_{i=0}^\infty \overline{\mathcal{H}^{i}}.
$$

Since $\mathcal{R}_\Gamma$ is hyperbolic, and $\mathcal{H}^i$ contains no critical value, the diameters of the components of $\mathcal{H}^{i}\cap \mathcal{J}(\mathcal{R}_\Gamma)$ shrink to $0$ uniformly.
On the other hand, since $ \mathcal{G}_\Gamma = \mathcal{N}_\Gamma$ on $\Lambda(G_\Gamma)$, the diameters of the components of $\mathcal{E}^{i}\cap \Lambda(G_\Gamma)$ shrink to $0$ uniformly by Lemma \ref{lem:maxr}.

After isotoping $h_0$, we may assume that $h_0 (\mathcal{E}^0) = \mathcal{H}^0$.
We consider the pullback sequence $\{h_i\}$ with
$$
h_{i-1} \circ \mathcal{G}_\Gamma = \mathcal{R}_\Gamma \circ h_i,
$$
where each $h_i$ carries $P(\mathcal{G}_\Gamma)$ to $P(\mathcal{R}_\Gamma)$. 
Since we have assumed $h_0 = h_1$ on $P^0$, inductively, we have $h_j = h_i$ on $P_i$ and $h_j(P^i) = Q^i$ for all $j\geq i$.
We define $h(x) := \lim h_i(x)$ for $x\in P^\infty$.

Then, a similar argument as in the proof of Proposition \ref{prop:mg} implies that $h$ is uniformly continuous on $P^\infty$. 
Indeed, given $\epsilon >0$, we can choose $N$ so that the diameter of any component of $\mathcal{H}^N\cap \mathcal{J}(\mathcal{R}_\Gamma)$ is less than $ \epsilon$.
We choose $\delta$ so that any two non-adjacent component of $\mathcal{E}^N\cap \Lambda(G_\Gamma)$ are at least $\delta$ distance away.
If $x,y\in P^\infty$ are at most $\delta$ distance apart, they lie in two adjacent components of $\mathcal{E}^N$, so $h_j(x), h_j(y)$ lie in two adjacent components of $\mathcal{H}^N$ for all $j\geq N$.
Hence, $d(h_j(x), h_j(y)) < 2\epsilon$.

Since $P^\infty$ and $Q^\infty$ are dense on $\Lambda(G_\Gamma)$ and $\mathcal{J}(\mathcal{R}_\Gamma)$ (respectively), we get a unique continuous extension
$$
h:\Lambda(G_\Gamma) \longrightarrow \mathcal{J}(\mathcal{R}_\Gamma).
$$

Applying the same argument on the sequence $\{h_j^{-1}\}$, we get the continuous inverse of $h$.
Thus, $h$ is a homeomorphism.
Since $h$ conjugates $\mathcal{G}_\Gamma$ to $\mathcal{R}_\Gamma$ on $P^\infty$, it also conjugates $\mathcal{G}_\Gamma$ to $\mathcal{R}_\Gamma$ on $\Lambda(G_\Gamma)$.
\end{proof}

We are now ready to prove Proposition~\ref{prop:tc} (cf. \cite[Proposition~6.13]{LLMM19}).

\begin{proof}[Proof of Proposition \ref{prop:tc}]
By Lemma \ref{lem:d2s}, the dual graph of a Tischler graph (of a critically fixed anti-rational map) is $2$-connected, simple and plane.

Conversely, given any $2$-connected, simple, plane graph $\Gamma$, by Lemma \ref{lem:te} and Lemma \ref{lem:it}, we can construct a critically fixed anti-rational map $\mathcal{R}_\Gamma$ whose Tischler graph is the planar dual of $\Gamma$.

It remains to show that $\mathcal{G}_\Gamma$ and $\mathcal{R}_\Gamma$ are topologically conjugate.
Let $h$ be the topological conjugacy on $\Lambda(G_\Gamma)$ produced in Lemma \ref{lem:tcl}.
Let $U$ be a critically fixed component $\mathcal{G}_\Gamma$.
Then by construction, there exists $\Phi: \mathbb{D}\longrightarrow U$ conjugating $m_{-d}(z) = \overline{z}^d$ to $\mathcal{G}_\Gamma$.
The map extends to a semiconjugacy between $\mathbb{S}^1$ and $\partial U$.

Similarly, let $V$ be the corresponding critically fixed Fatou component of $\mathcal{R}_\Gamma$.
Then we have $\Psi: \mathbb{D}\longrightarrow V$ conjugating $m_{-d}(z) = \overline{z}^d$ to $\mathcal{R}_\Gamma$, which extends to a semiconjugacy.
Note that there are $d+1$ different choices of such a conjugacy, we may choose one so that $h\circ \Phi = \Psi$ on $\mathbb{S}^1$.
Thus, we can extend the topological conjugacy $h$ to all of $U$ by setting $h:= \Psi\circ \Phi^{-1}$.

In this way, we obtain a homeomorphic extension of $h$ to all critically fixed components of $\mathcal{G}_\Gamma$.
Lifting these extensions to all the preimages of the critically fixed components of $\mathcal{G}_\Gamma$, we get a map defined on $\widehat\C$.
Since the diameters of the preimages of the critically fixed components (of $\mathcal{G}_\Gamma$) as well as the corresponding Fatou components (of $\mathcal{R}_\Gamma$) go to $0$ uniformly, we conclude that all these extensions paste together to yield a global homeomorphism, which is our desired topological conjugacy.
\end{proof}

In light of Proposition \ref{prop:tc}, we see that the association of the isomorphism class of a $2$-connected, simple, plane graph to the M\"obius conjugacy class of a critically fixed anti-rational map is well defined and surjective.
To verify that this is indeed injective, we remark that if two (plane) graphs $\Gamma$ and $\Gamma'$ are isomorphic as plane graphs, then the associated Tischler graphs are also isomorphic as plane graphs.
This means that the corresponding pair of critically fixed anti-rational maps are Thurston equivalent.
Thus by Thurston's rigidity result, they are M\"obius conjugate.

As in the kissing reflection group setting, we define a geometric mating of two (anti-)polynomials as follows.

\begin{definition}\label{defn:gm}
We say that a rational map (or an anti-rational map)  $R$ is a {\em geometric mating} of two polynomials (or two anti-polynomials) $P^\pm$ with connected Julia sets if we have
\begin{itemize}
\item a decomposition of the Fatou set $\mathcal{F}(R) = \mathcal{F}^+ \sqcup \mathcal{F}^-$ with $\mathcal{J}(R) = \partial \mathcal{F}^+ = \partial \mathcal{F}^-$; and
\item two continuous surjections from the filled Julia sets $\psi^\pm : \mathcal{K}(P^\pm) \to\overline{\mathcal{F}^\pm}$ that are conformal between $\Int{\mathcal{K}}(P^\pm)$ and $\mathcal{F}^\pm$ 
\end{itemize}
so that
$$
\psi^\pm \circ P^\pm = R \circ \psi^\pm.
$$
\end{definition}

As a corollary, we have
\begin{cor}\label{cor:h3}
\noindent\begin{enumerate}
\item A critically fixed anti-rational map $R$ is an anti-polynomial if and only if the dual graph $\Gamma$ of $\mathscr{T}(R)$ is outerplanar.

\item A critically fixed anti-rational map $R$ is a geometric mating of two anti-polynomials if and only if the dual graph $\Gamma$ of $\mathscr{T}(R)$ is Hamiltonian.

\item A critically fixed anti-rational map $R$ has a gasket Julia set if and only if the dual graph $\Gamma$ of $\mathscr{T}(R)$ is $3$-connected.
\end{enumerate}
\end{cor}
\begin{proof}
The first statement follows from Proposition \ref{prop:tc} as the graph $\Gamma$ is outerplanar if and only if there is a vertex in the dual graph $\Gamma^\vee=\mathscr{T}(R)$ with maximal valence if and only if $R$ has a fully branched fixed critical point.

For the second statement, let $\Gamma$ be a $2$-connected, simple, plane graph.
Let $G_\Gamma$ be an associated kissing reflection group, $\mathcal{G}_\Gamma$ and $\mathcal{R}_\Gamma$ be the associated branched covering and anti-rational map.
By Proposition \ref{prop:tc}, $\mathcal{G}_\Gamma$ and $\mathcal{R}_\Gamma$ are topologically conjugated by $h$.

If $\mathcal{R}_\Gamma$ is a geometric mating of two anti-polynomials $P^\pm$ which are necessarily critically fixed, then the conjugacy gives a decomposition of $\Omega(G_\Gamma) = \Omega^+\sqcup \Omega^-$. 
It is not hard to check that the $\Gamma$-actions on $\overline{\Omega^\pm}$ are conjugate to the actions of the function kissing reflection groups $G^\pm$ on their filled limit sets, where $G^\pm$ correspond to $P^\pm$.
Thus, $G_\Gamma$ is a geometric mating of the two function kissing reflection groups $G^\pm$, so $\Gamma$ is Hamiltonian by Proposition \ref{prop:mg}.

Conversely, if $\Gamma$ is Hamiltonian, then $G_\Gamma$ is a geometric mating of two function kissing reflection groups $G^\pm$, and we get a decomposition of $\Omega(G_\Gamma)$ by Proposition \ref{prop:mg}. The conjugacy $h$ transports this decomposition to the anti-rational map setting.
One can now check directly that $\mathcal{R}_\Gamma$ is a geometric mating of the anti-polynomials $P^\pm$ that are associated to the function kissing reflection groups $G^\pm$.

The third statement follows immediately from Propositions~\ref{prop:tc} and~\ref{prop:3g}.
\end{proof}

\begin{remark}
The Hubbard tree of a critically fixed anti-polynomial $P$ is a strict subset of $\mathscr{T}(P)\cap\mathcal{K}(P)$, where $\mathcal{K}(P)$ is the filled Julia set of $P$ (cf. \cite[\S 5]{LMM20}).
\end{remark}

\begin{remark}
It is not hard to see that the existence of a Hamiltonian cycle for $\mathscr{T}(R)^{\vee}$ is equivalent to the existence of an equator for $R$ in the sense of \cite[Definition~4.1]{Mey14}. This gives an alternative proof of the second statement of Corollary~\ref{cor:h3} (cf. \cite[Theorem~4.2]{Mey14}).
\end{remark}

\subsection{Dynamical correspondence}\label{dyn_corr_subsec}
Let $G_\Gamma$ and $\mathcal{R}_\Gamma$ be a kissing reflection group and the critically fixed anti-rational map associated to a $2$-connected simple plane graph $\Gamma$.
Here we summarize the correspondence between various dynamical objects associated with the group and the anti-rational map.

\begin{itemize}
\item \textbf{Markov partitions for limit and Julia sets:}
The associated circle packing $\mathcal{P}$ gives a Markov partition for the group dynamics on the limit set (more precisely, for the action of the Nielsen map $\mathcal{N}_\Gamma$ on $\Lambda(G_\Gamma)$).
On the other side, the faces of the Tischler graph determine a Markov partition for the action of $\mathcal{R}_\Gamma$ on the Julia set. The topological conjugacy respects this pair of Markov partitions and the itineraries of points with respect to the corresponding Markov partitions.

\item \textbf{Cusps and repelling fixed points:}
The orbits of parabolic fixed points of $G_\Gamma$ bijectively correspond to pairs of adjacent vertices, or equivalently, to the edges of $\Gamma$.
On the other side, the repelling fixed points of $\mathcal{R}_\Gamma$ are in bijective correspondence with the edges of the Tischler graph $\mathscr{T} = \Gamma^\vee$, and thus also with the edges of $\Gamma$.
They are naturally identified by the topological conjugacy.

\item \textbf{$\sigma$-invariant curves and $2$-cycles:}
Each $\sigma$-invariant simple closed curve on the conformal boundary $\partial\mathcal{M}(G_\Gamma)$ corresponds to a pair of non-adjacent vertices on a face of $\Gamma$ (see the discussion on pinching deformations after Proposition~\ref{thm:closure} for the definition of the involution $\sigma$).
On the other side, the above Markov partition for $\mathcal{R}_\Gamma\vert_{\mathcal{J}(\mathcal{R}(\Gamma))}$ shows that each $2$-cycle on the {\em ideal boundary} of an invariant Fatou component corresponds to a pair of non-adjacent vertices on the associated face of $\Gamma$.
These $\sigma$-invariant curves and $2$-cycles are naturally identified by the topological conjugacy.
We remark that it is possible that a $\sigma$-invariant curve yields an accidental parabolic element, in which case the corresponding $2$-cycle on the ideal boundary of the invariant Fatou component coalesces, and gives rise to a repelling fixed point of $\mathcal{R}_\Gamma$.
This happens if and only if the two corresponding vertices are adjacent in $\Gamma$.

\item \textbf{Question mark conjugacy:}
If we choose our group $G_\Gamma$ so that each conformal boundary is the double of a suitable regular ideal polygon, then the restriction of the homeomorphism $h$ of Lemma~\ref{lem:tcl} between the boundary of an invariant Fatou component and the boundary of the corresponding component of $\Omega(G_\Gamma)$ gives a homeomorphism between the ideal boundaries $\phi:\mathbb{S}^1\longrightarrow \mathbb{S}^1$ that conjugates $\overline{z}^e$ to the Nielsen map $\mathcal{N}_e$ of the regular ideal $(e+1)$-gon reflection group.
This conjugacy $\phi$ is a generalization of the Minkowski question mark function (see \cite[\S 3.2, \S 5.4.2]{LLMM18} \cite[Remark~9.1]{LLMM19}).

\item \textbf{Function groups and anti-polynomials:}
If $G_\Gamma$ is a function kissing reflection group (of rank $d+1$), then it can be constructed by pinching a $\sigma$-invariant multicurve $\alpha$ on one component of $\partial\mathcal{M}(\mathbf{G}_d)$ (recall that $\mathbf{G}_d$ is the regular ideal polygon group of rank $d+1$).
Under the natural orientation reversing identification of the two components of $\partial\mathcal{M}(\mathbf{G}_d)$, the multicurve $\alpha$ gives a $\sigma$-invariant multicurve $\alpha'$ on the component of $\partial\mathcal{M}(G_\Gamma)$ associated to the $G_\Gamma$-invariant domain $\Omega_0\subseteq\Omega(G_\Gamma)$. The multicurve $\alpha'$ consists precisely of the simple closed curves corresponding to the accidental parabolics of $G_\Gamma$.
On the other side, the corresponding $2$-cycles on the ideal boundary of the unbounded Fatou component (i.e., the basin of infinity) generate the lamination for the Julia set of the anti-polynomial $R_\Gamma$.

\item \textbf{QuasiFuchsian closure and mating:} If $G_\Gamma$ lies in the closure of the quasiFuchsian deformation space of $\mathbf{G}_d$, then $G_\Gamma$ is obtained by pinching two non-parallel $\sigma$-invariant multicurves $\alpha^+$ and $\alpha^-$ on the two components of $\partial\mathcal{M}(\mathbf{G}_d)$; equivalently, $G_\Gamma$ is a mating of two function kissing reflection groups $G^+$ and $G^-$. On the other side, the critically fixed anti-rational map $\mathcal{R}_\Gamma$ is a mating of the critically fixed anti-polynomials $P^+$ and $P^-$, which correspond to the groups $G^+$ and $G^-$ (respectively). The topological conjugacy between $\mathcal{N}_\Gamma\vert_{\Lambda(G_\Gamma)}$ and $\mathcal{R}_\Gamma\vert_{\mathcal{J}(\mathcal{R}_\Gamma)}$ is induced by the circle homeomorphism that conjugates $\mathcal{N}_d$ to $\overline{z}^d$.
\end{itemize}

\subsection{Mating of two anti-polynomials.}
In this subsection, we shall discuss the converse question of mateablity in terms of laminations.
Let 
$$
P(z) = \overline{z}^d+ a_{d-2}\overline{z}^{d-2}+a_{d-3}\overline{z}^{d-3}+...+a_0
$$ 
be a monic centered anti-polynomial with connected Julia set.
We denote the filled Julia set by $\mathcal{K}(P)$. There is a unique B\"ottcher coordinate 
$$
\psi: \C\setminus\overline{\D} \longrightarrow \C\setminus \mathcal{K}(P)
$$
with derivative $1$ at infinity that conjugates $\overline{z}^d$ to $P$.
We shall call a monic centered anti-polynomial equipped with this preferred B\"ottcher coordinate a {\em marked} anti-polynomial. 
Note that when the Julia set is connected, this marking is equivalent to an affine conjugacy class of anti-polynomials together with a choice of B\"ottcher coordinate at infinity.
If we further assume that the Julia set is locally connected, the map $\psi$ extends to a continuous semi-conjugacy $\psi:\mathbb{S}^1 \longrightarrow \mathcal{J}(P)$ between $\overline{z}^d$ and $P$.
This semi-conjugacy gives rise to a $\overline{z}^d$-{\em invariant lamination} on the circle $\mathbb{S}^1$ for the marked anti-polynomial.
The coordinate on $\mathbb{S}^1\cong\R/\Z$ will be called {\em external angle}. The image of $\{re^{2\pi i\theta}\in \C: r\geq 1\}$ under $\psi$ will be called an {\em external ray} at angle $\theta$, and will be denoted by $R(\theta)$.
If $x\in \mathcal{J}(P)$ is the intersection of two external rays, then $x$ is a cut-point of $\mathcal{J}(P)$ (equivalently, a cut-point of $\mathcal{K}(P)$). 

If we denote $\circled{$\C$} = \C \cup \{(\infty, w): w\in \mathbb{S}^1\}$ as the complex plane together with the circle at infinity, then a marked anti-polynomial extends continuously to $\circled{$\C$}$ by the formula $P(\infty, w) = (\infty, \overline{w}^d)$.

Let $P$ and $Q$ be two marked degree $d$ anti-polynomials. 
To distinguish the two domains $\circled{$\C$}$, we denote them by $\C_P$ and $\C_Q$. 
The quotient
\[\mathbb{S}^2_{P,Q}:=(\C_P\cup\C_Q)/\{(\infty_P,w)\sim(\infty_Q,\overline{w})\}\]
defines a topological sphere. 
Denote the equivalence class of $(\infty_P,w)$ by $(\infty,w)$, and define the \emph{equator} of $\mathbb{S}^2_{P,Q}$ to be the set $\{(\infty,w): w\in \mathbb{S}^1\}$.
The \emph{formal mating} of $P$ and $Q$ is the degree $d$ branched cover $P\mate Q:\mathbb{S}^2_{f,g}\to \mathbb{S}^2_{f,g}$ defined by the formula 

\[(P\mate Q)(z) = \begin{cases} 
      P(z) & z\in\C_P, \\
      (\infty,\overline{z}^d) & \text{for }(\infty,z), \\
      Q(z) & z\in\C_Q .
   \end{cases}\]

Suppose that $P$ and $Q$ have connected and locally connected Julia sets. 
The closure in $\mathbb{S}^2_{P,Q}$ of the external ray at angle $\theta$ in $\C_P$ is denoted $R_P(\theta)$, and likewise the closure of the external ray at angle $\theta$ in $\C_Q$ is denoted $R_Q(\theta)$. Then, $R_P(\theta)$ and $R_Q(-\theta)$ are rays in $\mathbb{S}^2_{P,Q}$ that share a common endpoint $(\infty,e^{2\pi i\theta})$.
The \emph{extended external ray at angle $\theta$} in $\mathbb{S}^2_{P,Q}$ is defined as $R(\theta):=R_P(\theta)\cup R_Q(-\theta)$. Each extended external ray intersects each of $\mathcal{K}(P)$, $\mathcal{K}(Q)$, and the equator at exactly one point.

We define the ray equivalence $\sim_{ray}$ as the smallest equivalence relation on $\mathbb{S}^2_{P,Q}$ so that two points $x,y\in \mathbb{S}^2_{P,Q}$ are equivalent if there exists $\theta$ so that $x,y\in R(\theta)$. A ray equivalence class $\gamma$ can be considered as an embedded graph: the vertices are the points in $\gamma\cap(\mathcal{K}(P)\cup \mathcal{K}(Q))$, and the edges are the external extended rays in $\gamma$.
We also denote by $[x]$ the ray equivalence class of $x$. The diameter of a ray equivalence class is computed with respect to the graph metric. The vertex set of a ray equivalence class $[x]$ has a natural partition consisting of the two sets $[x]\cap\mathcal{K}(P)$ and $[x]\cap\mathcal{K}(Q)$, and using this partition $[x]$ is seen to be a bipartite graph.

An equivalence relation $\sim$ on $\mathbb{S}^2$ is said to be \emph{Moore-unobstructed} if the quotient $\mathbb{S}^2/{\sim}$ is homeomorphic to $\mathbb{S}^2$. The following theorem is due to A. Epstein (see \cite[Proposition 4.12]{PM12}).

\begin{prop}\label{prop:epstein}
Let $P$ and $Q$ be two anti-polynomials of equal degree $d\geq 2$ with connected and locally connected Julia sets. If the equivalence classes of $\sim_{ray}$ are ray-trees having uniformly bounded diameter, then $\mathbb{S}^2_{P,Q}/{\sim_{ray}}$ is Moore-unobstructed.
\end{prop}

We remark that the proof of this theorem uses Moore's theorem. It is originally stated only for polynomials, but the proof extends identically to the case of anti-polynomials.
There is also a partial converse of the above statement which we will not be using here.

For a critically fixed anti-polynomial $P$, we have the following description of the cut-points of $\mathcal{K}(P)$. The result directly follows from the dynamical correspondence between function kissing reflection groups and critically fixed anti-polynomials discussed in \S \ref{dyn_corr_subsec}.

\begin{lemma}\label{lem:ef}
Let $P$ be a critically fixed anti-polynomial. Then, any cut-point of $\mathcal{K}(P)$ is eventually mapped to a repelling fixed point that is the landing point of a $2$-cycle of external rays.
\end{lemma}

An extended external ray that contains a repelling fixed point of a critically fixed anti-polynomial is called a \emph{principal ray}. The ray equivalence class of a principal ray is called a \emph{principal ray class}. 
In our setting, the Moore obstructions can be detected simply by looking at these principal ray classes.

\begin{lemma}\label{lem:noMoore}
Let $P$ and $Q$ be two marked anti-polynomials of equal degree $d\geq 2$, where $P$ is critically fixed and $Q$ is postcritically finite and hyperbolic.
\begin{enumerate}
\item \label{lem:Moore1} If a principal ray equivalence class of $\sim_{ray}$ contains more than one cut-point of $\mathcal{K}(P)$, it must contain a 2-cycle or a 4-cycle.

\item \label{lem:Moore2} The equivalence $\sim_{ray}$ is Moore-obstructed if and only if some principal ray equivalence class contains a 2-cycle or a 4-cycle. 
\end{enumerate}
\end{lemma}

\begin{proof}
We first note that since the second iterate of a postcritically finite anti-polynomial is a postcritically finite holomorphic polynomial with the same Julia set, it follows from \cite[\S 9, \S 19]{Milnor06} that the Julia sets of $P, Q$ are connected and locally connected.

Let $x\in\mathcal{K}(P)$ be a repelling fixed point whose ray equivalence class $[x]$ contains more than one cut-point of $\mathcal{K}(P)$. Each point of $\mathcal{J}(P)$ is the landing point of one or two rays. Since $[x]$ is connected, there must be cut-points $w',z'\in [x]$ in $\mathcal{K}(P)$ so that $w'\sim_{ray} z'$ with $w'$ and $z'$ having graph distance two. Let $\gamma$ be the union of the two rays realizing the connection between $w'$ and $z'$. For some iterate $j$, we have from Lemma \ref{lem:ef} that $(P\mate Q)^j(w')$ and $(P\mate Q)^j(z')$ are fixed points which we denote $w$ and $z$ respectively. The graph distance between $w$ and $z$ is either 0 or 2. In the former case $w=z$ and since $Q$ is hyperbolic and hence a local homeomorphism on $\mathcal{J}(Q)$, we have that $(P\mate Q)^j(\gamma)$ is a 2-cycle. Since $x$ is a fixed point, its ray class is forward invariant, so $(P\mate Q)^j(\gamma)$ must be contained in $[x]$. This yields a 2-cycle in $[x]$.

Suppose now that $w\neq z$.
There then exist two extended external rays $\alpha_z$ and $\alpha_w$ that land at $z$ and $w$ and share a common endpoint $v\in \mathcal{K}(Q)$. 
Let $\gamma$ be the concatenation of these two rays. Each of $\alpha_z,\alpha_w$ must have dynamical period 2 since they land at fixed cut-points of $P$. 
If $v$ is a fixed point, then the graph $\alpha_w\cup (P\mate Q)(\alpha_w)$ is a cycle of length $2$ for $\sim_{ray}$. 
If $v$ has period $2$, then $\gamma\cup (P\mate Q)(\gamma)$ is a cycle of length $4$ for $\sim_{ray}$. As argued before, either of these cycles must be in $[x]$.
This concludes the proof of statement \ref{lem:Moore1}.

If a principal ray equivalence contains a cycle, it is immediate that ${\sim_{ray}}$ is Moore-obstructed. Now assume that no principal ray equivalence class contains a $2$-cycle or a $4$-cycle. 
This means that the diameter of a principal ray equivalence class is bounded by $4$ as there can be at most one cut-point of $\mathcal{K}(P)$ in this class.
Any principal ray equivalence class must evidently be a tree.

Now let $z\in\mathcal{J}(P)$. If there is no cut-point in the equivalence class $[z]$, then $[z]$ is a tree and has diameter at most $2$.
Otherwise, by Lemma \ref{lem:ef}, it is eventually mapped to a principal ray equivalence class.
Since both $P$ and $Q$ are hyperbolic, the map $P\mate Q$ is a local homeomorphism on $\mathcal{J}(P)$ and $\mathcal{J}(Q)$.
Since all the principal ray equivalence classes are trees, in particular, simply connected, any component of its preimage under $P\mate Q$ is homeomorphic to it.
Thus by induction, we conclude that $[z]$ is a tree and has diameter at most $4$.
Now by Proposition \ref{prop:epstein}, the mating is Moore-unobstructed.
\end{proof}

We now prove that for a large class of pairs of anti-polynomials, the absence of Moore obstruction is equivalent to the existence of geometric mating.

\begin{prop}\label{prop:mig}
Let $P$ and $Q$ be marked anti-polynomials of equal degree $d\geq 2$, where $P$ is critically fixed and $Q$ is postcritically finite and hyperbolic. There is an anti-rational map that is the geometric mating of $P$ and $Q$ if and only if $\sim_{ray}$ is not Moore-obstructed.
\end{prop}
\begin{proof}
The only if part of the statement is immediate.
To prove the if part, we will first show that the absence of Moore obstruction can be promoted to the absence of Thurston obstruction.
Note that the only case where $(P\mate Q)^{2}$ has non-hyperbolic orbifold is when $P$ and $Q$ are both power maps, in which case the conclusion of the theorem is immediate. 
Thus we assume for the remainder of the proof that $(P\mate Q)^{2}$ has hyperbolic orbifold.

Suppose $(P\mate Q)^{2}$ has a Thurston obstruction $\Gamma$. We may assume that $\Gamma$ is irreducible. Since $Q$ has no Thurston obstruction, there is some edge $\lambda$ in the Hubbard tree $H_P$ so that $\Gamma\cdot\lambda\neq 0$. 
Note that $\lambda$ is an irreducible arc system consisting of one arc, so it follows from \cite[Theorem 3.9]{ST00} that $\Gamma$ must also be a Levy cycle consisting of one essential curve $\gamma$.

Since $P\mate Q$ is hyperbolic, the Levy cycle is not degenerate (as defined in \cite[$\S$1]{ST00}). Thus there is a periodic ray class that contains a cycle \cite[Theorem 1.4(2)]{ST00}. 
By the hypothesis that there is no Moore obstruction, no such periodic ray class exists. This is a contradiction, and so no Thurston obstruction exists for $(P\mate Q)^{2}$. By Proposition \ref{prop:ttar}, it follows that $P\mate Q$ is also unobstructed and Thurston equivalent to a rational map.

Having shown that $P\mate Q$ has no Thurston obstruction, the Rees-Shishikura theorem (which extends directly to orientation reversing maps) implies that the geometric mating of $P$ and $Q$ exists \cite[Theorem 1.7]{Shishikura00}.
\end{proof}

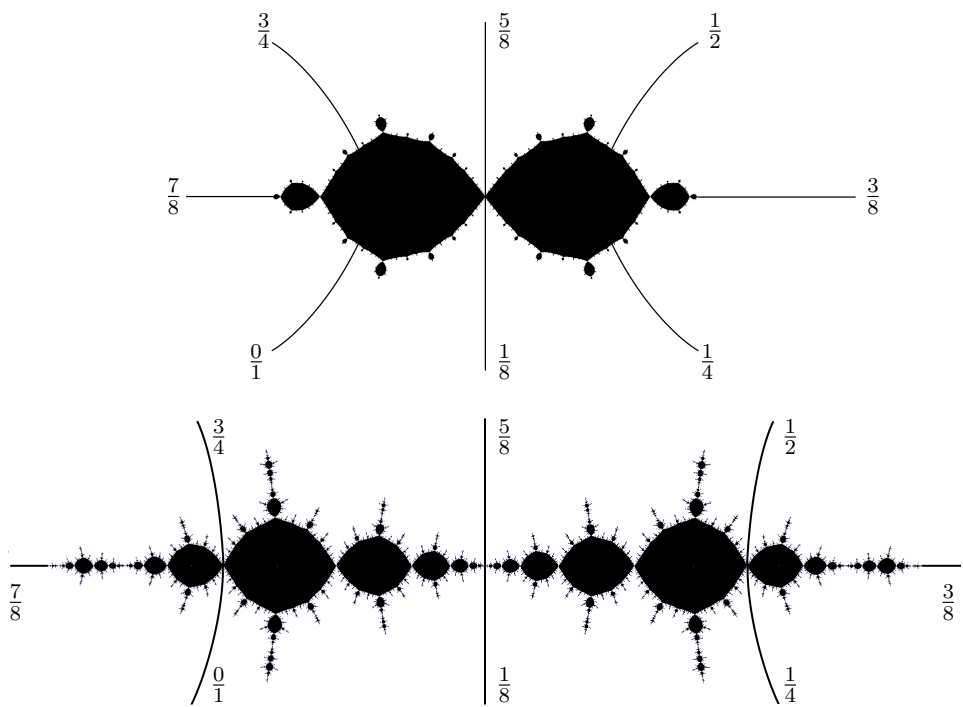
\begin{figure}[ht!]
  \centering
  \begin{tikzpicture}
   \node[anchor=south west,inner sep=0] at (0,0) {\includegraphics[width=1\linewidth]{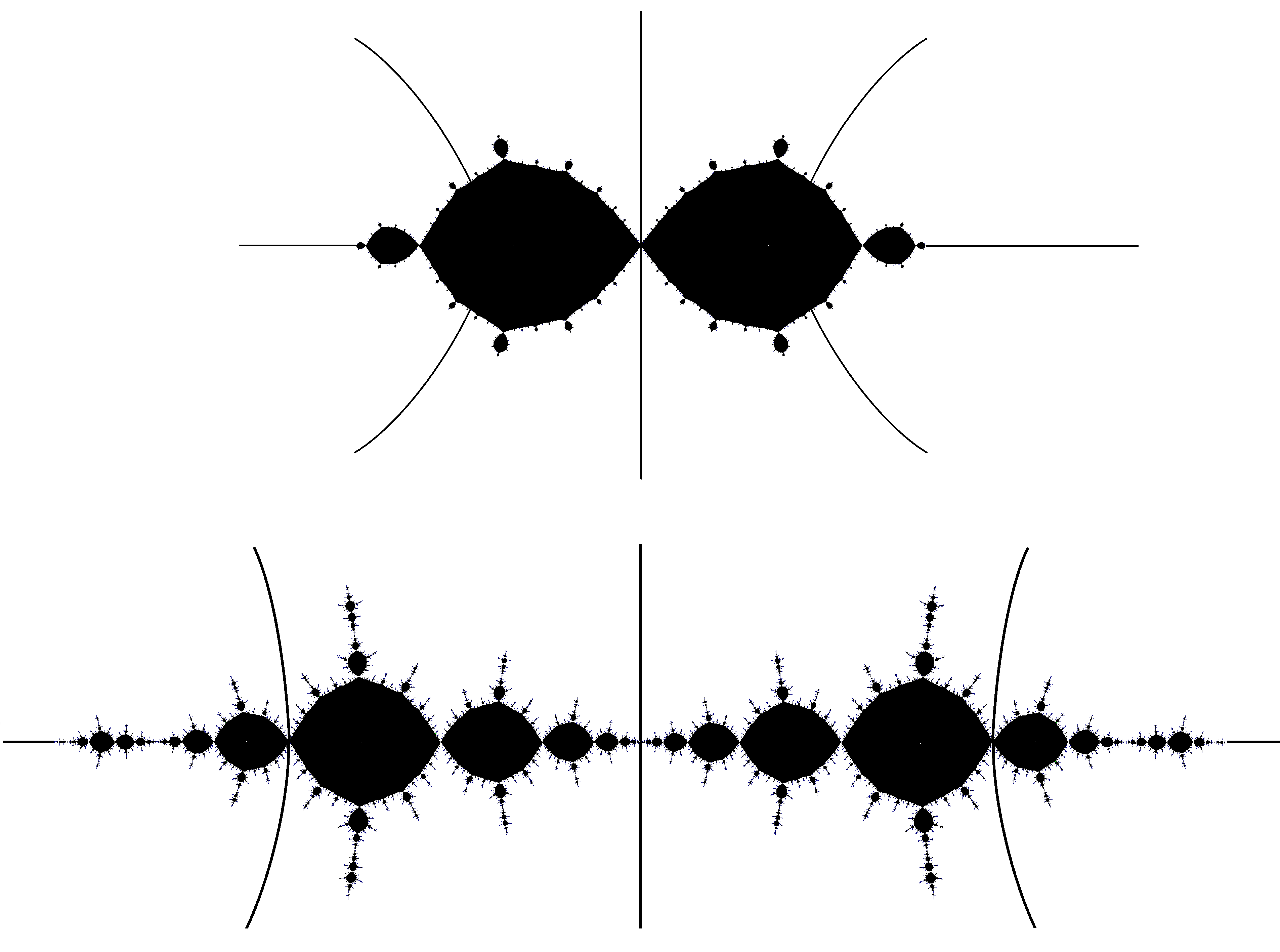}};
    \node at (2.2,7) {$\frac{7}{8}$};
        \node at (11.5,7) {$\frac{3}{8}$};

    \node at (3.3,4.8) {$\frac{0}{1}$};
    \node at (6.6,4.8) {$\frac{1}{8}$};
    \node at (9.3,4.8) {$\frac{1}{4}$};

    \node at (3.4,9.2) {$\frac{3}{4}$};
    \node at (6.6,9.2) {$\frac{5}{8}$};
    \node at (9.4,9.2) {$\frac{1}{2}$};
    \node at (.1,1.6) {$\frac{7}{8}$};
        \node at (12.5,1.6) {$\frac{3}{8}$};

    \node at (2.8,0.5) {$\frac{0}{1}$};
    \node at (6.6,0.5) {$\frac{1}{8}$};
    \node at (10.4,0.5) {$\frac{1}{4}$};

    \node at (2.8,3.8) {$\frac{3}{4}$};
    \node at (6.6,3.8) {$\frac{5}{8}$};
    \node at (10.4,3.8) {$\frac{1}{2}$};
  \end{tikzpicture}
   \caption{Two cubic anti-polynomials together with all rays that have period two or smaller. The first map is a critically fixed antipolynomial so that the 1/8 and 5/8 rays co-land. To second map is produced by tuning the first map with basilicas so that the pairs $(0/1,3/4)$ and $(1/2,1/4)$ co-land. The unique principal ray equivalence class has no cycle. Their mating is depicted below using the software \cite{Bar}.
   }
   \label{fig:MatingPolys}
  \end{figure}
  \begin{figure}[h!]
  \centering
  \includegraphics[width=1\linewidth]{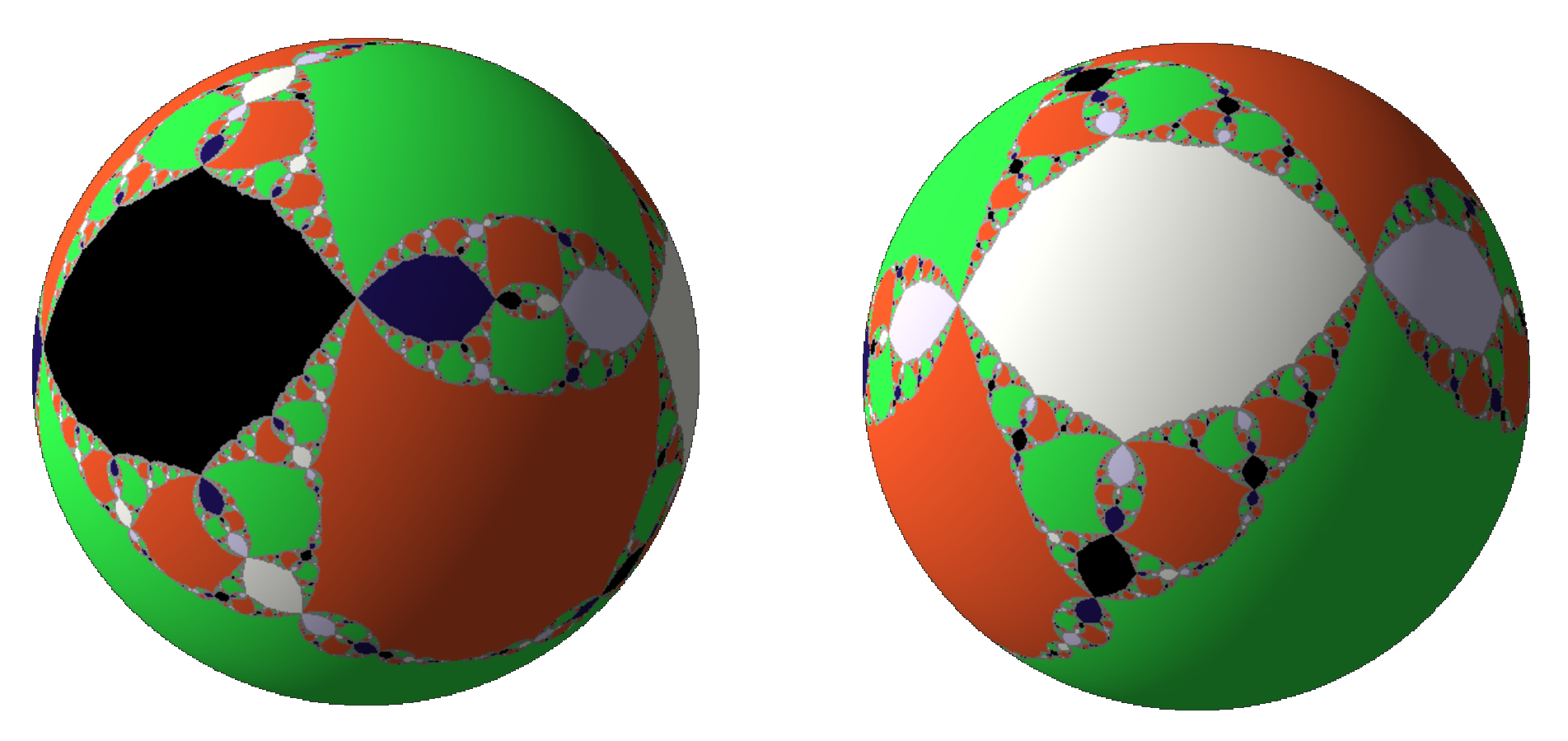}  

\label{fig:MatingSpheres}
\end{figure}

Recall that a pair of laminations on $\mathbb{S}^1$ for marked anti-polynomials is said to be non-parallel if they share no common leaf under the natural orientation reversing identification of the two copies of $\mathbb{S}^1$.
This is equivalent to saying that the ray equivalence classes contain no 2-cycle.
We immediately have the following corollary from Lemma \ref{lem:noMoore} and Proposition \ref{prop:mig}.
\begin{cor}\label{cor:mig}
A marked critically fixed anti-polynomial $P$ and a marked post-critically finite, hyperbolic anti-polynomial $Q$ are geometrically mateable if and only if the principal ray equivalence classes contain no $2$-cycles or $4$-cycles.

Two marked critically fixed anti-polynomials $P$ and $Q$ are geometrically mateable if and only if their laminations are non-parallel.
\end{cor}
\begin{proof}
The first statement is immediate.

To see the second one, we note that according to Lemma~\ref{lem:ef}, the lamination of a marked critically fixed anti-polynomial is generated by the period $2$ cycles of rays landing at the repelling fixed points (which are cut-points).
In light of this, it is easy to see that for two marked critically fixed anti-polynomials, no principal ray equivalence class contains a $4$-cycle.
The second statement now follows from this observation and the first statement.
\end{proof}

\vspace{2mm}

\noindent{\textbf{Conflicts of interest:} none.}
\vspace{2mm}

\noindent{\textbf{Financial Support:} The third author was supported by an endowment from Infosys Foundation and SERB research grant SRG/2020/000018.}

\end{document}